\documentclass[12pt]{amsart}
\usepackage{amssymb}
\usepackage{amsmath}
\usepackage{bm}
\usepackage{bbm}
\usepackage{caption}
\usepackage{subcaption}
\usepackage{float}
\usepackage{makecell}
\usepackage{multirow}
\usepackage{stmaryrd}
\usepackage{mathrsfs}
\usepackage[cmtip,all]{xy}
\usepackage[colorlinks=true,linkcolor=blue,citecolor=blue,urlcolor=blue]{hyperref}
\ifpdf
\usepackage[pdftex,final]{graphicx}
\usepackage[pdftex,lmargin=1in,rmargin=1in,tmargin=1in,bmargin=1in]{geometry}

\usepackage{tikz}
\usetikzlibrary{matrix,arrows}

\tikzset{cdlabel/.style={above,sloped,
    execute at begin node=$\scriptstyle,execute at end node=$}}
\tikzset{algarrow/.style={->, thick}}
\tikzset{blgarrow/.style={->, thick}}
\tikzset{clgarrow/.style={->, thick}}
\tikzset{tensoralgarrow/.style={double, double equal sign distance, -implies}}
\tikzset{tensorblgarrow/.style={double, double equal sign distance, -implies}}
\tikzset{tensorclgarrow/.style={double, double equal sign distance, -implies}}
\tikzset{tensorelgarrow/.style={double, double equal sign distance, -implies}}
\tikzset{modarrow/.style={->, dashed}}
\tikzset{Amodar/.style={->, dashed}}
\tikzset{Dmodar/.style={->, dashed}}
\tikzset{DAmodar/.style={->, dashed}}

\usepackage{tikz-cd}

\newread\testin

\graphicspath{{draws/}{mpdraws/}}
\makeatletter
\def\input@path{{}{draws/}}
\makeatother

\makeatletter
\newcommand\mi@kern[1]{%
  \settowidth\@tempdima{$\mi@obj^{#1}$}
  \kern-\@tempdima
  #1
  \settowidth\@tempdima{$\mi@obj$}
  \kern\@tempdima
}

\newtoks\mi@toksp
\newtoks\mi@toksb
\DeclareRobustCommand{\manyindices}[5]{
  \def\mi@obj{#5}
  \mi@toksp\expandafter{\mi@kern{#2}}
  \mi@toksb\expandafter{\mi@kern{#1}}
  \@mathmeasure4\textstyle{#5_{#1}^{#2}}
  \@mathmeasure6\textstyle{#5_{#3}^{#4}}
  \dimen0-\wd6 \advance\dimen0\wd4
  \@mathmeasure8\textstyle{\hphantom{{}_{#1}^{#2}}#5^{\the\mi@toksp#4}_{\the\mi@toksb#3}}
  \hbox to \dimen0{}{\kern-\dimen0\box8}
}
\makeatother 

\usepackage{ifpdf}
\ifpdf
  
\else
  
\fi

\newtheorem{theorem}{Theorem}[section]
\newtheorem{corollary}[theorem]{Corollary}
\newtheorem{lemma}[theorem]{Lemma}
\newtheorem{proposition}[theorem]{Proposition}

\newtheorem{question}[theorem]{Question}

\theoremstyle{definition}
\newtheorem{definition}[theorem]{Definition}

\newtheorem{example}[theorem]{Example}

\theoremstyle{remark}
\newtheorem{remark}[theorem]{Remark}

\numberwithin{equation}{section}

\newcommand{\twopartdef}[4]
{
	\left\{
		\begin{array}{ll}
			#1 & \mbox{if } #2 \\
			#3 & \mbox{if } #4
		\end{array}
	\right.
}

\newcommand{\bb}[1]{\mathbb{#1}}

\newcommand{\scr}[1]{\mathscr{#1}}
\renewcommand{\frak}[1]{\mathfrak{#1}}
\newcommand{\ZZ}{\bb{Z}}
\newcommand{\QQ}{\bb{Q}}
\newcommand{\RR}{\bb{R}}
\newcommand{\CC}{\bb{C}}

\DeclareMathOperator{\im}{im}

\newcommand{\spinc}{\text{Spin}^{c}}
\newcommand{\wt}[1]{\widetilde{#1}}
\newcommand{\wh}[1]{\widehat{#1}}
\newcommand{\ol}[1]{\overline{#1}}

\newcommand{\del}{\partial}

\newcommand{\<}{\langle}
\renewcommand{\>}{\rangle}

\newcommand{\mbf}[1]{\mathbf{#1}}

\makeatletter
\def\widebreve{\mathpalette\wide@breve}
\def\wide@breve#1#2{\sbox\z@{$#1#2$}%
     \mathop{\vbox{\m@th\ialign{##\crcr
\kern0.08em\brevefill#1{0.8\wd\z@}\crcr\noalign{\nointerlineskip}%
                    $\hss#1#2\hss$\crcr}}}\limits}
\def\brevefill#1#2{$\m@th\sbox\tw@{$#1($}%
  \hss\resizebox{#2}{\wd\tw@}{\rotatebox[origin=c]{90}{\upshape(}}\hss$}
\makeatletter

\newcommand{\C}{\mathcal{C}}

\newcommand{\D}{\mathcal{D}}

\newcommand{\DDD}{\frak{D}}

\newcommand{\HH}{\mathbb{H}}
\newcommand{\I}{\mathcal{I}}

\newcommand{\III}{\mathfrak{I}}
\newcommand{\K}{\mathcal{K}}

\renewcommand{\L}{\mathcal{L}}
\newcommand{\M}{\mathcal{M}}

\newcommand{\Q}{\mathcal{Q}}

\renewcommand{\SS}{\mathbb{S}}
\newcommand{\SSS}{\mathfrak{S}}

\DeclareMathOperator{\irr}{irr}

\DeclareMathOperator{\ev}{ev}
\DeclareMathOperator{\odd}{odd}

\DeclareMathOperator{\Pin}{Pin}
\DeclareMathOperator{\SWF}{SWF}

\DeclareMathOperator{\rot}{rot}

\usepackage{slashed}
\usepackage{mathtools}
\newcommand{\fsl}[1]{\ensuremath{\mathrlap{\!\not{\phantom{#1}}}#1}}
\newcommand{\dirac}{\fsl{\partial}}

\renewcommand{\deg}{\text{deg}}

\newcommand{\mbfk}{\mbf{k}}

\newcommand{\mbfs}{\mbf{s}}
\newcommand{\mbft}{\mbf{t}}

\DeclareMathOperator{\st}{st}

\DeclareMathOperator{\sign}{sign}
\DeclareMathOperator{\Spin}{Spin}

\copyrightinfo{2009}{American Mathematical Society}

\title{Non-Smoothable $\ZZ/p$-actions on Nuclei}

\author{Imogen Montague}
\address{Department of Mathematics, Boston College, Chestnut Hill, MA 02467}
\email{imogen.montague@austin.utexas.edu}

\begin{document}

\begin{abstract}
In this article we construct examples of non-smoothable $\ZZ/p$-actions on indefinite spin 4-manifolds with boundary for all primes $p\geq 5$. For example, we show that for each prime $p\geq 5$ and each $n\geq 1$ there exists a locally linear $\ZZ/p$-action on the Gompf nucleus $N(2pn)$ which is not smoothable with respect to any smooth structure on $N(2pn)$. Furthermore we investigate the behavior of these actions under two different types of equivariant stabilizations with $S^{2}\times S^{2}$, namely \emph{free} and \emph{homologically trivial} stabilizations --- in particular we show that our non-smoothable $\ZZ/p$-action on $N(2pn)$ remains non-smoothable after $2n-2$ free stabilizations, and after arbitrarily many homologically trivial stabilizations. We also show that free stabilizations satisfy a Wall stabilization principle in the sense that any non-smoothable $\ZZ/p$-action becomes smoothable after some finite number free stabilizations (under certain assumptions), whereas our aforementioned result implies that homologically trivial stabilizations do not satisfy this property.  The proofs of these results use equivariant $\kappa$-invariants defined by the author in \cite{Mon:SW}, calculations of equivariant $\eta$-invariants for the odd signature and Dirac operators on Seifert-fibered spaces, as well as an analysis of the geometric $S^{1}$-action on the Seiberg-Witten moduli spaces of Seifert-fibered spaces induced by rotation in the fibers, which may be of independent interest.
\end{abstract}

\maketitle

\setcounter{tocdepth}{3}


\section{Introduction}
\label{sec:intro}

By a result of Edmonds \cite{Edmonds87} every closed, simply-connected topological 4-manifold $X$ admits a locally linear $\ZZ_{p}$-action for any odd prime $p$, $\ZZ_{p}:=\ZZ/p\ZZ$. Moreover these actions can be taken to be homologically trivial, and for $p\geq 5$ they can be taken to be pseudofree (i.e., having only isolated fixed points).

On the other hand, gauge-theoretic techniques have been employed to show that not all of these actions are smooth. In particular Kiyono \cite{Kiyono11} used the orbifold version of Furuta's $10/8$-ths inequality to show that: if $X$ is any closed simply connected spin 4-manifold not homeomorphic to $S^{4}$ or $S^{2}\times S^{2}$, then for all sufficiently large primes $p$ there exists a locally linear $\ZZ_{p}$-action on $X$ which is not topologically conjugate to a smooth action with respect to \emph{any} smooth structure on $X$ --- we refer to such an action as a \emph{non-smoothable} $\ZZ_{p}$-action.

A natural question to ask is whether analogous results hold in the case of 4-manifolds with boundary. In particular one can ask the following extension question:

\begin{question}
Given a triple of the form $(Y,\sigma,X)$, where $Y$ is a 3-manifold, $\sigma:Y\to Y$ is a self-diffeomorphism of prime order $p\geq 2$, and $X$ is a smooth 4-manifold with $\del X=Y$, when does $\sigma$ extend to a locally linear $\ZZ_{p}$-action on $X$? If there is such an extension, is there a smooth structure on $X$ for which the action is topologically conjugate to a smooth action?
\end{question}

The above question is more difficult than the closed case --- it is unknown whether an analogue of Edmonds' result holds in the case of 4-manifolds with non-empty boundary.

An important class of 3-manifolds with $\ZZ_{p}$-actions are given by Seifert-fibered spaces. Indeed for $Y$ a Seifert-fibered space, there is a canonical $S^{1}$-action $\rho:S^{1}\times Y\to Y$ given by rotation in the $S^{1}$-fibers. For each $p\geq 2$, let $\rho_{p}:Y\to Y$ denote the generator of the $\ZZ_{p}$-action given by $\rho_{p}(y):=\rho(e^{2\pi i/p},y)$, which we call the \emph{standard} $\ZZ_{p}$-action on $Y$. If $Y$ is in particular a Brieskorn homology sphere, by (\cite{MS86}, \cite{LS92}, \cite{Perelman02}, \cite{BLP05}, \cite{DL09}) any effective smooth $\ZZ_{p}$-action on $Y$ is conjugate to $\rho_{p}$ for primes $p\geq 3$.

Baraglia--Hekmati \cite{BarHek2} proved that the standard $\ZZ_{p}$-action on a Seifert-fibered homology sphere cannot extend over any 4-manifold $X$ with $b_{1}(X)=b_{2}(X)=0$, subject to some additional conditions (see also \cite{KL93}, \cite{AH16}, \cite{AH21}). On the other hand Kwasik--Lawson \cite{KL93} and Anvari--Hambleton \cite{AH16} showed that some of these actions extend locally linearly, in particular providing examples of non-smoothable $\ZZ_{p}$-actions on contractible 4-manifolds.

In the case where $b_{2}(X)>0$, there is much less known. For $p=2$, Konno-Miyazawa-Taniguchi \cite{KMT} showed that there exist non-smoothable locally linear involutions on the connected sum $M(2,3,6n\pm 1)\#^{N}S^{2}\times S^{2}$ for all sufficiently large $N$ which extend the involution $\rho_{2}$ on $\Sigma(2,3,6n\pm 1)$. Here $M(2,3,6n\pm 1)$ denotes the Milnor fiber with boundary $\Sigma(2,3,6n\pm 1)$, and we exclude the exceptional case $M(2,3,5)$. A similar statement holds for boundary connected sums of the above manifolds. This result has the striking property that the above non-smoothablity property persists after arbitrarily many equivariant stabilizations with $S^{2}\times S^{2}$ of \emph{homologically trivial} type.

In this article we prove a complementary result to that of \cite{KMT} in the case of higher order actions, and construct non-smoothable $\ZZ_{p}$-actions on 4-manifolds $X$ with non-empty boundary and $b_{2}(X)>0$ for all primes $p\geq 5$. 

For $n\geq 1$ let $N(2n)$ be the Gompf nucleus with intersection form $H:=\big(\begin{smallmatrix} 0 & 1 \\ 1 & 0\end{smallmatrix}\big)$ and boundary $-\Sigma(2,3,12n-1)$, given by the Kirby diagram given in Figure \ref{fig:N(2n)}. We denote by $P(2n)$ the spin 4-manifold with intersection form $-E_{8}\oplus H$ and boundary $-\Sigma(2,3,12n-5)$, given by the plumbing given in Figure \ref{fig:P(2n)}. Here $-E_{8}$ denotes the unique even unimodular negative-definite form with signature $-8$. Note that $P(2)$ coincides with the complement of a Milnor fiber $M(2,3,7)$ in the $K3$ surface.

\begin{figure}
\centering
\begin{subfigure}{.5\textwidth}
  \centering
  \includegraphics[height=2cm]{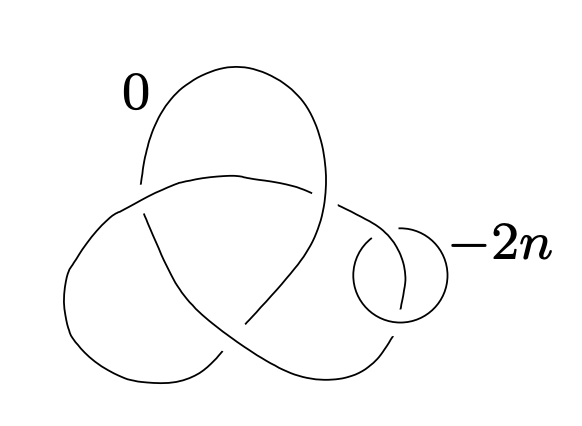}
  \caption{A Kirby diagram for $N(2n)$.}
  \label{fig:N(2n)}
\end{subfigure}%
\begin{subfigure}{.5\textwidth}
  \centering
  \includegraphics[height=2cm]{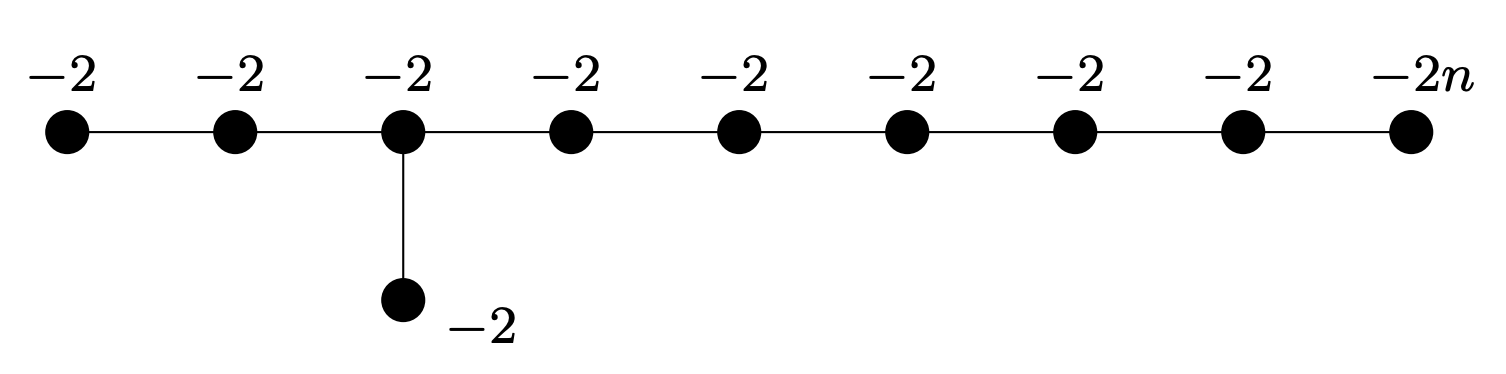}
  \caption{A plumbing diagram for $P(2n)$.}
  \label{fig:P(2n)}
\end{subfigure}
\caption{Two spin 4-manifolds with boundaries $-\Sigma(2,3,12n-1)$ and $-\Sigma(2,3,12n-5)$, respectively.}
\label{fig:diagrams}
\end{figure}

Our first result concerns the non-existence of smooth extensions:

\begin{theorem}
\label{theorem:intro_no_smooth_extension}
Let $(p,X)$ be any of the following pairs:
\begin{enumerate}
    \item $p\geq 3$ is prime, and $X$ is any smooth 4-manifold homeomorphic to the Gompf nucleus $N(2n)$ or its stabilization $N(2n)\# S^{2}\times S^{2}$, $n\geq 1$, provided $(n,p)\neq (1,5)$.
    \item $p\geq 3$ is prime, and $X$ is any smooth 4-manifold homeomorphic to $P(2n)$ or its stabilization $P(2n)\#S^{2}\times S^{2}$, $n\geq 1$.
    \item $p\geq 3$ is prime, $p\neq 5$, and $X$ is any smooth 4-manifold homeomorphic to the Milnor fiber $M(2,3,11)$.
    \item $p\geq 3$ is prime, and $X$ is any smooth 4-manifold homeomorphic to the Milnor fiber $M(2,3,7)$.
\end{enumerate}
Then no effective smooth $\ZZ_{p}$-action on $\del X$ can extend to a smooth homologically trivial $\ZZ_{p}$-action over $X$.
\end{theorem}

We make a few remarks regarding the above theorem:
\begin{itemize}
    \item The homological triviality restriction in the above theorem is vacuous for sufficiently large $p$, since the classification of $\ZZ_{p}$-representations over $\QQ$ implies that a $\ZZ_{p}$-action on a compact oriented $4$-manifold $X$ must act trivially on $H_{2}(X)/\text{tors}$ if $p\geq b_{2}(X)+2$. For example, in case (1) of the above theorem any $\ZZ_{p}$-action on $N(2n)$ must be homologically trivial if $p\geq 5$, and similarly for $N(2n)\#S^{2}\times S^{2}$ provided $p\geq 7$.
    \item The exceptional case $(n,p)=(1,5)$ from case (1) appears on the level of Heegaard Floer/monopole Floer homology as well --- in particular Baraglia--Hekmati \cite{BarHek2} showed that 
    \[\wh{HF}(\Sigma(2,3,12n-1))\cong\wh{HF}(Q(p;2,3,12n-1),\frak{s}_{0})\iff (n,p)=(1,5),\]
    where $Q(p;2,3,12n-1)$ denotes the quotient of $\Sigma(2,3,12n-1)$ by $\rho_{p}$, equipped with its unique self-conjugate $\spinc$-structure $\frak{s}_{0}$.
    \item It would be interesting to see whether the standard $\ZZ_{5}$-action on $\del N(2)$ (corresponding to the exceptional case from case (1)) does in fact extend smoothly over $N(2)$. By a $G$-signature calculation, any such extension which is pseudofree must have 3 fixed points with corresponding fixed point data $\{(1,4),(1,4),(2,3)\}$. 
    \item In cases (3)-(4), one can use the inequality from \cite{KMT} to show that the $\ZZ_{2}$-action $\rho_{2}:Y\to Y$ cannot extend to a homologically trivial smooth involution over $X$. However, $\rho_{2}$ extends as a non-homologically trivial smooth involution over $X$ since $M(2,3,11)$ and $M(2,3,7)$ can be realized as branched double-covers of Seifert surfaces pushed into $B^{4}$ for the torus knots $T(3,11)$ and $T(3,7)$, respectively. In a similar vein, the $\ZZ_{3}$- and $\ZZ_{p}$-actions extend to smooth non-homologically trivial actions over $M(2,3,7)$ and $M(2,3,11)$ for $p=7$ and $11$, respectively.
    \item We can also obstruct the existence of smooth extensions over certain cobordisms between Brieskorn spheres. For example if $X$ is homeomorphic to the cobordism from $\Sigma(2,3,5)$ to $\Sigma(2,3,7)$ obtained as the complement of the embedding $M(2,3,5)\hookrightarrow M(2,3,7)$, then no effective smooth $\ZZ_{p}$-action over $\del X$ can extend to a smooth homologically trivial $\ZZ_{p}$-action over $X$ for any prime $p\geq 3$. We defer the full statement to Theorem \ref{theorem:no_smooth_extension_cobordisms}.
\end{itemize}
The proof of Theorem \ref{theorem:intro_no_smooth_extension} makes use of the invariants constructed in \cite{Mon:SW}, as well as an analysis of the geometric $S^{1}$-action on the Seiberg-Witten moduli spaces of Seifert-fibered spaces induced by rotation in the fibers (see Section \ref{sec:s_1_action} for more details).

On the other hand by using techniques from \cite{Edmonds87} and \cite{KL93}, we can show that in some cases the $\ZZ_{p}$-actions from Theorem \ref{theorem:intro_no_smooth_extension} admit homologically trivial locally linear extensions:

\begin{theorem}
\label{theorem:intro_locally_linear_extension}
Let $p\geq 5$ be prime. Then for every $n\geq 1$, the standard $\ZZ_{p}$-action on $\del N(2pn)$ extends to a locally linear homologically trivial pseudofree $\ZZ_{p}$-action $N(2pn)$, and similarly for $P(2pn-p+1)$.
\end{theorem}

Combined with Theorem \ref{theorem:intro_no_smooth_extension} we obtain the following corollary:

\begin{corollary}
\label{cor:intro_nonsmoothable_actions}
For each prime $p\geq 5$ and each $n\geq 1$, there exist non-smoothable homologically trivial pseudofree $\ZZ_{p}$-actions
\begin{align*}
    &\tau_{p,n}:N(2pn)\to N(2pn) & &\tau'_{p,n}:P(2pn-p+1)\to P(2pn-p+1)
\end{align*}
extending the standard $\ZZ_{p}$-actions on their boundaries.
\end{corollary}

We also study the behavior of these non-smoothable actions under equivariant stabilization. Let $\tau$ be a locally linear effective $\ZZ_{p}$-action on a topological 4-manifold $X$, with $\tau$-fixed point set $X^{\tau}\subset X$. In this article we consider two different types of stabilizations of the pair $(X,\tau)$ by (connected sums of) $S^{2}\times S^{2}$:
\begin{enumerate}
    \item Suppose $X^{\tau}$ contains an isolated fixed point $x\in X$ such that the action of $\tau$ on a neighborhood of $x$ can be identified with the action of
    \[\begin{pmatrix}
    e^{2\pi ia/p} & 0 \\
    0 & e^{2\pi ib/p}
    \end{pmatrix}\]
    for some $(a,b)\in(\ZZ_{p})^{\times}$ --- we refer to such a point as a \emph{fixed point of type $(a,b)$}. There is a natural $\ZZ_{p}$-action $\tau_{a,b}:S^{2}\times S^{2}\to S^{2}\times S^{2}$ given by rotating the two sphere factors by $\frac{2\pi a}{p}$ and $\frac{2\pi b}{p}$, respectively, which has four fixed points of types $(a,b), (a,b), (-a,b), (-a,b)$; let $y\in S^{2}\times S^{2}$ be one of the fixed points of type $(-a,b)$. We can then perform a topological equivariant connected sum $(X,\tau)\#(S^{2}\times S^{2},\tau_{a,b})$ along equivariant neighborhoods of $x\in X$ and $y\in S^{2}\times S^{2}$, resulting in a well-defined locally linear action on $X\# S^{2}\times S^{2}$. We call such a stabilization a \emph{homologically trivial} stabilization of $(X,\tau)$, as the induced action acts trivially on $H_{2}(S^{2}\times S^{2})<H_{2}(X\#S^{2}\times S^{2})$. A similar such stabilization can be done in the neighborhood of a fixed point contained in a surface component of $X^{\tau}$ if we consider an action on $S^{2}\times S^{2}$ which rotates only one of the sphere factors.
    \item Let $e:\text{Int}(B^{4})\times\ZZ_{p}\hookrightarrow X$ be an equivariant embedding. If one glues a copy of $(S^{2}\times S^{2}\setminus \text{Int}(B^{4}))\times\ZZ_{p}$ along its boundary to $\del(X\setminus\text{im}(e))$, then $\tau$ extends naturally to a locally linear pseudofree action on $X\#^{p}S^{2}\times S^{2}$. We call such a stabilization a \emph{free} stabilization of $(X,\tau)$, as the induced action freely permutes the $p$ copies of $S^{2}\times S^{2}\setminus B^{4}$ in $X\#^{p}S^{2}\times S^{2}$.
\end{enumerate}

A natural question to ask is whether non-smoothable $\ZZ_{p}$-actions must be conjugate to a smooth action after a finite number of equivariant stabilizations of either type, i.e., whether there exists a \emph{Wall stabilization} phenomenon for non-smoothable group actions. We outline some situations where Wall stabilization holds for non-smoothable phenomena (as opposed to \emph{exotic} phenomena):
\begin{itemize}
    \item Any compact orientable topological 4-manifold $X$ with vanishing Kirby-Siebenmann invariant admits a smooth structure after sufficiently many stabilizations with $S^{2}\times S^{2}$ (\cite{FQ90}, \cite{FNOP19}).
    \item Any self-homeomorphism of a smooth simply-connected 4-manifold $X$ is isotopic to a diffeomorphism after a single stabilization with $S^{2}\times S^{2}$ is $X$ is closed (\cite{RS:wall}, \cite{Wall:diff}, \cite{Perron86}, \cite{Quinn86}), and after sufficiently many stabilizations with $S^{2}\times S^{2}$ if $X$ has non-empty boundary (\cite{OP22}, \cite{Saeki06}).
    \item Any properly embedded locally flat embedding of a surface into a smooth orientable 4-manifold $X$ with or without boundary is isotopic to a smooth embedding after sufficiently many external stabilizations with $S^{2}\times S^{2}$ \cite{CK23}.
    \item To the author's knowledge it is unknown whether there exists a Wall stabilization phenomenon for locally flat embeddings of surfaces into 4-manifolds with respect to \emph{internal stabilizations}, i.e., under taking relative connected sums with pairs of the form $(S^{4},T^{2})$.
\end{itemize}

Using an equivariant smoothing argument by Kwasik-Schultz \cite{KS18}, we show that non-smoothable $\ZZ_{p}$-actions satisfy a Wall stabilization principle with respect to \emph{free} stabilizations:

\begin{theorem}
\label{theorem:intro_wall}
Let $X$ be a compact oriented topological 4-manifold with vanishing Kirby-Siebenmann invariant, and let $\tau:X\to X$ be a locally linear effective $\ZZ_{p}$-action whose fixed-point set is contained in the interior of $X$. Furthermore, suppose that the compact manifold obtained from the quotient orbifold $X/\tau$ by removing an open tubular neighborhood of the orbifold singular set also has vanishing Kirby-Siebenmann invariant. Then there exists $N\geq 0$ such that if $\tau_{N}:X\#^{pN}S^{2}\times S^{2}\to X\#^{pN}S^{2}\times S^{2}$ denotes the $N$-fold free stabilization of $\tau$, there exists a smooth structure on $X\#^{pN}S^{2}\times S^{2}$ for which $\tau_{N}$ is smooth.
\end{theorem}

Given a pair $(X,\tau)$ as in Theorem \ref{theorem:intro_wall}, let $\text{stab}_{\text{fr}}(X,\tau)\in\ZZ_{\geq 0}$ denote the number of free stabilizations required to make $\tau$ smooth with respect to some smooth structure on the stabilized manifold, which we call the \emph{free stabilization number} of $(X,\tau)$. A natural question to ask is whether there exists a universal bound on the free stabilization number for any such pair $(X,\tau)$, provided $X$ is smooth.

On the other hand, the non-smoothable involutions constructed in \cite{KMT} show there there does not exist a Wall stabilization phenomenon with respect to homologically trivial stabilizations of non-smoothable $\ZZ_{p}$-actions for $p=2$. One may ask whether this is the case for higher order actions.

The following theorem addresses both of these questions:

\begin{theorem}
\label{theorem:intro_stabilizations}
For each prime $p\geq 5$ and $n\geq 1$, let $\tau_{p,n}$, $\tau'_{p,n}$ denote the non-smoothable $\ZZ_{p}$-actions from Corollary \ref{cor:intro_nonsmoothable_actions}, and let $\tau_{p,n,M,N}$ and $\tau'_{p,n,M,N}$ denote the locally linear $\ZZ_{p}$-actions obtained from $\tau_{p,n}$ and $\tau'_{p,n}$, respectively, by performing $M$ homologically trivial stabilizations and $N$ free stabilizations. Then $\tau_{p,n,M,N}$ and $\tau'_{p,n,M,N}$ are non-smoothable for any $M\geq 0$ and any $0\le N\le 2n-2$.
\end{theorem}

\begin{corollary}
Non-smoothable $\ZZ_{p}$-actions on compact smooth 4-manifolds do not satisfy a Wall stabilization phenomenon with respect to homologically trivial stabilizations for primes $p\geq 5$. Furthermore, for each prime $p\geq 5$ and each $N\geq 0$ there exists a non-smoothable pseudofree $\ZZ_{p}$-action $\tau$ on a compact smooth 4-manifold $X$ such that $(X,\tau)$ satisfies the hypotheses of Theorem \ref{theorem:intro_wall} and $\text{stab}_{\text{fr}}(X,\tau)>N$.
\end{corollary}

The proof of Theorem \ref{theorem:intro_stabilizations} relies on some computations of the equivariant $\eta$-invariants of the Dirac operator on the bounding 3-manifolds. For these, we will use formulas to be proven in an upcoming article \cite{Mon:eta} (see Section \ref{subsec:proposition} for more details).

Via a doubling argument, we obtain the following corollary which concerns non-smoothable equivariant embeddings of 3-manifolds into connected sums of $S^{2}\times S^{2}$:

\begin{corollary}
\label{cor:equivariant_embeddings}
Let $p\geq 5$ be prime. Then for each $k\geq 2$, there exists a smooth homologically trivial pseudofree $\ZZ_{p}$-action $\tau_{p,k}:\#^{k}S^{2}\times S^{2}\to \#^{k}S^{2}\times S^{2}$ such that:
\begin{enumerate}
    \item For each $n\geq 1$ there exists a locally flat equivariant embedding 
    \[i_{p,n,k}:(\Sigma(2,3,12pn-1),\rho_{p})\hookrightarrow(\#^{k}S^{2}\times S^{2},\tau_{p,k})\]
    which is isotopic but not equivariantly isotopic to a smooth embedding.
    \item Let $\tau_{p,k,M,N}$ denote the result of $M$ homologically trivial and $N$ free stabilizations of $\tau_{p,k}$, performed away from the image of $i_{p,n,k}$, for all $n\geq 1$. Then for each $n\geq 1$, the induced locally flat equivariant embedding
    \[i_{p,n,k,M,N}:(\Sigma(2,3,12pn-1),\rho_{p})\hookrightarrow(\#^{k+M+pN}S^{2}\times S^{2},\tau_{p,k,N})\]
    is isotopic but not equivariantly isotopic to a smooth embedding for any $M,N\geq 0$, provided at least $\max\{0,N-2n+2\}$ of the free stabilizations are performed on one of the connected components of $\#^{k}S^{2}\times S^{2}\setminus i_{p,n,k}(\Sigma(2,3,12pn-1))$ for each $n\geq 1$.
\end{enumerate}
A similar statement holds for the family $\Sigma(2,3,12pn-6p+1)$, provided $k\geq 8$.
\end{corollary}

In particular, the above corollary implies that there does not exist a Wall stabilization principle for locally flat equivariant embeddings of 3-manifolds with respect to homologically trivial equivariant stabilizations of the ambient manifold.

It is interesting to compare Corollary \ref{cor:equivariant_embeddings} to the results of \cite{KMT22}, where they showed the existence of exotic smooth embeddings of 3-manifolds into closed 4-manifolds which remain exotic after arbitrarily many stabilizations with $S^{2}\times S^{2}$, provided all of the stabilizations are done on ``one side''. The above corollary allows for arbitrarily many homologically trivial stabilizations and up to $2n-2$ free stabilizations on ``either'' side.

\subsection{Organization}
\label{subsec:intro_organization}

The paper is organized as follows. In Section \ref{sec:s_1_action} we analyze the geometric $S^{1}$-action on Seiberg--Witten moduli spaces of Seifert-fibered spaces. We provide an overview of equivariant $\kappa$-invariants in Section \ref{sec:kappa}, and subsequently prove Theorem \ref{theorem:intro_no_smooth_extension} and its counterpart for cobordisms Theorem \ref{theorem:no_smooth_extension_cobordisms}. In Section \ref{sec:nonsmoothable_actions} we then construct the non-smoothable $\ZZ_{p}$-actions from Corollary \ref{cor:intro_nonsmoothable_actions}. Finally in Section \ref{sec:stabilizations} we prove Theorems \ref{theorem:intro_wall}, \ref{theorem:intro_stabilizations}, and Corollary \ref{cor:equivariant_embeddings}.

\subsection{Acknowledgements}
\label{subsec:intro_acknowledgements}
I would like to thank Nima Anvari, David Auckly, Hokuto Konno, Jin Miyazawa, Audrey Rosevear, Nikolai Saveliev, Matthew Stoffregen and Masaki Taniguchi for interesting and helpful conversations. Special thanks to Daniel Ruberman for his help with the proof of Theorem \ref{theorem:intro_wall}.
\section{Lifting the \texorpdfstring{$S^{1}$}{S1}-Action}
\label{sec:s_1_action}

The proof of Theorem \ref{theorem:intro_no_smooth_extension} relies on understanding how the standard $S^{1}$-action on Seifert-fibered homology spheres lifts to the based moduli space of irreducible Seiberg-Witten solutions on them. In this section we give a characterization of this action in the form of Theorem \ref{theorem:S_1_action}.

Let $Y$ be a Seifert-fibered 3-manifold, presented as the unit circle bundle of an orbifold line bundle $L_{0}$ of degree $\ell\neq 0$ over some oriented orbifold surface $\Sigma$ of genus $g$ with $n$ singular points of isotropy orders $\alpha_{1},\dots,\alpha_{n}$. Furthermore, let $\frak{s}$ be a spin structure on $Y$ with associated spinor bundle $\SS\to Y$. In \cite{MOY}, Mrowka, Ozsv\'ath and Yu analyzed the moduli space of solutions to the unperturbed Seiberg-Witten equations on $(Y,\frak{s})$, with respect to a certain connection compatible with the Seifert metric, which we denote by $\nabla^{\infty}$. We refer to $\nabla^{\infty}$ as the \emph{adiabatic connection} following Nicolaescu \cite{Nic00} (in \cite{MOY} the authors referred to $\nabla^{\infty}$ as the \emph{reducible} connection, and denoted it by $\nabla^{\circ}$).

Let $\M(Y,\frak{s},\nabla^{\infty})$ denote this moduli space, and let $\M^{\irr}(Y,\frak{s},\nabla^{\infty})\subset \M(Y,\frak{s},\nabla^{\infty})$ denote the subset of irreducible solutions. In \cite{MOY} it was shown that $\M^{\irr}(Y,\frak{s},\nabla^{\infty})$ decomposes as a disjoint union
\[\M^{\irr}(Y,\frak{s},\nabla^{\infty})=\coprod_{\substack{E \\ \text{deg} E\le -\frac{\chi(\Sigma)}{2} \\ \pi^{*}(E)\oplus\pi^{*}(K_{\Sigma}^{-1}\otimes E)\approx\SS}}\Big(\C^{+}(E)\amalg\C^{-}(E)\Big)\]
of pairs of connected components $\C^{\pm}(E)$, which each pair in one-to-one correspondence with isomorphism classes of orbifold line bundles $E\to\Sigma$ such that:
\begin{enumerate}
    \item $\text{deg} E\le -\frac{1}{2}\chi(\Sigma)$, where 
        \[\chi(\Sigma)=2-2g+\sum_{i=1}^{n}\tfrac{1-\alpha_{i}}{\alpha_{i}}\]
        denotes the orbifold Euler characteristic of $\Sigma$.
    \item $\pi^{*}(E)\oplus\pi^{*}(K_{\Sigma}^{-1}\otimes E)\approx\SS$, where $\pi:Y\to\Sigma$ denotes the projection map, and $K^{-1}_{\Sigma}=(K_{\Sigma})^{-1}$ denotes the anti-canonical bundle of degree $\deg K_{\Sigma}^{-1}=-\deg K_{\Sigma}=\chi(\Sigma)$.
\end{enumerate}
We will often specify such a line bundle by its \emph{Seifert data} $E=(e;\epsilon_{1},\dots,\epsilon_{n})$ (see \cite{MOY}, Section 2).

The moduli space of irreducible solutions modulo the \emph{based} gauge group, which we denote by $\wt{\M}^{\irr}(Y,\frak{s},\nabla^{\infty})$, forms a principal $S^{1}$-bundle over the irreducible moduli space $\M^{\irr}(Y,\frak{s},\nabla^{\infty})$. One can show that the $S^{1}$-action $\rho:S^{1}\times Y\to Y$ given by rotation in the fibers induces a \emph{geometric} $S^{1}$-action on $\wt{\M}^{\irr}(Y,\frak{s},\nabla^{\infty})$, which acts freely with quotient $\M^{\irr}(Y,\frak{s},\nabla^{\infty})$. In particular, let $e^{i\theta}\cdot(a,\phi)$ denote the usual $S^{1}$-action on the based moduli space, and let
\[\wh{\rho}_{*}:S^{1}\times\wt{\M}^{\irr}(Y,\frak{s},\nabla^{\infty})\to\wt{\M}^{\irr}(Y,\frak{s},\nabla^{\infty})\]
denote the geometric $S^{1}$-action induced by the unique \emph{spin lift} $\wh{\rho}$ of $\rho$ to the principal $\Spin(3)$-bundle on $Y$ corresponding to $\frak{s}$. In \cite{Mon:SW} it was shown that $\wh{\rho}$ is a spin lift of even type if $\rho(Y)=0$, and of odd type if $\rho(Y)=\frac{1}{2}$, where
\[\rho(Y):=\twopartdef{0}{\text{the }\alpha_{i}\text{ are all odd},}{\frac{1}{2}}{\text{one of the }\alpha_{i}\text{ is even}.}\]

As in \cite{Mon:SW} we define the \emph{rotation number} of $E$ to be the unique half integer $\rot(E)\in\frac{1}{2}\ZZ$ such that 
\[\wh{\rho}_{*}(e^{i\theta},[a,\phi])=\twopartdef{e^{\pm i\rot(E)\theta}\cdot[a,\phi]}{\rho(Y)=0}{e^{\pm 2i\rot(E)\theta}\cdot[a,\phi]}{\rho(Y)=\frac{1}{2}}\]
for all $[a,\phi]\in\wt{\C}^{\pm}(E)$, where $\wt{\C}^{\pm}(E)$ denotes the preimage of $\C^{\pm}(E)$ under the projection $\wt{\M}^{\irr}(Y,\frak{s},\nabla^{\infty})\to\M^{\irr}(Y,\frak{s},\nabla^{\infty})$.

The following theorem provides an explicit formula for these rotation numbers:

\begin{theorem}
\label{theorem:S_1_action}
    Let $E=(e;\epsilon_{1},\dots,\epsilon_{n})$ be a line bundle over $\Sigma$ corresponding to a pair of connected components $\wt{\C}^{\pm}(E)\subset\wt{\M}^{\irr}(Y,\frak{s},\nabla^{\infty})$. Then the rotation number $\rot(E)\in\frac{1}{2}\ZZ$ is given by the following formula:
    \[\rot(E)=\frac{1}{\ell}\Bigg(g-e+\frac{n-2}{2}-\sum_{i=1}^{n}\frac{2\epsilon_{i}+1}{2\alpha_{i}}\Bigg).\]
    In particular if $Y=\Sigma(\alpha_{1},\dots,\alpha_{n})$ is a Seifert-fibered integer homology sphere with $\ell<0$ equipped with its unique spin structure, then
    \[\rot(E)=\alpha_{1}\cdots\alpha_{n}\Bigg(e-\frac{n-2}{2}+\sum_{i=1}^{n}\frac{2\epsilon_{i}+1}{2\alpha_{i}}\Bigg).\]
\end{theorem}

\begin{remark}
If $CSD(E)\in\RR$ denotes the value of the Chern-Simons-Dirac functional evaluated on any $(a,\phi)\in\C^{\pm}(E)$, then Theorem \ref{theorem:S_1_action} along with (\cite{MOY}, Theorem 5.23) implies the following relationship between $\rot(E)$ and $CSD(E)$:
\[\rot(E)=-\frac{\sign(\ell)}{2\pi|\ell|^{1/2}}\big|CSD(E)\big|^{1/2}.\]
\end{remark}

For the following, let $\wt{\nabla}^{\infty}:\Gamma(TY\otimes\SS)\to\Gamma(\SS)$ denote the spin covariant derivative induced by the adiabatic connection $\nabla^{\infty}$. Let $\zeta\in\Gamma(TY)$ be the Killing vector field associated to the $S^{1}$-action $\rho$, and let $\wt{\nabla}^{\infty}_{\zeta}:\Gamma(\SS)\to\Gamma(\SS)$ denote the spin covariant derivative in the direction of $\zeta$.

Let $\wt{\L}_{\zeta}:\Gamma(\SS)\to\Gamma(\SS)$ denote the Lie derivative of spinors along the Killing vector field $\zeta$, given by the formula
\begin{align*}
    &\wt{\L}_{\zeta}(\psi)(y):=\frac{d}{ds}\Bigg|_{s=0}\wh{\rho}(e^{is},\psi(\rho(e^{is},y))),
\end{align*}
for $\psi\in\Gamma(\SS)$, $y\in Y$. The significance of this operator for us comes from the following observation: the action of $\wh{\rho}$ on spinors endows $\Gamma(\SS)$ with the structure of an infinite-dimensional $S^{1}$-representation, and hence we can decompose $\Gamma(\SS)$ as
\[\Gamma(\SS)=\bigoplus_{\substack{\lambda\in\frac{1}{2}\ZZ \\ \lambda\equiv\rho(Y)\pmod{\ZZ}}}W_{\lambda},\]
where $W_{\lambda}\subset\Gamma(\SS)$ is such that 
\[\wh{\rho}_{*}(e^{is},\psi)=\twopartdef{e^{i\lambda s}\psi}{\rho(Y)=0}{e^{2i\lambda s}\psi}{\rho(Y)=\frac{1}{2}}\]
for all $\psi\in W_{\lambda}$. From its defining formula, we see that $\wt{\L}_{\zeta}$ can be interpreted as the formal derivative of $\wh{\rho}$ in the sense that $\wt{\L}_{\zeta}(\psi)=i\lambda\psi$ for all $\psi\in W_{\lambda}$. Hence the action of $\wh{\rho}_{*}$ on $\Gamma(\SS)$ is determined by that of $\wt{\L}_{\zeta}$.

We shall make use of the following lemma, whose proof will be featured in the upcoming article \cite{Mon:eta}:

\begin{lemma}[\cite{Mon:eta}]
\label{lemma:spinors_lie_derivative}
$\wt{\L}_{\zeta}=\wt{\nabla}_{\zeta}^{\infty}$.
\end{lemma}

We are now ready to prove Theorem \ref{theorem:S_1_action}:

\begin{proof}[Proof of Theorem \ref{theorem:S_1_action}]
Lemma \ref{lemma:spinors_lie_derivative} implies in particular that the eigenspaces of $\L_{\zeta}$ coincide with the eigenspaces of $\wt{\nabla}_{\zeta}^{\infty}$. In \cite{MOY} and \cite{Nic00}, it was shown that the $\lambda$-eigenspace of $\wt{\nabla}_{\zeta}^{\infty}$ can be canonically identified with the space of sections $\Gamma(E_{\lambda})$ of the orbifold line bundle $E_{\lambda}\to\Sigma$ which satisfies
\begin{align*}
    &\SS\approx\pi^{*}(E_{\lambda})\oplus\pi^{*}(K^{-1}_{\Sigma}\otimes E_{\lambda}), & &\deg E_{\lambda} = \ell_{s}-\ell(\lambda-\rho(Y)),
\end{align*}
where $\ell_{s}$ denotes the degree of the unique orbifold line bundle $L\to\Sigma$ satisfying
\begin{align*}
    &\pi^{*}(L)\oplus\pi^{*}(K^{-1}_{\Sigma}\otimes L)\approx\SS, & &\frac{\deg K_{\Sigma}-2\ell_{s}}{2\ell}=\rho(Y).
\end{align*}
We can therefore write
\begin{equation}
\label{eq:lambda}
    \lambda=\rho(Y)+\frac{\ell_{s}-\deg E_{\lambda}}{\ell}=\frac{\deg K_{\Sigma}-2\ell_{s}}{2\ell}+\frac{2\ell_{s}-2\deg E_{\lambda}}{2\ell}=\frac{\deg K_{\Sigma}-2\deg E_{\lambda}}{2\ell}.
\end{equation}

In \cite{MOY}, it was shown that any solution $[(a,\psi)]\in\M^{\irr}(Y,\frak{s},\nabla^{\infty})$ has a representative $(a,\psi)$ for which $\psi\in\Gamma(\SS)$ is an eigenvector of $\wt{\nabla}^{\infty}$. Moreover, the correspondence between irreducible solutions and line bundles over $\Sigma$ is given by $[(a,\psi)]\mapsto E_{\lambda}$, where $\lambda$ is the corresponding eigenvalue of $\psi$. With respect to our notation, this implies that the rotation number $\rot(E_{\lambda})$ is precisely equal to $\lambda$.

Therefore given a bundle $E=(e;\epsilon_{1},\dots,\epsilon_{n})$ corresponding to a pair of components $\wt{\C}^{\pm}(E)\subset\wt{\M}^{\irr}(Y,\frak{s},\nabla^{\infty})$, by (\ref{eq:lambda}) the rotation number is given by
\begin{align*}
    \rot(E)&=\frac{\deg K_{\Sigma}-2\deg E}{2\ell}=\frac{2g-2+\sum_{i=1}^{n}\frac{\alpha_{i}-1}{\alpha_{i}}-2(e+\sum_{i=1}^{n}\frac{\epsilon_{i}}{\alpha_{i}})}{2\ell} \\
    &=\frac{1}{\ell}\Bigg(g-e+\frac{n-2}{2}-\sum_{i=1}^{n}\frac{2\epsilon_{i}+1}{2\alpha_{i}}\Bigg),
\end{align*}
which is what was to be proven. Finally, the formula for $Y=\Sigma(\alpha_{1},\dots,\alpha_{n})$ is obtained from the above formula by setting $g=0$, $\ell=-\frac{1}{\alpha_{1}\cdots\alpha_{n}}$.
\end{proof}

We record here the rotation numbers corresponding to the irreducible solutions on the families of Brieskorn spheres $\Sigma(2,3,6n\pm 1)$:

\smallskip

\begin{center}
\renewcommand{\arraystretch}{1.4} 
\label{table:rotation_numbers}
\begin{tabular}{|c|c|c|}
\hline
$Y$ & $E$ & $\rot(E)$ \\ \hline\hline
$\Sigma(2,3,12n+5)$ & $(0;0,0,k),\;0\le k\le n-1$ & $-\tfrac{1}{2}(12(n-k)-1)$ \\ \hline
$\Sigma(2,3,12n-5)$ & $(0;0,0,k),\;0\le k\le n-1$ & $-\tfrac{1}{2}(12(n-k)-11)$ \\ \hline
$\Sigma(2,3,12n-1)$ & $(0;0,0,k),\;0\le k\le n-1$ & $-\tfrac{1}{2}(12(n-k)-7)$ \\ \hline
$\Sigma(2,3,12n+1)$ & $(0;0,0,k),\;0\le k\le n-1$ & $-\tfrac{1}{2}(12(n-k)-5)$ \\
\hline
\end{tabular}
\end{center}

\section{Equivariant Kappa Invariants for \texorpdfstring{$\ZZ_{p}$}{Zp}-Actions}
\label{sec:kappa}

In Section \ref{subsec:background} we review the construction of equivariant $\kappa$-invariants $\K(Y,\frak{s},\sigma)$ from \cite{Mon:SW}, and in Section \ref{subsec:calculations} we calculate these invariants in the case where $Y=\pm\Sigma(2,3,6n\pm 1)$, $\sigma=\rho_{p}$ for odd primes $p$. In Section \ref{subsec:10_8_ths} we prove a slight variant of the equivariant relative $10/8$-ths inequality from \cite{Mon:SW}, as well as a proposition which obstructs smooth extensions of $\ZZ_{p}$-actions over spin 4-manifolds with boundary for which Manolescu's relative $10/8$-ths inequality is sharp. Afterwards in Section \ref{subsec:obstructing_smooth_extensions} we prove Theorem \ref{theorem:intro_no_smooth_extension}, as well as an analogous theorem for certain cobordisms.

\subsection{Background}
\label{subsec:background}

Let $p$ be an odd prime, and consider triples of the form $(Y,\frak{s},\sigma)$, where $Y$ denotes a rational homology sphere equipped with a spin structure $\frak{s}$, and $\sigma:Y\to Y$ is a self-diffeomorphism of order $p$ which preserves $\frak{s}$; we refer to such a triple as a \emph{$\ZZ_{p}$-equivariant spin rational homology sphere}. If $Y$ has a unique spin structure, e.g., if $Y$ is an integer homology sphere, then we will drop $\frak{s}$ from the notation and simply write $(Y,\sigma)$.

In \cite{Mon:SW}, the author constructed a set of \emph{equivariant $\kappa$-invariants} associated to $(Y,\frak{s},\sigma)$, which take the form of a finite subset $\K(Y,\frak{s},\sigma)$ of a certain poset $\Q^{p}$ constructed from the complex representation ring $R(\Pin(2)\times\ZZ_{p})$. In particular, $\Q^{p}=(\Q^{p},\preceq,+)$ has the structure of a \emph{$\QQ$-graded additive poset}, i.e., $(\Q^{p},+)$ is an additive monoid endowed with a partial order $\preceq$ which is compatible with $+$ in a suitable sense, along with an additive poset homomorphism $|\cdot|:(\Q^{p},\preceq,+)\to(\QQ,\le,+)$ referred to as the \emph{$\QQ$-grading} on $\Q^{p}$. For every $n\geq 1$, the $\QQ$-vector space
\[\QQ^{n}=\text{span}_{\QQ}\{\vec{e}_{0},\dots,\vec{e}_{n-1}\}\]
has a natural additive poset structure with respect to vector space addition and the product partial order given by
\begin{align*}
    &(a_{0},\dots,a_{n-1})\preceq (b_{0},\dots,b_{n-1}) & &\iff & &a_{i}\le b_{i}\text{ for all }i=0,\dots,n-1,
\end{align*}
as well as a $\QQ$-grading given by $|(a_{0},\dots,a_{n-1})|:=a_{0}+\cdots+a_{n-1}$. The following properties of $\Q^{p}$ will be useful to us:
\begin{itemize}
    \item There exists a surjection of $\QQ$-graded additive posets
    \begin{equation}
    \label{eq:kappa_projection_1}
        \Pi:(\QQ^{p},\preceq,+)\to(\Q^{p},\preceq,+),
    \end{equation}
    which we call the \emph{defining projection} for $\Q^{p}$. In particular, we will often denote elements of $\Q^{p}$ by $[\vec{v}]$, $\vec{v}\in\QQ^{p}$.
    \item There exists a surjection of $\QQ$-graded additive posets
    \begin{equation}
    \label{eq:kappa_projection_2}
        \pi:(\Q^{p},\preceq,+)\to(\QQ^{2},\preceq,+)
    \end{equation}
    such that if $\wt{\pi}$ denotes the projection
    \begin{align*}
        \wt{\pi}:(\QQ^{p},\preceq,+)&\to(\QQ^{2},\preceq,+) \\
        (a_{0},\dots,a_{p-1})&\mapsto(a_{0},a_{1}+\cdots+a_{p-1}),
    \end{align*}
    then $\wt{\pi}=\pi\circ\Pi$.
\end{itemize}
We will now briefly outline the construction of $\K(Y,\frak{s},\sigma)\subset\Q^{p}$ from \cite{Mon:SW}. Consider the groups
\begin{align*}
    &G^{\ev}_{p}=\Pin(2)\times\ZZ_{p}, & &G^{\odd}_{p}=\Pin(2)\times_{\ZZ_{2}}\ZZ_{2p}.
\end{align*}
For $*\in\{\ev,\odd\}$, a space $X$ of type $\CC$-$G^{*}_{p}$-$\SWF$ is a finite $G^{*}_{p}$-CW-complex such that:
\begin{enumerate}
    \item The $S^{1}$-fixed point set $X^{S^{1}}$ is equivariantly homotopy equivalent to a complex representation sphere on which $j\in\Pin(2)<G^{*}_{p}$ acts by scalar multiplication by $-1\in\CC$.
    \item $\Pin(2)<G^{*}_{p}$ acts freely on the complement $X\setminus X^{S^{1}}$.
\end{enumerate}
The inclusion map $\iota:X^{S^{1}}\hookrightarrow X$ induces a map $\iota^{*}:\wt{K}_{G^{*}_{p}}(X)\to\wt{K}_{G^{*}_{p}}(X^{S^{1}})$ on reduced $G^{*}_{p}$-equivariant $K$-theory, whose image $\III(X):=\im(\iota^{*})$ can be viewed as an ideal of the complex representation ring $R(G^{*}_{p})$ via the equivariant Bott isomorphism $\wt{K}_{G^{*}_{p}}(X^{S^{1}})\cong R(G^{*}_{p})$. There exist presentations
\begin{align*}
    &R(G^{\ev}_{p})=\ZZ[w_{0},\dots,w_{p-1},z_{0},\dots,z_{p-1}]/\I^{\ev}_{p} \\
    &R(G^{\odd}_{p})=\ZZ[w_{0},\dots,w_{p-1},z_{1/2},\dots,z_{(2p-1)/2}]/\I^{\odd}_{p}
\end{align*}
for some ideals $\I^{*}_{p}\subset R(G^{*}_{p})$, $*\in\{\ev,\odd\}$, allowing us to define the following subsets:
\begin{align*}
    &I(X):=\{(k_{0},\dots,k_{p-1})\in\QQ^{p}\;|\;\exists x\in\III(X)\text{ such that }w_{0}x=w_{0}^{k_{0}+1}w_{1}^{k_{1}}\cdots w_{p-1}^{k_{p-1}}\}\subset\QQ^{p}, \\
    &\mbfk(X):=\min(\Pi(I(X)))=\{[\vec{k}]\in \Pi(I(X))\;|\;[\vec{k}]\not\succ[\vec{k}']\text{ for all }[\vec{k}']\in \Pi(I(X))\}\subset\Q^{p},
\end{align*}
where $\Pi:\QQ^{p}\to\Q^{p}$ denotes the defining projection map from (\ref{eq:kappa_projection_1}).

A $\CC$-$G^{*}_{p}$-spectrum class is a triple $[(X,\mbfs,\mbft)]$ where $X$ is a space of type $\CC$-$G^{*}_{p}$-$\SWF$, $\mbfs\in R(\ZZ_{p})$, and $\mbft\in R(\ZZ_{2p})^{*}\otimes\QQ$, considered up to a certain notion of equivalence. Here, we identify $R(\ZZ_{2p})=\ZZ[\xi]/(\xi^{2p}-1)$, and $R(\ZZ_{2p})^{\ev},R(\ZZ_{2p})^{\odd}< R(\ZZ_{2p})$ denote the additive subgroups
\begin{align*}
    &R(\ZZ_{2p})^{\ev}=\text{span}_{\ZZ}\{\xi^{2k}\;|\;k=0,\dots,p-1\}, & &R(\ZZ_{2p})^{\odd}=\text{span}_{\ZZ}\{\xi^{2k+1}\;|\;k=0,\dots,p-1\}.
\end{align*}
Given $\mbft\in R(\ZZ_{2p})^{*}\otimes\QQ$, we write
\begin{equation}
\label{eq:associated_vector}
    \vec{\mbft}=\twopartdef{(t_{0},\dots,t_{p-1})\in\QQ^{p}}{\ast=\ev,\;\mbft=\sum_{i=0}^{p-1}t_{i}\xi^{2i},}{(t_{1/2},\dots,t_{(2p-1)/2})\in\QQ_{1/2}^{p}}{\ast\in\odd,\;\mbft=\sum_{i=0}^{p-1}t_{(2i+1)/2}\xi^{2i+1},}
\end{equation}
where $\QQ^{p}_{1/2}=\text{span}_{\QQ}\{\vec{e}_{1/2},\vec{e}_{3/2},\dots,\vec{e}_{(2p-1)/2}\}$. We then extend the definition of $\mbfk(X)$ to the setting of $\CC$-$G^{*}_{p}$-spectrum classes by setting
\[\mbfk\big([(X,\mbfs,\mbft)]\big)=\mbfk(X)-[\DDD^{*}(\vec{\mbft})]:=\{[\vec{k}]-[\DDD^{*}(\vec{\mbft})]\;|\;[\vec{k}]\in\mbfk(X)\}\subset\Q^{p},\]
where $\DDD^{*}$ is the ``doubling'' map
\begin{align*}
    \DDD^{\ev}:\QQ^{p}&\to\QQ^{p} & \DDD^{\odd}:\QQ^{p}_{1/2}&\to\QQ^{p} \\
    \vec{e}_{i}&\mapsto{\left\{
		\begin{array}{ll}
            \vec{e}_{2i} & \mbox{if } 0\le i\le\frac{p-1}{2}, \\
			\vec{e}_{2i-p} & \mbox{if } \frac{p+1}{2}\le p-1,
		\end{array}
	\right.} & \vec{e}_{(2i+1)/2}&\mapsto{
	\left\{
		\begin{array}{ll}
            \vec{e}_{2i+1} & \mbox{if } 0\le i\le\frac{p-3}{2}, \\
			\vec{e}_{2i+1-p} & \mbox{if } \frac{p-1}{2}\le p-1.
		\end{array}
	\right.}
\end{align*}
Now let $(Y,\frak{s},\sigma)$ be a $\ZZ_{p}$-equivariant spin rational homology sphere, and let $\wh{\sigma}$ be a choice of spin lift of $\sigma$. The \emph{Seiberg--Witten Floer stable homotopy type} of $(Y,\frak{s},\wh{\sigma})$ is the $\CC$-$G^{*}_{p}$-spectrum class
\[\SWF(Y,\frak{s},\wh{\sigma})=\big[\big(\Sigma^{W}I_{-\lambda}^{\lambda},\tfrac{1}{2}c(\mbf{v}_{-\lambda}^{0}(\wt{\RR})+[W]),\mbf{v}_{-\lambda}^{0}(\HH)+\tfrac{1}{2}n(Y,\frak{s},\wh{\sigma},g)\big)\big],\]
where:
\begin{itemize}
    \item $\ast\in\{\ev,\odd\}$ depending on whether $\wh{\sigma}$ is an even or odd spin lift.
    \item $I_{-\lambda}^{\lambda}$ is a $G^{*}_{p}$-equivariant Conley index for the projected flow of $CSD$ onto the finite-dimensional subspace $V_{-\lambda}^{\lambda}$ of the Coulomb slice with respect to an eigenvalue cut-off $\lambda>>0$, and a choice of equivariant metric $g$ on $Y$.
    \item $\mbf{v}_{-\lambda}^{0}(\wt{\RR})\in RO(\ZZ_{p})$ and $\mbf{v}_{-\lambda}^{0}(\HH)\in R(\ZZ_{2p})^{*}$ are representations corresponding to the equivariant vector space decomposition $V_{-\lambda}^{0}=V_{-\lambda}^{0}(\wt{\RR})\oplus V_{-\lambda}^{0}(\HH)$, where $RO(\ZZ_{p})$ denotes the real representation ring of $\ZZ_{p}$.
    \item $c:RO(\ZZ_{p})\to R(\ZZ_{p})$ denotes the complexification map, and $W$ is a real $\ZZ_{p}$-representation chosen so that $c(\mbf{v}_{-\lambda}^{0}(\wt{\RR})+[W])\in R(\ZZ_{p})$ is divisible by 2.
    \item $n(Y,\frak{s},\wh{\sigma},g)\in R(\ZZ_{2p})^{*}\otimes\QQ$ is the \emph{equivariant correction term} of the quadruple $(Y,\frak{s},\wh{\sigma},g)$, whose variation under one-parameter variations of the metric is given by the equivariant spectral flow of the Dirac operator $\dirac$ on $(Y,\frak{s})$.
\end{itemize}
The set of equivariant $\kappa$-invariants of $(Y,\frak{s},\sigma)$ are then given by
\[\K(Y,\frak{s},\sigma):=2\cdot\mbfk(\SWF(Y,\frak{s},\wh{\sigma})):=\{[2\vec{k}]\;|\;[\vec{k}]\in\mbfk^{\st}(\SWF(Y,\frak{s},\wh{\sigma}))\},\]
which in \cite{Mon:SW} was shown to be independent of the choice of spin lift $\wh{\sigma}$ of $\sigma$. In this article, we will also make use of the \emph{projected} equivariant $\kappa$-invariants
\[\K^{\pi}(Y,\frak{s},\sigma):=\pi\big(\K(Y,\frak{s},\sigma)\big)\subset\QQ^{2},\]
where $\pi$ is the surjection $\Q^{p}\to\QQ^{2}$ from (\ref{eq:kappa_projection_2}).

If the ideal $\III(\Sigma^{W}I^{\lambda}_{-\lambda})\subset R(G^{*}_{p})$ is generated by a single monomial in the $z_{i}$-variables for any choice of spin lift $\wh{\sigma}$ of $\sigma$, we say that $(Y,\frak{s},\sigma)$ is Floer $K_{G^{*}_{p}}$-split, or just $K_{G^{*}_{p}}$-split for brevity. In this case $\K(Y,\frak{s},\sigma)$ consists of a single element.

\subsection{Calculations}
\label{subsec:calculations}

We now specialize to the case of Seifert-fibered homology spheres, equipped with the standard $\ZZ_{p}$-action $\rho_{p}$ along with the distinguished spin lift $\wh{\rho}_{p}$ as in the previous section. We will use the equivariant correction term calculated with respect to the adiabatic connection $\nabla^{\infty}$ from Section \ref{sec:s_1_action}, which we denote by
\[n(Y,\wh{\rho}_{p},g,\nabla^{\infty})\in R(\ZZ_{2p})^{*}\otimes\QQ.\]
Let $\vec{n}(Y,\wh{\rho}_{p},g,\nabla^{\infty})\in\QQ^{p}$ or $\QQ_{1/2}^{p}$ be the associated vector as in (\ref{eq:associated_vector}). For brevity, we will write
\begin{align*}
    &\vec{n}(Y,p):=\DDD^{*}(\vec{n}(Y,\wh{\rho}_{p},g,\nabla^{\infty}))\in\QQ^{p}, &
    &\vec{n}^{\pi}(Y,p)=(n_{0}(Y,p),n_{1}(Y,p)):=\pi([\vec{n}(Y,p)])\in\QQ^{2}.
\end{align*}
In \cite{Mon:SW} it was shown that $\vec{n}(Y,p)$ and $\vec{n}^{\pi}(Y,p)$ satisfy the following properties:
\begin{enumerate}
    \item $|\vec{n}(Y,p)|=|\vec{n}^{\pi}(Y,p)|=n(Y,\frak{s},g,\nabla^{\infty})$.,
    \item $\vec{n}(-Y,p)=-\vec{n}(Y,p)$, $\vec{n}^{\pi}(-Y,p)=-\vec{n}^{\pi}(Y,p)$.
\end{enumerate}
We have the following result for the families $Y=\Sigma(2,3,12n+5)$ or $\Sigma(2,3,12n+1)$:

\begin{proposition}
\label{prop:seifert_K_split}[\cite{Mon:SW}, Corollary 8.17]
Let $p\geq 3$ be prime, and let $Y=\Sigma(2,3,12n+5)$, $n\geq 0$, or $\Sigma(2,3,12n+1)$, $n\geq 1$. Then $(Y,\rho_{p})$ is Floer $K_{G^{*}_{p}}$-split, and
\begin{align*}
    &\K(\pm Y,\rho_{p})=\{\mp[\vec{n}(Y,p)]\}, & 
    &\K^{\pi}(\pm Y,\rho_{p})=\{\mp\vec{n}^{\pi}(Y,p)\}.
\end{align*}
\end{proposition}

In contrast, the equivariant $\kappa$-invariants of $(\pm\Sigma(2,3,12n-5),\rho_{p})$ and $(\pm\Sigma(2,3,12n-1),\rho_{p})$ have a much richer structure:

\begin{proposition}
\label{prop:seifert}
Let $p\geq 3$ be prime.
\begin{enumerate}
    \item The set of equivariant $\kappa$-invariants of $\Sigma(2,3,12n-5)$ and $\Sigma(2,3,12n-1)$ are given as follows:
    \begin{align*}
        &\K(\Sigma(2,3,12n-5),\rho_{p}) \\
        &\qquad={\left\{
        \begin{array}{ll}
        \{[2\vec{e}_{0}],[2\vec{e}_{1}],[2\vec{e}_{2}]\}-[\vec{n}(\Sigma(2,3,12n-5),3)] & \mbox{if } p=3\text{ and }n\geq 1, \\
        \{[2\vec{e}_{0}],[2\vec{e}_{1}],[2\vec{e}_{p-1}]\}-[\vec{n}(\Sigma(2,3,12n-5),p)] & \mbox{if } p\geq 5\text{ and }n=1, \text{ or} \\
        & \mbox{if } p=7\text{ and }n=2,\\
        \{[2\vec{e}_{0}-\vec{n}(\Sigma(2,3,12n-5),p)]\} & \mbox{if } p\geq 5, p\neq 7\text{ and }n\geq 2, \text{ or} \\
        &\mbox{if } p=7\text{ and }n\geq 3,
        \end{array}
        \right.} \\
        &\K(\Sigma(2,3,12n-1),\rho_{p}) \\
        &\qquad={\left\{
        \begin{array}{ll}
        \{[2\vec{e}_{0}],[2\vec{e}_{1}],[2\vec{e}_{2}]\}-[\vec{n}(\Sigma(2,3,12n-1),3)] & \mbox{if } p=3\text{ and }n\geq 1, \\
        \{[2\vec{e}_{0}],[2\vec{e}_{5}],[2\vec{e}_{p-5}]\}-[\vec{n}(\Sigma(2,3,12n-1),p)] & \mbox{if } p\geq 7\text{ and }n=1, \text{ or} \\
        & \mbox{if } p=11\text{ and }n=2,\\
        \{[2\vec{e}_{0}-\vec{n}(\Sigma(2,3,12n-1),p)]\} & \mbox{if } p\geq 5, p\neq 11\text{ and }n\geq 2, \text{ or} \\
        &\mbox{if } p=5\text{ and }n=1, \\
        &\mbox{if } p=11\text{ and }n\geq 3.
        \end{array}
        \right.}
    \end{align*}
    In particular, we have that:
    \begin{align*}
        &\K^{\pi}(\Sigma(2,3,12n-5),\rho_{p}) \\
        &\qquad={\left\{
        \begin{array}{ll}
        \{(2,0),(0,2)\}-\vec{n}^{\pi}(\Sigma(2,3,12n-5),p) & \mbox{if } p=3\text{ and }n\geq 1, \\
        & \mbox{if } p\geq 5\text{ and }n=1, \text{ or} \\
        & \mbox{if } p=7\text{ and }n=2,\\
        \{(2,0)-\vec{n}^{\pi}(\Sigma(2,3,12n-5),p)\} & \mbox{if } p\geq 5, p\neq 7\text{ and }n\geq 2, \text{ or} \\
        &\mbox{if } p=7\text{ and }n\geq 3,
        \end{array}
        \right.} \\
        &\K^{\pi}(\Sigma(2,3,12n-1),\rho_{p}) \\
        &\qquad={\left\{
        \begin{array}{ll}
        \{(2,0),(0,2)\}-\vec{n}^{\pi}(\Sigma(2,3,12n-1),p) & \mbox{if } p=3\text{ and }n\geq 1, \\
        & \mbox{if } p\geq 7\text{ and }n=1, \text{ or} \\
        & \mbox{if } p=11\text{ and }n=2,\\
        \{(2,0)-\vec{n}^{\pi}(\Sigma(2,3,12n-1),p)\} & \mbox{if } p\geq     5, p\neq 11\text{ and }n\geq 2, \text{ or} \\
        &\mbox{if } p=5\text{ and }n=1, \\
        &\mbox{if } p=11\text{ and }n\geq 3.
        \end{array}
        \right.}
    \end{align*}
    \item The set of equivariant $\kappa$-invariants of $-\Sigma(2,3,12n-5)$ and $-\Sigma(2,3,12n-1)$ are given as follows:
    \begin{align*}
        &\K(-\Sigma(2,3,12n-5),\rho_{p}) \\
        &\qquad\qquad=\Big\{2[\vec{a}]+[\vec{n}(\Sigma(2,3,12n-5),p)]\;\Big|\;\vec{a}\succeq(0,-A_{n,p,1},\dots,-A_{n,p,p-1}),\;|\vec{a}|=0\Big\}, \\
        &\K(-\Sigma(2,3,12n-1),\rho_{p}) \\
        &\qquad\qquad=\Big\{2[\vec{a}]+[\vec{n}(\Sigma(2,3,12n-1),p)]\;\Big|\;\vec{a}\succeq(0,-B_{n,p,1},\dots,-B_{n,p,p-1}),\;|\vec{a}|=0\Big\},
    \end{align*}
    where for each $j=0,1,\dots,p-1$:
    \begin{align*}
        &A_{n,p,j}:=\#\{1\le k\le n\;|\;12k\equiv 11-j\pmod{p}\}, \\
        &B_{n,p,j}:=\#\{1\le k\le n\;|\;12k\equiv 7-j\pmod{p}\}.
    \end{align*}
    In particular, we have that:
    \begin{align*}
        &\K^{\pi}(-\Sigma(2,3,12n-5),\rho_{p})=\{(2k,-2k)+\vec{n}^{\pi}(\Sigma(2,3,12n-5),p)\;|\;0\le k\le n-A_{n,p,0}\}, \\
        &\K^{\pi}(-\Sigma(2,3,12n-1),\rho_{p})=\{(2k,-2k)+\vec{n}^{\pi}(\Sigma(2,3,12n-1),p)\;|\;0\le k\le n-B_{n,p,0}\}.
    \end{align*}
\end{enumerate}
\end{proposition}

\begin{proof}
We first show (1). For $r\in\frac{1}{2}\ZZ\setminus\ZZ$ let $Z_{r,p}$ denote the $G^{\odd}_{p}$-space
\[Z_{r,p}:=G^{\odd}_{p}/\<e^{-2\pi ir/p}\mu\>,\]
and for $r_{1},\dots,r_{n}\in\frac{1}{2}\ZZ\setminus\ZZ$ let $Z_{r_{1},\dots,r_{n};p}$ denote the disjoint union
\[Z_{r_{1},\dots,r_{n};p}:=Z_{r_{1},p}\amalg\cdots\amalg Z_{r_{n},p}.\]
By (\cite{Mon:SW}, Proposition 8.16), the $G^{\odd}_{p}$-equivariant Seiberg--Witten Floer spectra of $Y=\Sigma(2,3,12n-5)$ or $\Sigma(2,3,12n-1)$ is given by the $\CC$-$G^{\odd}_{p}$-spectrum class
\[\SWF(Y,\wh{\rho}_{p})=\big[\big(\wt{\Sigma}Z_{r_{1},\dots,r_{n};p},0,\tfrac{1}{2}n(Y,\wh{\rho}_{p},g,\nabla^{\infty})\big)\big],\]
where $\wt{\Sigma}Z_{r_{1},\dots,r_{n};p}$ denotes the unreduced suspension of $Z_{r_{1},\dots,r_{n};p}$, and $r_{1},\dots,r_{n}$ denote the set of rotation numbers for the $n$ pairs of irreducible solutions on $Y$. Furthermore, it was shown that (\cite{Mon:SW}, Example 4.83)
\begin{align*}
    &\mbfk\big(\wt{\Sigma}Z_{r_{1},\dots,r_{n};p}\big) \\
    &={
    \left\{
    \begin{array}{ll}
    \{[\vec{e}_{0}],[\vec{e}_{2r}],[\vec{e}_{p-2r}]\} & \mbox{if } \exists\,r\in\frac{1}{2}\ZZ\setminus\ZZ\text{ such that }r_{i}\equiv\pm r\pmod{p}\;\;\forall i=1,\dots,n\;(\dagger) \\
    \{[\vec{e}_{0}]\} & \mbox{otherwise}
    \end{array}
    \right.}
\end{align*}
and hence
\[\K(Y,\rho_{p})={
    \left\{
    \begin{array}{ll}
    \{[2\vec{e}_{0}],[2\vec{e}_{2r}],[2\vec{e}_{p-2r}]\}-[\vec{n}(Y,p)] & \mbox{if }(\dagger)\text{ holds}\\
    \{[2\vec{e}_{0}]\}-[\vec{n}(Y,p)] & \mbox{otherwise}
    \end{array}
    \right.}\]
for $Y=\Sigma(2,3,12n-5)$ or $\Sigma(2,3,12n-1)$ (See \cite{Mon:SW}, Proposition 8.19). From Table \ref{table:rotation_numbers} we see that the set of rotation numbers $r_{1},\dots,r_{n}$ are given by:
\begin{align*}
    &(r_{1},\dots,r_{n})=(-\tfrac{1}{2},-\tfrac{13}{2},\dots,-\tfrac{12n-11}{2}) & &\text{ if }Y=\Sigma(2,3,12n-5), \\
    &(r_{1},\dots,r_{n})=(-\tfrac{5}{2},-\tfrac{17}{2},\dots,-\tfrac{12n-7}{2}) & &\text{ if }Y=\Sigma(2,3,12n-1).
\end{align*}
Now consider the case $Y=\Sigma(2,3,12n-5)$:
\begin{enumerate}
    \item If $p=3$ and $n\geq 1$, we have that 
    \[r_{k}=-\tfrac{12k-11}{2}\equiv-\tfrac{1}{2}\pmod{3}\qquad\text{for all }k=1,\dots,n.\]
    Hence
    \[\K(\Sigma(2,3,12n-5))=\{[2\vec{e}_{0}],[2\vec{e}_{1}],[2\vec{e}_{2}]\}-[\vec{n}(\Sigma(2,3,12n-5),3)].\]
    \item If $p\geq 5$:
    \begin{enumerate}
        \item If $p\geq 5$ and $n=1$, then $(\dagger)$ always holds for $r=-\tfrac{1}{2}$, and so:
        \[\K(\Sigma(2,3,7))=\{[2\vec{e}_{0}],[2\vec{e}_{1}],[2\vec{e}_{p-1}]\}-[\vec{n}(\Sigma(2,3,7),p)].\]
        \item $p=7$ and $n=2$, then 
        \[r_{1}=-\tfrac{1}{2}\equiv\tfrac{13}{2}=-r_{2}\pmod{7},\]
        and hence
        \[\K(\Sigma(2,3,19))=\{[2\vec{e}_{0}],[2\vec{e}_{1}],[2\vec{e}_{6}]\}-[\vec{n}(\Sigma(2,3,19),7)].\]
        \item If $p\geq 5$, $n\geq 2$, $(p,n)\neq (7,2)$, then there exists some $2\le k\le n$ such that 
        \[r_{k}=-\tfrac{12k-11}{2}\not\equiv\pm\tfrac{1}{2}=\pm r_{1}\pmod{p}.\]
        Hence
        \[\K(\Sigma(2,3,12n-5))=\{[2\vec{e}_{0}]\}-\vec{n}(\Sigma(2,3,12n-5),p)].\]
    \end{enumerate}
\end{enumerate}
The case $Y=\Sigma(2,3,12n-1)$ is analogous and left as an exercise to the reader.

Next we prove (2). For $r\in\frac{1}{2}\ZZ\setminus\ZZ$ let $\HH_{r}$ denote the 2-dimensional complex $G^{\odd}_{p}$-representation on which:
\begin{itemize}
    \item $\Pin(2)\subset G^{\odd}_{p}$ acts by the usual (left) action of $\Pin(2)$ on the quaternions $\HH$.
    \item $\<\mu\>=\ZZ_{2p}\subset G^{\odd}_{p}$ acts by $\mu\mapsto(\begin{smallmatrix}
    e^{2\pi ir/p} & 0 \\
    0 & e^{-2\pi ir/p}
    \end{smallmatrix})$.
\end{itemize}
For $r_{1},\dots,r_{n}\in\frac{1}{2}\ZZ\setminus\ZZ$ consider the $G^{\odd}_{p}$-space
\[X_{r_{1},\dots,r_{n};p}:=S(\HH_{r_{1}}\oplus\cdots\oplus\HH_{r_{n}})\setminus Z_{r_{1},\dots,r_{n};p},\]
where the embedding 
$Z_{r_{1},\dots,r_{n};p}\hookrightarrow S(\HH_{r_{1}}\oplus\cdots\oplus\HH_{r_{n}})$ is induced by the collection of $G^{\odd}_{p}$-equivariant embeddings $Z_{r_{k},p}\approx S^{1}\cup jS^{1}\hookrightarrow S(\HH_{r_{k}})$.

Again in (\cite{Mon:SW}, Proposition 8.16) it was shown that the equivariant Floer spectra of $-Y=-\Sigma(2,3,12n-5)$ or $-\Sigma(2,3,12n-1)$ are given by
\[\SWF(-Y,\wh{\rho}_{p})=\Big[\Big(\wt{\Sigma}X_{r_{1},\dots,r_{n};p},0,\tfrac{1}{2}n(-Y,\wh{\rho}_{p},g,\nabla^{\infty})+\sum_{i=1}^{n}\xi^{2r_{i}}\Big)\Big],\]
where $r_{1},\dots,r_{n}$ denote the set of rotation numbers for $Y$ as above. Furthermore, it was shown that (\cite{Mon:SW}, Example 4.83):
\begin{align*}
    &\mbfk\big(\wt{\Sigma}X_{r_{1},\dots,r_{n};p}\big)=\Big\{[\vec{a}]\;\Big|\;\vec{a}\succeq(n_{0},0,\dots,0),\;|\vec{a}|=n\Big\}, \\
    &n_{j}:=\#\{1\le k\le n\;|\;2r_{k}\equiv j\pmod{p}\},\qquad 0\le j\le p-1,
\end{align*}
and as a consequence (\cite{Mon:SW}, Proposition 8.19):
\[\K(-Y,\rho_{p})=\Big\{2[\vec{a}]-2[(n_{0},\dots,n_{p-1})]+[\vec{n}(Y,p)]\;\Big|\;\vec{a}\succeq(n_{0},0,\dots,0),\;|\vec{a}|=n\Big\}.\]
The result then follows via the identifications $n_{j}=A_{n,p,j}$ if $Y=\Sigma(2,3,12n-5)$ and $n_{j}=B_{n,p,j}$ if $Y=\Sigma(2,3,12n-1)$, via Table \ref{table:rotation_numbers}.
\end{proof}

\begin{corollary}
\label{cor:more_than_one_element}
Let $Y=\pm\Sigma(2,3,12n-5)$ or $\pm \Sigma(2,3,12n-1)$. The subset $\K(Y,\rho_{p})\subset\QQ^{2}$ contains more than one element if and only if:
\begin{enumerate}
    \item $Y=\Sigma(2,3,12n-5)$: $p=3$, $n\geq 1$ or $p\geq 5$, $n=1$ or $(n,p)=(2,7)$.
    \item $Y=\Sigma(2,3,12n-1)$: $p=3$, $n\geq 1$ or $p\geq 7$, $n=1$ or $(n,p)=(2,11)$.
    \item $Y=-\Sigma(2,3,12n-5)$: $p\geq 3$, $n\geq 1$.
    \item $Y=-\Sigma(2,3,12n-1)$: $p\geq 3$, $n\geq 1$, $(n,p)\neq (1,5)$.
\end{enumerate}
\end{corollary}

\subsection{Equivariant Relative 10/8-ths Inequalities}
\label{subsec:10_8_ths}

We first recall some definitions from \cite{Mon:SW}, specialized to the case of $\ZZ_{p}$-actions for $p$ an odd prime.

Let $(X,\frak{t})$ be a compact spin 4-manifold such that $\del X$ is either empty or a disjoint union of rational homology spheres, and let $\tau:X\to X$ be a locally linear $\ZZ_{p}$-action such that $\tau^{*}(\frak{t})-\frak{t}=0\in H^{1}(X,\ZZ_{2})$. To the pair $(X,\tau)$ we associate two quantities $\vec{b}_{2}^{+}(X,\tau)\in\ZZ^{p}$ and $\vec{\SSS}(X,\tau)\in\QQ^{p}$ which satisfy $|\vec{b}_{2}^{+}(X,\tau)|=b_{2}^{+}(X)$, $|\vec{\SSS}(X,\tau)|=\sigma(X)$:

First let $H^{2}_{+}(X,\CC)\subset H^{2}(X,\CC)$ be a maximal positive definite subspace, and let
\[\vec{b}_{2}^{+}(X,\tau)=(b_{2}^{+}(X,\tau)_{0},\dots,b_{2}^{+}(X,\tau)_{p-1})\in\ZZ^{p},\]
where $b_{2}^{+}(X,\tau)_{k}$ denotes the complex dimension of the $e^{2\pi ik/p}$-eigenspace of the induced action of $\tau$ on $H^{2}_{+}(X,\CC)$ for $k=0,\dots,p-1$.

Now let $p_{1},\dots,p_{M}$ and $\Sigma_{1},\dots,\Sigma_{N}$ be enumerations of the fixed points and fixed surfaces of $\tau$, respectively. For each $i=1,\dots,M$ choose a complex structure on the normal bundle $\nu(p_{i})$ of $p_{i}$, so that the action of $\tau$ on $\nu(p_{i})$ can be identified with $(\begin{smallmatrix}
e^{2\pi ia_{i}/p} & 0 \\
0 & e^{2\pi ib_{i}/p}
\end{smallmatrix})$ for some $(a_{i},b_{i})\in(\ZZ_{p})^{\times}$. Similarly for each $j=1,\dots,N$ choose a complex structure on the normal bundle $\nu(\Sigma_{j})$ of $\Sigma_{j}$ so that the action of $\tau$ on $\nu(\Sigma_{j})$ can be identified with multiplication by $e^{2\pi ic_{i}/p}$ for some $c_{i}\in\ZZ_{p}^{\times}$. Finally, let $[\Sigma_{j}]^{2}$ denote the self-intersection number of $\Sigma_{j}$ calculated with respect to the trivialization of $\nu(\del\Sigma_{j})\subset\del X$ induced by the so-called \emph{canonical framing} on each component of $\del\Sigma_{j}$ (see \cite{Mon:SW} for more details). Note that the canonical framing of a knot in an \emph{integer} homology sphere agrees with the Seifert framing. We call
\[\D(X,\tau)=\{(a_{1},b_{1}),\dots,(a_{M},b_{M});(c_{1},[\Sigma_{1}]^{2}),\dots,(c_{N},[\Sigma_{N}]^{2})\}\]
the \emph{fixed-point data} of $(X,\tau)$, which is a well-defined invariant of $(X,\tau)$ modulo the equivalences $(a_{i},b_{i})\sim(b_{i},a_{i})\sim(-a_{i},-b_{i})$ and $(c_{j},[\Sigma_{j}]^{2})\sim(-c_{j},[\Sigma_{j}]^{2})$, as well as up to reordering of the $\{p_{i}\}$ and $\{\Sigma_{j}\}$. For $\ell=0,\dots,p-1$ let
\begin{align}
\label{eq:S_invariant}
\begin{split}
    \SSS(X,\tau)_{\ell}&=\tfrac{1}{p}\sigma(X)+\tfrac{2}{p}\sum_{i=1}^{M}\sum_{k=1}^{p-1}(-1)^{k(a_{i}+b_{i}+\ell)}e^{-\pi ik\ell/p}\csc(\tfrac{ka_{i}\pi}{p})\csc(\tfrac{kb_{i}\pi}{p}) \\
    &+\tfrac{2}{p}\sum_{j=1}^{N}\sum_{k=1}^{p-1}(-1)^{k(c_{i}+\ell)}[\Sigma_{j}]^{2}e^{-\pi ik\ell/p}\csc(\tfrac{kc_{i}\pi}{p})\cot(\tfrac{kc_{i}\pi}{p})\in\QQ,
\end{split}
\end{align}
and let $\vec{\SSS}(X,\tau)\in\QQ^{p}$ be the vector with entries $\SSS(X,\tau)_{0},\dots,\SSS(X,\tau)_{p-1}$. Note that 
$\sum_{k=0}^{p-1}\SSS(X,\tau)_{k}=\sigma(X)$.

In the case where $\tau$ is smooth, we have the following proposition from \cite{Mon:SW}:

\begin{proposition}
\label{prop:10_8ths_fillings}[\cite{Mon:SW}, Theorem 7.21]
Let $p$ be an odd prime, let $(Y,\frak{s},\sigma)$ be a $\ZZ_{p}$-equivariant spin rational homology sphere, and let $(X,\frak{t},\tau)$ be a smooth $\ZZ_{p}$-equivariant spin filling of $Y$ such that $b_{1}(X)=0$. Furthermore, suppose that $b_{2}^{+}(X,\tau)_{k}$ is even for all $k=1,\dots,p-1$. Then
\begin{align}
    &b_{2}^{+}(X,\tau)_{0}+\kappa_{0}\geq-\tfrac{1}{8}\SSS(X,\tau)_{0}+C, \label{eq:filling_0} \\
    &b_{2}^{+}(X)-b_{2}^{+}(X,\tau)_{0}+\kappa_{1}\geq -\tfrac{1}{8}\sigma(X)+\tfrac{1}{8}\SSS(X,\tau)_{0}, \label{eq:filling_1}
\end{align}
for all $(\kappa_{0},\kappa_{1})\in\K^{\pi}(Y,\frak{s},\sigma)$, where:
\[C=C(X,\frak{t},\tau):={
	\left\{
		\begin{array}{ll}
			0 & \mbox{if } b_{2}^{+}(X,\tau)_{0}=0, \\
			1 & \mbox{if } b_{2}^{+}(X,\tau)_{0}\text{ odd}, \\
            2 & \mbox{if } b_{2}^{+}(X,\tau)_{0}\geq 2\text{ even}.
		\end{array}
	\right.}\]
\end{proposition}

The next proposition is the key ingredient in establishing Theorem \ref{theorem:intro_no_smooth_extension}:

\begin{proposition}
\label{prop:kappa_small_fillings}
Let $p$ be an odd prime, and let $(Y,\frak{s},\sigma,X,\frak{t})$ be a quintuple, where:
\begin{itemize}
    \item $(Y,\frak{s},\sigma)$ is a $\ZZ_{p}$-equivariant spin rational homology sphere,
    \item $(X,\frak{t})$ is a spin filling of $(Y,\frak{s})$ with $b_{1}(X)=0$,
\end{itemize}
and suppose that:
\begin{enumerate}
    \item Manolescu's relative $10/8$-ths inequality is sharp for $(Y,\frak{s},X,\frak{t})$, i.e.,
    \[b_{2}^{+}(X)+\kappa(Y,\frak{s})=-\tfrac{1}{8}\sigma(X)+\left\{
		\begin{array}{ll}
			0 & \mbox{if } b_{2}^{+}(X)=0, \\
			1 & \mbox{if } b_{2}^{+}(X)\text{ odd}, \\
            2 & \mbox{if } b_{2}^{+}(X)\geq 2\text{ even},
		\end{array}
	\right.\]
    where $\kappa(Y,\frak{s})$ denotes the invariant defined in \cite{Man14}.
    \item There exist at least two distinct elements 
    \[\vec{\kappa},\vec{\kappa}'\in \K^{\pi}(Y,\frak{s},\sigma)\subset\QQ^{2}\]
    such that $|\vec{\kappa}|=|\vec{\kappa}'|=\kappa(Y,\frak{s})$.
\end{enumerate}
Then $\sigma$ cannot extend to a smooth homologically trivial $\frak{t}$-preserving $\ZZ_{p}$-action over $X$.
\end{proposition}

\begin{proof}
Let $(Y,\frak{s},\sigma,X,\frak{t})$ be as in the statement of the proposition, and suppose there were a smooth homologically trivial $\frak{t}$-preserving $\ZZ_{p}$-action $\tau:X\to X$ extending $\sigma$. By assumption we have that $b_{2}^{+}(X,\tau)_{0}=b_{2}^{+}(X)$, and so by Proposition \ref{prop:10_8ths_fillings} we have that
\begin{align}
    &b_{2}^{+}(X)+\kappa_{0}\geq-\tfrac{1}{8}\SSS(X,\tau)_{0}+C \label{eq:small_filling_0} \\
    &\kappa_{1}\geq -\tfrac{1}{8}\sigma(X)+\tfrac{1}{8}\SSS(X,\tau)_{0} \label{eq:small_filling_1}
\end{align}
for all $(\kappa_{0},\kappa_{1})\in\K^{\pi}(Y,\frak{s},\sigma)$, where
\[C=\left\{\begin{array}{ll}
			0 & \mbox{if } b_{2}^{+}(X)=0, \\
			1 & \mbox{if } b_{2}^{+}(X)\text{ odd}, \\
            2 & \mbox{if } b_{2}^{+}(X)\geq 2\text{ even}.
		\end{array}
	\right.\]
By assumption (1), we have that
\[C=b_{2}^{+}(X)+\kappa(Y,\frak{s})+\tfrac{1}{8}\sigma(X),\]
and hence (\ref{eq:small_filling_0}) is equivalent to the inequality
\begin{equation}
\label{eq:small_filling_2}
    \kappa_{0}-\kappa(Y,\frak{s})\geq\tfrac{1}{8}(\sigma(X)-\SSS(X,\tau)_{0}).
\end{equation}
Now if $\vec{\kappa}=(\kappa_{0},\kappa_{1})\in\K^{\pi}(Y,\frak{s},\sigma)$ is such that $|\vec{\kappa}|=\kappa(Y,\frak{s})$, then $\kappa_{1}=\kappa(Y,\frak{s})-\kappa_{0}$. Hence (\ref{eq:small_filling_1}) and (\ref{eq:small_filling_2}) imply that
\[(\kappa_{0},\kappa_{1})=\big(\kappa(Y,\frak{s})+\tfrac{1}{8}\sigma(X)-\tfrac{1}{8}\SSS(X,\tau)_{0},-\tfrac{1}{8}\sigma(X)+\tfrac{1}{8}\SSS(X,\tau)_{0}\big),\]
for any such pair $(\kappa_{0},\kappa_{1})$, contradicting assumption (2) in the proposition.
\end{proof}

We also have analogous results for cobordisms:

\begin{proposition}
\label{prop:10_8ths_cobordisms}[Theorem 7.20, \cite{Mon:SW}]
Let $p$ be an odd prime, let $(Y_{j},\frak{s}_{j},\sigma_{j})$, $j=0,1$ be $\ZZ_{p}$-equivariant spin rational homology spheres, and let $(X,\frak{t},\tau)$ be a $\ZZ_{p}$-equivariant spin cobordism from $Y_{0}$ to $Y_{1}$ such that $b_{1}(X)=0$. Furthermore, suppose that:
\begin{itemize}
    \item The $\tau$ fixed-point set $X^{\tau}\subset X$ is non-empty.
    \item $b_{2}^{+}(X,\tau)_{j}$ is even for all $j=1,\dots,p-1$.
\end{itemize}
and let
\[C=C(X,\frak{t},\tau;Y_{0},\frak{s}_{0},\sigma_{0}):={
	\left\{
		\begin{array}{ll}
            -1 & \mbox{if } b_{2}^{+}(X,\tau)_{0}\text{ odd, }(Y_{0},\frak{s}_{0},\sigma_{0})\text{ not }K_{G^{*}_{p}}\text{-split}, \\
			0 & \mbox{if } b_{2}^{+}(X,\tau)_{0}\text{ even, }(Y_{0},\frak{s}_{0},\sigma_{0})\text{ not }K_{G^{*}_{p}}\text{-split, or}\\
            & \mbox{if } b_{2}^{+}(X,\tau)_{0}=0,\;(Y_{0},\frak{s}_{0},\sigma_{0})\; K_{G^{*}_{p}}\text{-split}, \\
			1 & \mbox{if } b_{2}^{+}(X,\tau)_{0}\text{ odd, }(Y_{0},\frak{s}_{0},\sigma_{0})\; K_{G^{*}_{p}}\text{-split}, \\
            2 & \mbox{if } b_{2}^{+}(X,\tau)_{0}\geq 2\text{ even, }(Y_{0},\frak{s}_{0},\sigma_{0})\; K_{G^{*}_{p}}\text{-split}.
		\end{array}
	\right.}\]
Then for all $(\kappa^{0}_{0},\kappa^{1}_{0})\in\K^{\pi}(Y_{0},\frak{s}_{0},\sigma_{0})$:
\begin{enumerate}
    \item For each $(\kappa^{0}_{1},\kappa^{1}_{1})\in\K^{\pi}(Y_{1},\frak{s}_{1},\sigma_{1})$, the following implications hold:
    \begin{align*}
        &b_{2}^{+}(X,\tau)_{0}+\kappa^{0}_{1}\le-\tfrac{1}{8}\SSS(X,\tau)_{0}+\kappa^{0}_{0}+C \\
        \implies&b_{2}^{+}(X)-b_{2}^{+}(X,\tau)_{0}+\kappa_{1}^{1}\geq-\tfrac{1}{8}\sigma(X)+\tfrac{1}{8}\SSS(X,\tau)_{0}+\kappa_{0}^{1},
    \end{align*}
    and
    \begin{align*}
        &b_{2}^{+}(X)-b_{2}^{+}(X,\tau)_{0}+\kappa_{1}^{1}\le-\tfrac{1}{8}\sigma(X)+\tfrac{1}{8}\SSS(X,\tau)_{0}+\kappa_{0}^{1} \\
        \implies&b_{2}^{+}(X,\tau)_{0}+\kappa^{0}_{1}\geq-\tfrac{1}{8}\SSS(X,\tau)_{0}+\kappa^{0}_{0}+C. \\
    \end{align*}
    \item There exists $(\kappa^{0}_{1},\kappa^{1}_{1})\in\K^{\pi}(Y_{1},\frak{s}_{1},\sigma_{1})$ such that:
    \begin{align*}
        &b_{2}^{+}(X,\tau)_{0}+\kappa_{1}^{0}\geq-\tfrac{1}{8}\SSS(X,\tau)_{0}+\kappa_{0}^{0}+C, \\
        &b_{2}^{+}(X)-b_{2}^{+}(X,\tau)_{0}+\kappa_{1}^{1}\geq -\tfrac{1}{8}\sigma(X)+\tfrac{1}{8}\SSS(X,\tau)_{0}+\kappa_{0}^{1}.
    \end{align*}
\end{enumerate}
\end{proposition}

The proof of the following proposition is similar to the proof of Proposition \ref{prop:kappa_small_fillings}, and we leave it as an exercise to the reader:

\begin{proposition}
\label{prop:kappa_small_cobordisms}
Let $p$ be an odd prime, and let $(Y_{0},\frak{s}_{0},\sigma_{0},Y_{1},\frak{s}_{1},\sigma_{1},X,\frak{t})$ be an octuple, where:
\begin{itemize}
    \item $(Y_{j},\frak{s}_{j},\sigma_{j})$, $j=0,1$ are $\ZZ_{p}$-equivariant spin rational homology spheres,
    \item $(X,\frak{t})$ is a spin cobordism from $(Y_{0},\frak{s}_{0},\sigma_{0})$ to $(Y_{1},\frak{s}_{1},\sigma_{1})$ with $b_{1}(X)=0$.
\end{itemize}
Suppose that:
\begin{enumerate}
    \item Manolescu's relative $10/8$-ths inequality is sharp for $(Y_{0},\frak{s}_{0},Y_{1},\frak{s}_{1},X,\frak{t})$, i.e.,
    \[b_{2}^{+}(X)+\kappa(Y_{1},\frak{s}_{1})=-\tfrac{1}{8}\sigma(X)+\kappa(Y_{0},\frak{s}_{0})+\left\{
		\begin{array}{ll}
                -1 & \mbox{if } b_{2}^{+}(X)\text{ odd, }(Y_{0},\frak{s}_{0})\text{ not }K_{\Pin(2)}\text{-split},\\
			0 & \mbox{if } b_{2}^{+}(X)\text{ even, }(Y_{0},\frak{s}_{0})\text{ not }K_{\Pin(2)}\text{-split, or} \\
                & \mbox{if }b_{2}^{+}(X)=0,\;(Y_{0},\frak{s}_{0})\;K_{\Pin(2)}\text{-split}, \\
			1 & \mbox{if } b_{2}^{+}(X)\text{ odd, }(Y_{0},\frak{s}_{0})\;K_{\Pin(2)}\text{-split}, \\
            2 & \mbox{if } b_{2}^{+}(X)\geq 2\text{ even, }(Y_{0},\frak{s}_{0})\;K_{\Pin(2)}\text{-split}.
		\end{array}
	\right.\]
    \item One of the following holds:
    \begin{enumerate}
        \item $(Y_{0},\frak{s}_{0},\sigma_{0})$ is $K_{G^{*}_{p}}$-split, and there exist at least two distinct elements 
        \[\vec{\kappa}_{1},\vec{\kappa}'_{1}\in \K^{\pi}(Y_{1},\frak{s}_{1},\sigma_{1})\subset\QQ^{2}\]
        such that $|\vec{\kappa}_{1}|=|\vec{\kappa}'_{1}|=\kappa(Y_{1},\frak{s}_{1})$.
        \item $(Y_{1},\frak{s}_{1},\sigma_{1})$ is $K_{G^{*}_{p}}$-split, and there exist at least two distinct elements 
        \[\vec{\kappa}_{0},\vec{\kappa}'_{0}\in \K^{\pi}(Y_{0},\frak{s}_{0},\sigma_{0})\subset\QQ^{2}\]
        such that $|\vec{\kappa}_{0}|=|\vec{\kappa}'_{0}|=\kappa(Y_{0},\frak{s}_{0})$.
    \end{enumerate}
\end{enumerate}
Then the $\ZZ_{p}$-action $\sigma_{0}\amalg\sigma_{1}:-Y_{0}\amalg Y_{1}\to -Y_{0}\amalg Y_{1}$ cannot extend to a smooth homologically trivial $\frak{t}$-preserving $\ZZ_{p}$-action over $X$.
\end{proposition}

\subsection{Obstructing Smooth Extensions}
\label{subsec:obstructing_smooth_extensions}

We are now ready to prove Theorem \ref{theorem:intro_no_smooth_extension} from the introduction:

\begin{proof}[Proof of Theorem \ref{theorem:intro_no_smooth_extension}]
Note that in Cases (1)-(4):
\begin{enumerate}
    \item $N(2n)$ (respectively $N(2n)\# S^{2}\times S^{2}$) has intersection form $H$ (resp. $2H$), and boundary $-\Sigma(2,3,12n-1)$.
    \item $P(2n)$ (respectively $P(2n)\# S^{2}\times S^{2}$) has intersection form $-E_{8}\oplus H$ (resp. $-E_{8}\oplus 2H$), and boundary $-\Sigma(2,3,12n-5)$.
    \item $M(2,3,11)$ has intersection form $-2E_{8}\oplus 2H$, and boundary $\Sigma(2,3,11)$.
    \item $M(2,3,7)$ has intersection form $-E_{8}\oplus 2H$, and boundary $\Sigma(2,3,7)$.
\end{enumerate}
In \cite{Man14} it was shown that
\begin{align*}
    &\kappa(\Sigma(2,3,12n-1))=2, &\kappa(-\Sigma(2,3,12n-1))=0, \\
    &\kappa(\Sigma(2,3,12n-5))=1, &\kappa(-\Sigma(2,3,12n-5))=1.
\end{align*}
Using the above information it is straightforward to check that all of the cases (1-6) satisfy criterion (1) of Theorem \ref{prop:10_8ths_fillings}. So it suffices to check that they satisfy criterion (2). Note that for $Y=\pm\Sigma(2,3,12n-5)$ or $\pm\Sigma(2,3,12n-1)$, by Proposition \ref{prop:seifert} all of the elements $\vec{\kappa}\in\K^{\pi}(Y,\rho_{p})$ satisfy $|\vec{\kappa}|=\kappa(Y)$. Hence by Corollary \ref{cor:more_than_one_element}, $\rho_{p}$ cannot extend to a smooth homologically trivial action. The result then follows from conjugation invariance of the equivariant $\kappa$-invariants (\cite{Mon:SW}, Theorem 1.3) and the fact that any $\ZZ_{p}$-action on $Y$ is conjugate to $\rho_{p}$.
\end{proof}

We have a similar theorem which obstructs the existence of smooth homologically $\ZZ_{p}$-actions over cobordisms obtained as complements of embeddings $M(p,q,r)\hookrightarrow M(p',q',r')$:

\begin{theorem}
\label{theorem:no_smooth_extension_cobordisms}
Let $(p,X)$ be any of the following pairs, where $p$ is an odd prime, and $X$ is a smooth spin 4-manifold with two connected boundary components:
\begin{enumerate}
    \item $p\geq 3$ and $X$ is any smooth 4-manifold homeomorphic to the complement of:
    \begin{enumerate}
        \item $M(2,3,5)\subset M(2,3,7)$,
        \item $M(2,3,5)\subset M(2,3,11)$, $p\neq 5$,
        \item $M(2,3,7)\subset M(2,3,13)$,
        \item $M(2,3,7)\subset M(2,3,17)$,
        \item $M(2,3,11)\subset M(2,3,13)$, $p\neq 5$,
        \item $M(2,3,11)\subset M(2,3,17)$, $p\neq 5$.
    \end{enumerate}
    \item $p=3$ and $X$ is any smooth 4-manifold homeomorphic to the complement of:
    \begin{enumerate}
        \item $M(2,3,12n-7)\subset M(2,3,12n-5)$, $n\geq 1$,
        \item $M(2,3,12n-7)\subset M(2,3,12n-1)$, $n\geq 1$,
        \item $M(2,3,12n-5)\subset M(2,3,12n+1)$, $n\geq 1$,
        \item $M(2,3,12n-5)\subset M(2,3,12n+5)$, $n\geq 1$,
        \item $M(2,3,12n-1)\subset M(2,3,12n+1)$, $n\geq 1$,
        \item $M(2,3,12n-1)\subset M(2,3,12n+5)$, $n\geq 1$,
        \item $M(2,3,12n+1)\subset M(2,3,12n+7)$, $n\geq 1$,
        \item $M(2,3,12n+1)\subset M(2,3,12n+11)$, $n\geq 1$.
    \end{enumerate}
    \item $p=7$ and $X$ is any smooth 4-manifold homeomorphic to the complement of:
    \begin{enumerate}
        \item $M(2,3,13)\subset M(2,3,19)$,
        \item $M(2,3,17)\subset M(2,3,19)$,
        \item $M(2,3,19)\subset M(2,3,25)$,
        \item $M(2,3,19)\subset M(2,3,29)$.
    \end{enumerate}
    \item $p=11$ and $X$ is any smooth 4-manifold homeomorphic to the complement of:
    \begin{enumerate}
        \item $M(2,3,13)\subset M(2,3,23)$,
        \item $M(2,3,17)\subset M(2,3,23)$,
        \item $M(2,3,23)\subset M(2,3,25)$,
        \item $M(2,3,23)\subset M(2,3,29)$.
    \end{enumerate}
\end{enumerate}
Then no effective smooth $\ZZ_{p}$-action on $\del X$ can extend to a smooth homologically trivial $\ZZ_{p}$-action over $X$.
\end{theorem}

\begin{proof}
Note that the intersection forms of $M(2,3,6n\pm 1)$ are given by:
\begin{align*}
    &Q_{M(2,3,12n+5)}=-(2n+1)E_{8}\oplus 4nH, & &Q_{M(2,3,12n-5)}=-(2n-1)E_{8}\oplus (4n-2)H, \\
    &Q_{M(2,3,12n-1)}=-2nE_{8}\oplus (4n-2)H, & &Q_{M(2,3,12n+1)}=-2nE_{8}\oplus 4nH.
\end{align*}
The result then follows from the calculations in \cite{Man14}, Propositions \ref{prop:seifert_K_split}, \ref{prop:seifert}, \ref{prop:kappa_small_cobordisms}, and Corollary \ref{cor:more_than_one_element}.
\end{proof}

\section{Nonsmoothable \texorpdfstring{$\ZZ_{p}$}{Z/p}-Actions}
\label{sec:nonsmoothable_actions}

\subsection{\texorpdfstring{$\ZZ[\ZZ_{p}]$}{Z[Z/p]} h-Cobordisms}
\label{subsec:Z[Z_p]_h_cobordism}

In order to construct the locally linear actions from Theorem \ref{theorem:intro_locally_linear_extension} we use the methods outlined in \cite{Edmonds87}, \cite{KL93} (see also \cite{AH16}, Section 5).

Indeed, one can start with a standard locally linear homologically trivial pseudofree $\ZZ_{p}$-action on a closed simply-connected topological 4-manifold with the desired intersection form (i.e., either $H$ or $-E_{8}\oplus H$ corresponding to (1) and (2) in the theorem). The construction then reduces to showing there exists an equivariant homology cobordism from $(Y,\rho_{p})$ to $S^{3}$ equipped with a generalized lens space action, where $Y=\Sigma(2,3,12pn-1)$ or $\Sigma(2,3,12pn-6p+1)$. In particular, we will need the following definitions:

\begin{definition}
A \emph{$\ZZ[\ZZ_{p}]$ homology lens space} is a triple $(Q,Y,\sigma)$ where $Q$ is a 3-dimensional homology lens space, and $Y$ is an integer homology sphere equipped with a free $\ZZ_{p}$-action $\sigma:Y\to Y$ and an identification $Y/\sigma=Q$. A \emph{$\ZZ[\ZZ_{p}]$ h-cobordism} between two homology lens spaces $(Q,Y,\sigma)$ and $(Q',Y',\sigma')$ is a triple $(Z,W,\tau)$ where $Z$ is a topological cobordism from $Q$ to to $Q'$ with $\pi_{1}(Z)\cong\ZZ_{p}$, and $W$ is an integer homology cobordism from $Y$ to $Y'$, equipped with a free $\ZZ_{p}$-action $\tau:W\to W$ and an identification $W/\tau=Z$ which is compatible with the corresponding identifications on the two boundary components.
\end{definition}

Equivalently, a $\ZZ[\ZZ_{p}]$ homology lens space is a homology lens space $Q$ equipped with an injection $\phi:\ZZ_{p}\to \pi_{1}(Q)$, and a $\ZZ[\ZZ_{p}]$ h-cobordism between two homology lens spaces $(Q,\phi)$ and $(Q',\phi')$ consists of a cobordism $W$ from $Q$ to $Q'$ along with an isomorphism $\psi:\ZZ_{p}\xrightarrow{\cong}\pi_{1}(W)$ such that $\psi=\iota_{*}\circ\phi=\iota'_{*}\circ\phi$, where $\iota_{*}:\pi_{1}(Q)\to\pi_{1}(W)$ and $\iota'_{*}:\pi_{1}(Q')\to\pi_{1}(W)$ denote the maps induced by inclusion.

We will often drop the additional notation and simply refer to such a triple $(Q,Y,\sigma)$ via the underlying homology lens space $Q$.

\begin{example}
Let $p\geq 2$, let $a,b\in\ZZ_{p}^{\times}$, and let $\sigma_{(p;a,b)}:S^{3}\to S^{3}$ be the $\ZZ_{p}$-action given by the restriction of the action $(w,z)\mapsto(e^{2\pi ia/p}w,e^{2\pi ib/p}z)$ on $\CC^{2}$. We refer to the quotient $L(p;a,b):=S^{3}/\sigma_{(p;a,b)}$ as a \emph{generalized lens space}.
\end{example}

\begin{example}
Given a Seifert-fibered homology sphere $Y=\Sigma(\alpha_{1},\dots,\alpha_{n})$ and an integer $p\geq 2$ such that $(p,\alpha_{i})=1$ for all $i=1,\dots,n$, we denote by $Q(p;\alpha_{1},\dots,\alpha_{n})=Y/\rho_{p}$ the $\ZZ[\ZZ_{p}]$ homology lens space corresponding to the standard $\ZZ_{p}$-action, which we call a \emph{Seifert-fibered homology lens space}.
\end{example}

We will need to consider the following invariant, introduced \cite{APS2}:

\begin{definition}
Let $(Q,Y,\sigma)$ be a $\ZZ[\ZZ_{p}]$ homology lens space. The \emph{$\alpha$-invariant} of $Q$ is defined to be $\alpha(Q):=-\eta^{(1,p)}_{\sign}(Y)$, where $\eta^{(q,p)}_{\sign}(Y)$ denotes the equivariant eta-invariant of the odd signature operator of $Y$ at $\sigma^{q}$, $1\le q\le p-1$. 
\end{definition}

We have the following criterion which determines when the quotient of a Brieskorn homology sphere by the standard $\ZZ_{p}$-action is $\ZZ[\ZZ_{p}]$ h-cobordant to a generalized lens space:

\begin{proposition}[\cite{Edmonds87}, \cite{KL93}]
\label{prop:h_cobordism}
Let $p\geq 2$, let $Q(p;\alpha_{1},\alpha_{2},\alpha_{3})$ be the quotient of a Brieskorn homology sphere by a $\ZZ_{p}$-action, and let $L(p;a,b)$ a generalized lens space. Then there exists a $\ZZ[\ZZ_{p}]$ h-cobordism from $Q(p;\alpha_{1},\alpha_{2},\alpha_{3})$ to $L(p;a,b)$ if the following conditions are satisfied:
\begin{enumerate}
    \item $\alpha_{1}\alpha_{2}\alpha_{3}\equiv ab\pmod{p}$.
    \item The unordered triples $\{\alpha_{1},\alpha_{2},\alpha_{3}\}$ and $\{a,b,1\}$ are congruent modulo $p$ up to sign.
    \item $\alpha(Q(p;\alpha_{1},\alpha_{2},\alpha_{3}))=\alpha(L(p;a,b))$.
\end{enumerate}
\end{proposition}

\begin{example}
Using the fact that the $\ZZ_{p}$-action $\sigma_{(p;a,b)}:S^{3}\to S^{3}$ extends over $B^{4}$ with a single fixed point of type $(a,b)$, the $G$-signature theorem \cite{Don78} implies that $\alpha(L(p;a,b))=\cot(\frac{a\pi}{p})\cot(\frac{b\pi}{p})$.
\end{example}

In order to compute the $\alpha$-invariants of Seifert-fibered homology lens spaces, one can refer to the computations of $\rho$-invariants given in \cite{Auckly91}, \cite{KL93}. Instead we opt to use the orbifold version of the $G$-signature theorem \cite{Liang79} applied to the canonical orbifold disk bundle bounded by $Y$, as the resulting formula is more amenable to computation:

\begin{proposition}
\label{prop:equivariant_signature_formula}
Let $Y=\Sigma(\alpha_{1},\dots,\alpha_{n})$ be a Seifert fibered homology sphere, and for each $i=1,\dots,n$ let $p_{i}$ be any integer which satisfies $p_{i}\equiv -\frac{\alpha}{\alpha_{i}}\pmod{\alpha_{i}}$, where $\alpha=\alpha_{1}\cdots\alpha_{n}$. Furthermore, let $r\geq 2$ be such that $(r,\alpha_{i})=1$ for all $i=1,\dots,n$, and let $1\le q\le r-1$. Then
\[\eta_{sign}^{(q,r)}(Y)=1-\tfrac{1}{\alpha}\csc^{2}(\tfrac{q\pi}{r})-\sum_{i=1}^{n}c(p_{i},\alpha_{i};\tfrac{p_{i}q}{r},\tfrac{\alpha_{i}q}{r}),\]
where $c(b,a;x,y)$ denotes the Dedekind-Dieter cotangent sum \cite{Dieter84}:
\begin{align*}
    &c(b,a;x,y):=\tfrac{1}{a}\sum_{k=0}^{a-1}c(\tfrac{b}{a}(k+y)-x)c(\tfrac{k+y}{a}), &
    &c(z):=\twopartdef{\cot(\pi z)}{z\not\in\ZZ,}{0}{z\in\ZZ.}
\end{align*}
In particular:
\[\alpha(Q(r;\alpha_{1},\dots,\alpha_{n}))=\tfrac{1}{\alpha}\csc^{2}(\tfrac{\pi}{r})-1+\sum_{i=1}^{n}c(p_{i},\alpha_{i};\tfrac{p_{i}}{r},\tfrac{\alpha_{i}}{r}).\]
\end{proposition}

\begin{remark}
One could also use methods similar to those in \cite{Anv:rho} to obtain a formula in the setting where $r$ is not necessarily relatively prime to one of the $\alpha_{i}$.
\end{remark}

\begin{proof}
Fix a presentation of $Y$ as the unit sphere bundle of an orbifold complex line bundle $L\to\Sigma$, where $\Sigma=S^{2}(\alpha_{1},\dots,\alpha_{n})$ is the orbifold surface with underlying surface $S^{2}$ and $n$ singular points $x_{1},\dots,x_{n}$ of isotropy orders $\alpha_{1},\dots,\alpha_{n}$, and orbifold charts $U_{i}=\wt{U}_{i}/\ZZ_{\alpha_{i}}$ around the $x_{i}$. Then the disk bundle $X=D(L)$ is an oriented 4-orbifold bounded by $Y$, with isolated singular points $y_{1},\dots,y_{n}$ corresponding to to the images of $x_{1},\dots,x_{n}\in\Sigma$ under the embedding $\Sigma\hookrightarrow D(L)$ as the zero-section. Note that with our orientation conventions, $b_{2}^{+}(X)=0$, $b_{2}^{-}(X)=-1$.

The action of $\rho_{r}$ on $Y$ has a canonical extension $\tau_{r}$ over $X$, given by rotation in the $D^{2}$-fibers. By our assumption that $(r,\alpha_{i})=1$ for all $i$, the fixed point set of of $\tau^{q}_{r}$ for all $q=1,\dots,r-1$ is precisely the zero-section $\Sigma\subset X$. For each $i=1,\dots,n$, let $V_{i}=\wt{V}_{i}/\ZZ_{\alpha_{i}}$ be an orbifold chart containing $y_{i}$ , with $\wt{y}_{i}\in \wt{V}_{i}$ denoting the lift of $y_{i}$, fix a generator $g_{i}:\wt{V}_{i}\to\wt{V}_{i}$ of the $\ZZ_{\alpha_{i}}$-action on $\wt{V}_{i}$, and fix a lift $\wt{\tau}_{r,i}:\wt{V}_{i}\to\wt{V}_{i}$ of $\tau_{r}|_{V_{i}}$. By the orbifold G-Signature theorem \cite{Liang79} we have that
\[\eta_{\sign}^{(q,r)}(Y)=-\sign^{(q,r)}(X)+[\Sigma]^{2}\csc^{2}(\tfrac{\psi}{2})-\sum_{i=1}^{n}\tfrac{1}{\alpha_{i}}\sum_{k=1}^{\alpha_{i}-1}\cot(\tfrac{a_{i,k}}{2})\cot(\tfrac{b_{i,k}}{2}),\]
where:
\begin{enumerate}
    \item $\sign^{(q,r)}(X)$ denotes the equivariant signature of $X$ at $\tau_{r}^{q}$.
    \item $[\Sigma]^{2}$ denotes the orbifold Euler class of $N(\Sigma)\subset X$.
    \item $\psi$ denotes the angle of rotation of $\tau_{r}^{q}$ acting on $N(\Sigma)$ with respect to some identification of $N(\Sigma)$ as a orbifold complex line bundle.
    \item For each $i=1,\dots,n$, $(a_{i,k},b_{i,k})$ denotes the pair of angles for which $\wt{\tau}^{q}_{r,i}\circ g_{i}^{k}$ acts on $N(\wt{y}_{i})\subset\wt{V}_{i}$ with respect to some identification of $N(\wt{y}_{i})$ with $\CC^{2}$.
\end{enumerate}
Since $\tau^{q}_{r}$ is induced from an $S^{1}$-action $\tau:S^{1}\times X\to X$, it must act trivially on homology. Hence $\sign^{(q,r)}(X)=\sigma(X)=-1$. The orbifold self-intersection number $[\Sigma]^{2}$ is precisely the degree of $L$, which is given by $-\frac{1}{\alpha}$. Under the identification $N(\Sigma)\cong X=D(L)$, we see that $\psi$ is equal to $\frac{2\pi q}{r}$. Finally, we can identify $N(\wt{y}_{i})$ with the total space of the $\ZZ_{\alpha_{i}}$-equivariant chart defining $X=D(L)$ over $\wt{U}_{i}$, where the first and second coordinates are given by the surface and bundle directions, respectively. Under this identification, we see that $(a_{i,k},b_{i,k})=(\frac{2kp_{i}\pi}{\alpha_{i}},\frac{2k\pi}{\alpha_{i}}+\frac{2q\pi}{r})$. Hence
\[\eta_{sign}^{(q,r)}(Y)=1-\tfrac{1}{\alpha}\csc^{2}(\tfrac{q\pi}{r})-\sum_{i=1}^{n}\tfrac{1}{\alpha_{i}}\sum_{k=1}^{\alpha_{i}-1}\cot(\tfrac{kp_{i}\pi}{\alpha_{i}})\cot(\tfrac{k\pi}{\alpha_{i}}+\tfrac{q\pi}{r}),\]
from which the result follows.
\end{proof}

The sums $c(b,a;x,y)$ satisfy a reciprocity formula which allows for effective computation via the Euclidean algorithm. In particular, we have the following lemma obtained as a special case of (\cite{Dieter84}, Theorem 3.2):

\begin{lemma}
\label{lemma:dedekind_dieter_reciprocity}
Let $r\geq 2$, and let $a>b>0$ be relatively prime integers. Let $a_{0},\dots,a_{n}\in\ZZ_{+}$ be defined inductively by $a_{0}=a$, $a_{1}=b$, and $a_{j+1}=a_{j-1}-q_{j}a_{j}$ for some $q_{j}\in\ZZ_{+}$, terminating at $a_{n}=1$. Furthermore, let $s_{0},\dots,s_{n}\in\ZZ_{+}$ be defined inductively by $s_{0}=0$, $s_{1}=1$, and $s_{j+1}=s_{j}q_{j}+s_{j-1}$. Finally let
\begin{align*}
    &c_{2}(z):=\twopartdef{\csc^{2}(\pi z)}{z\not\in\ZZ,}{\frac{1}{3}}{z\in\ZZ,} & &\delta(z,r):=\twopartdef{r}{z\in\ZZ\text{ and }r|z,}{0}{\text{ otherwise.}}
\end{align*}
Then
\begin{align*}
    c(b,a;\tfrac{b}{r},\tfrac{a}{r})&=\sum_{j=1}^{n}(-1)^{j}\Big[c(\tfrac{a_{j-1}}{r})c(\tfrac{a_{j}}{r})-\tfrac{a_{j-1}}{a_{j}}\tfrac{\delta(a_{j},r)}{r}c_{2}(\tfrac{a_{j-1}}{r})-\tfrac{a_{j}}{a_{j-1}}\tfrac{\delta(a_{j-1},r)}{r}c_{2}(\tfrac{a_{j}}{r})\Big] \\
    &\qquad\qquad+(-1)^{n-1}\tfrac{s_{n}}{a}c_{2}(\tfrac{1}{r})+\twopartdef{-1}{n\text{ odd},}{0}{n\text{ even}.}
\end{align*}
\end{lemma}

\begin{proposition}
\label{prop:top_equivariant_cobordism}
Let $p\geq 5$ be prime. Then for all $n\geq 1$, there exists a topological $\ZZ[\ZZ_{p}]$ h-cobordism from $Q(p;2,3,12pn-1)$ to $L(p;-2,3)$; similarly there exists a topological $\ZZ[\ZZ_{p}]$ h-cobordism from $Q(p;2,3,12pn-6p+1)$ to $L(p;2,3)$.
\end{proposition}

\begin{proof}
We prove the first claim; the second one is proved similarly. Since 
\begin{align*}
    &72pn-6\equiv -6\equiv (-2)(3)\pmod{p}, \\
    &\{2,3,12pn-1\}\equiv\{-2,3,-1\}\pmod{p},
\end{align*}
by Proposition \ref{prop:h_cobordism} it suffices to check that $\alpha(Q(p;2,3,12pn-1))=\alpha(L(p;-2,3))$. Note that
\begin{align*}
    &-(3)(12pn-1)\equiv 1\pmod{2}, & &-(2)(12pn-1)\equiv 2\pmod{3}, \\
    &-(2)(3)\equiv 12pn-7\pmod{12pn-1}, & &
\end{align*}
and so by Proposition \ref{prop:equivariant_signature_formula} we have that
\begin{align*}
    &\alpha(Q(p;2,3,12pn-11)) \\
    =&\tfrac{1}{72pn-6}c_{2}(\tfrac{1}{p})-1+c(1,2;\tfrac{1}{p},\tfrac{2}{p})+c(2,3;\tfrac{2}{p},\tfrac{3}{p})+c(12pn-7,12pn-1;\tfrac{12pn-7}{p},\tfrac{12pn-1}{p}).
\end{align*}
Using Lemma \ref{lemma:dedekind_dieter_reciprocity}, we can deduce that
\begin{align*}
    &c(1,2;\tfrac{1}{p},\tfrac{2}{p})=-c(\tfrac{1}{p})c(\tfrac{2}{p})+\tfrac{1}{2}c_{2}(\tfrac{1}{p})-1, & &c(2,3;\tfrac{2}{p},\tfrac{3}{p})=c(\tfrac{1}{p})c(\tfrac{2}{p})-c(\tfrac{2}{p})c(\tfrac{3}{p})-\tfrac{1}{3}c_{2}(\tfrac{1}{p}),
\end{align*}
and
\begin{align*}
    &c(12pn-7,1pn-1;\tfrac{12pn-7}{p},\tfrac{12pn-1}{p}) \\
    &\qquad\qquad=c(\tfrac{1}{p})c(\tfrac{5}{p})-c(\tfrac{1}{p})c(\tfrac{7}{p})-c(\tfrac{5}{p})c(\tfrac{6}{p})-c(\tfrac{6}{p})c(\tfrac{7}{p})-\tfrac{2pn}{12pn-1}c_{2}(\tfrac{1}{p}) \\
    &\qquad\qquad\qquad\qquad-\tfrac{\delta(5,p)}{p}\Big(\tfrac{1}{5}c_{2}(\tfrac{1}{p})-\tfrac{6}{5}c_{2}(\tfrac{6}{p})\Big)-\tfrac{\delta(7,p)}{p}\Big(\tfrac{12pn-1}{12pn-7}c_{2}(\tfrac{1}{p})+\tfrac{6}{12pn-7}c_{2}(\tfrac{6}{p})\Big).
\end{align*}
Putting this all together and simplifying, we obtain
\begin{align*}
    \alpha(Q(p;2,3,12pn-11))&=-2+c(\tfrac{1}{p})c(\tfrac{5}{p})-c(\tfrac{1}{p})c(\tfrac{7}{p})-c(\tfrac{2}{p})c(\tfrac{3}{p})-c(\tfrac{5}{p})c(\tfrac{6}{p})-c(\tfrac{6}{p})c(\tfrac{7}{p}) \\
    &\qquad-\tfrac{\delta(5,p)}{p}\Big(\tfrac{1}{5}c_{2}(\tfrac{1}{p})-\tfrac{6}{5}c_{2}(\tfrac{6}{p})\Big)-\tfrac{\delta(7,p)}{p}\Big(\tfrac{12pn-1}{12pn-7}c_{2}(\tfrac{1}{p})+\tfrac{6}{12pn-7}c_{2}(\tfrac{6}{p})\Big).
\end{align*}
In particular for $p\geq 11$, we have that
\begin{equation}
\label{eq:alpha_p>=11}
\alpha(Q(p;2,3,12pn-11))=-2+c(\tfrac{1}{p})c(\tfrac{5}{p})-c(\tfrac{1}{p})c(\tfrac{7}{p})-c(\tfrac{2}{p})c(\tfrac{3}{p})-c(\tfrac{5}{p})c(\tfrac{6}{p})-c(\tfrac{6}{p})c(\tfrac{7}{p}).
\end{equation}
Using the identity
\begin{align*}
    &\cot(x+y)\cot(x)+\cot(x+y)\cot(y)=\cot(x)\cot(y)-1, & &x,y\not\in\pi\ZZ,
\end{align*}
we see that
\begin{align*}
    &c(\tfrac{1}{p})c(\tfrac{5}{p})-c(\tfrac{5}{p})c(\tfrac{6}{p})=c(\tfrac{1}{p})c(\tfrac{6}{p})+1, & &c(\tfrac{1}{p})c(\tfrac{7}{p})+c(\tfrac{6}{p})c(\tfrac{7}{p})=c(\tfrac{1}{p})c(\tfrac{6}{p})-1.
\end{align*}
Plugging these into (\ref{eq:alpha_p>=11}), we see that
\[\alpha(Q(p;2,3,12pn-11))=-c(\tfrac{2}{p})c(\tfrac{3}{p})=\alpha(L(p;-2,3)),\]
which was what was to be shown. The cases $p=5,7$ follow by a similar calculation which we leave to the reader.
\end{proof}

\subsection{Constructing Locally Linear Actions}
\label{subsec:constructing_locally_linear}

For $(a,b)\in(\ZZ_{p}^{\times})^{2}$ let $\tau_{(p;a,b)}$ be the following $\ZZ_{p}$-action on $S^{2}\times S^{2}$, which was described in the introduction:
\begin{align*}
    \tau_{(p;a,b)}:S^{2}\times S^{2}&\to S^{2}\times S^{2} \\
    (z_{1},r_{1},z_{2},r_{2})&\mapsto (e^{2\pi i a/p}z_{1},r_{1},e^{2\pi i b/p}z_{2},r_{2}).
\end{align*}
Here we identify $S^{2}=\{(z,r)\in\CC\times\RR\;|\;|z|^{2}+r^{2}=1\}$. As mentioned in the introduction, we have that
\[\D(S^{2}\times S^{2},\tau_{(p;a,b)})=\{(a,b),(a,b),(-a,b),(-a,b)\}.\]
The following lemma will help us build the non-smoothable $\ZZ_{p}$-actions on $P(2pn-p+1)$:

\begin{lemma}
\label{lemma:E_8_S^2_S^2}
Let $-E_{8}$ denote the simply-connected topological 4-manifold with intersection form $-E_{8}$. Then for $p\geq 5$ prime there exists a locally linear homologically trivial pseudofree $\ZZ_{p}$-action $\tau$ on $X=-E_{8}\# S^{2}\times S^{2}$ with $12$ fixed points and corresponding fixed point data as follows:
\begin{itemize}
    \item $p=5$:
    \begin{align*}
        \D(X,\tau)=&\{(1,1),(1,1),(1,1),(1,1),(1,2),(-1,2), \\
        &\qquad\qquad(-2,3),(-2,3),(-2,3),(-2,3),(-2,3),(2,3)\}.
    \end{align*}
    \item $p=7$: 
    \begin{align*}
        \D(X,\tau)=&\{(1,1),(1,1),(1,2),(-1,2),(-2,3),(-2,3), \\
        &\qquad\qquad(-2,3),(-2,3),(-2,3),(2,3),(3,3),(3,3)\}.
    \end{align*}
    \item $p=11$:
    \begin{align*}
        \D(X,\tau)=&\{(1,1),(1,2),(-1,2),(-2,3),(-2,3),(2,3), \\
        &\qquad\qquad(-3,4),(-4,5),(-5,6),(-6,7),(-7,8),(-8,9)\}.
    \end{align*}
    \item $p\geq 13$:
    \begin{align*}
        \D(X,\tau)=&\{(1,2),(-1,10),(-2,3),(-2,3),(2,3),(-3,4), \\
        &\qquad\qquad(-4,5),(-5,6),(-6,7),(-7,8),(-8,9),(-9,10)\}.
    \end{align*}
\end{itemize}
\end{lemma}

\begin{proof}
We will first make use of Kiyono's \cite{Kiyono11} techniques to first construct a locally linear $\ZZ_{p}$-action on the topological manifold $-E_{8}$. For each triple $\alpha=(a,b,c)\in\ZZ_{p}^{3}$ such that no two of $a,b,c$ are congruent to each other modulo $p$, let $\CC P^{2}_{\alpha}$ denote the smooth $\ZZ_{p}$-equivariant manifold
\[\CC P^{2}_{\alpha}=(\CC_{a}\oplus\CC_{b}\oplus\CC_{c}\setminus\{0,0,0\})/\CC^{*}\]
diffeomorphic to $\CC P^{2}$, endowed with a smooth $\ZZ_{p}$-action with precisely 3 fixed points and corresponding fixed-point data $\{(a-c,b-c),(b-a,c-a),(a-b,c-b)\}$. We denote by $\ol{\CC P^{2}_{\alpha}}$ the same manifold with the orientation reversed, which has corresponding fixed-point data $\{(a-c,c-b),(b-a,a-c),(a-b,b-c)\}$.

For each $i=1,\dots,8$, let $R(i)$ denote the remainder of $i$ divided by $p-3$, and let $\alpha_{i}=(-1,R(i),R(i)+1)$. By (\cite{Kiyono11}, Example 2.4), there are precisely $7$ cancelling pairs of the form $\{(a,b),(-a,b)\}$ contained in the fixed point data of the disjoint union $\amalg_{i=1}^{8}\ol{\CC P_{\alpha_{i}}^{2}}$. Hence by (\cite{Kiyono11}, Proposition 2.5) there exists a homologically trivial pseudofree $\ZZ_{p}$-action $\tau':-E_{8}\to -E_{8}$ with fixed-point data consisting of the $10$ fixed points of $\amalg_{i=1}^{8}\ol{\CC P_{\alpha_{i}}^{2}}$ not appearing in the $7$ cancelling pairs. One can check that the fixed point data of $\tau'$ is given as follows:
\begin{itemize}
    \item $p=5$:
    \begin{align*}
        \D(E_{8},\tau')=&\{(1,1),(1,1),(1,1),(1,1),(1,2),(-1,2), \\
        &\qquad\qquad\qquad\qquad(-2,3),(-2,3),(-2,3),(-2,3)\}.
    \end{align*}
    \item $p=7$: 
    \begin{align*}
        \D(E_{8},\tau')=&\{(1,1),(1,1),(1,2),(-1,2),(-2,3),(-2,3), \\
        &\qquad\qquad\qquad\qquad(-2,3),(-2,3),(3,3),(3,3)\}.
    \end{align*}
    \item $p=11$:
    \begin{align*}
        \D(E_{8},\tau')=&\{(1,1),(1,2),(-1,2),(-2,3),(-3,4),(-4,5), \\
        &\qquad\qquad\qquad\qquad(-5,6),(-6,7),(-7,8),(-8,9)\}.
    \end{align*}
    \item $p\geq 13$:
    \begin{align*}
        \D(E_{8},\tau')=&\{(1,2),(-1,10),(-2,3),(-3,4),(-4,5),(-5,6) \\
        &\qquad\qquad\qquad\qquad(-6,7),(-7,8),(-8,9),(-9,10)\}.
    \end{align*}
\end{itemize}

Note that in all cases $(E_{8},\tau')$ contains at least one fixed point of type $(-2,3)$. By performing a homologically trivial stabilization at any such point with a copy of $(S^{2}\times S^{2},\tau_{(p;-2,3)})$, we obtain the desired action $\tau$ on $X=-E_{8}\#S^{2}\times S^{2}$.
\end{proof}

We now prove Theorem \ref{theorem:intro_locally_linear_extension}:

\begin{proof}[Proof of Theorem \ref{theorem:intro_locally_linear_extension}]
Let $p\geq 5$ be an odd prime, $n\geq 1$. We first construct the $\ZZ_{p}$-action on $N(2pn)$. 

Let $\tau_{(p;-2,3)}:S^{2}\times S^{2}\to S^{2}\times S^{2}$ be the $\ZZ_{p}$-action described above, let $B$ be a regular $\ZZ_{p}$-equivariant neighborhood of one of the fixed points of type $(-2,3)$, and let $X'=S^{2}\times S^{2}\setminus B$, $\tau'_{(p;-2,3)}:=\tau_{(p;-2,3)}|_{X'}$. With the induced orientation, the restriction of $\tau_{(p;-2,3)}$ to $\del X'\approx S^{3}$ can be identified with $\sigma_{(p;2,3)}:S^{3}\to S^{3}$ as defined above (note the change in sign due to orientation reversal).

By Proposition \ref{prop:top_equivariant_cobordism} there exists a topological $\ZZ_{p}$-equivariant homology cobordism $(W,\tau_{W})$ from $(S^{3},\sigma_{(p;-2,3)})$ to $(\Sigma(2,3,12pn-1),\rho_{p})$. Reversing the orientation, we see that $(-W,\tau_{W})$ provides $\ZZ_{p}$-equivariant homology cobordism from $(S^{3},\sigma_{(p;2,3)})$ to $(-\Sigma(2,3,12pn-1),\rho_{p})$. Let $X:=X'\cup(-W)$, along with the induced topologically locally linear $\ZZ_{p}$-action $\tau=\tau'\cup\tau_{W}$.

By (\cite{Boyer86}, \cite{Boyer93}, \cite{Stong93}), the set of simply-connected compact topological 4-manifolds with fixed boundary an integer homology sphere $Y$ and even intersection form are determined up to homeomorphism by their intersection forms (see also \cite{KonnoTaniguchi}, Theorem 4.4). Therefore since $X$ has intersection form $H$ and boundary $\del X=-\Sigma(2,3,12pn-1)$, $X$ must be homeomorphic to $N(2pn)$.

We similarly use Lemma \ref{lemma:E_8_S^2_S^2} to construct our desired $\ZZ_{p}$-action on $P(2pn-p+1)$. Indeed, let $\tau_{p}:-E_{8}\#S^{2}\times S^{2}\to -E_{8}\#S^{2}\times S^{2}$ be the $\ZZ_{p}$-action constructed in the lemma, and let
$X'$ denote $-E_{8}\times S^{2}\times S^{2}$ with a small equivariant neighborhood of a point of type $(2,3)$ removed, with restricted action $\tau'_{p}:=\tau_{p}|_{X'}$. Let $(W,\tau_{W})$ denote topological $\ZZ_{p}$-equivariant homology cobordism from $(S^{3},\sigma_{(p;2,3)})$ to $(\Sigma(2,3,12pn-6p+1),\rho_{p})$ guaranteed by Proposition \ref{prop:top_equivariant_cobordism}. By letting $X=X'\cup(-W)$, we see that $X$ is a simply-connected compact topological 4-manifold with intersection form $-E_{8}\oplus H$ and boundary $\del X=-\Sigma(2,3,12pn-6p+1)$, and comes equipped with the homologically trivial pseudofree $\ZZ_{p}$-action $\tau=\tau'_{p}\cup\tau_{W}$. The same argument as above implies that $X$ must be homeomorphic to $P(2pn-p+1)$.
\end{proof}

\section{Stabilizations}
\label{sec:stabilizations}

We begin this section by proving Theorem \ref{theorem:intro_wall} from the introduction:

\begin{proof}[Proof of Theorem \ref{theorem:intro_wall}]
Let $(X,\tau)$ be as in the statement of the theorem, and let $X^{\tau}\subset X$ denote the $\tau$-fixed point set. By the claim in the proof of (\cite{KS18}, Theorem 1.1), there exists a closed $G$-invariant neighborhood $A\subset X$ of $X^{\tau}$ such that:
\begin{enumerate}
    \item $A$ is a topological 4-manifold with boundary.
    \item There exists a smooth structure on $A$ such that the restricted $\ZZ_{p}$-action $\tau_{A}:=\tau|_{A}$ is smooth.
\end{enumerate}
Let $X'=\ol{X\setminus A}$ denote the closure of the complement, along with the free topological $\ZZ_{p}$-action $\tau'=\tau|_{X'}$. The Equivariant Collar Neighborhood Theorem implies that the smooth structure on $\del X'\approx -\del A$ induced by the smooth structure on $A$ extends to an equivariant open collar neighborhood $C\approx \del X'\times [0,1)$ of $\del X'\subset X'$.

It therefore suffices to show that we can extend the equivariant smooth structure on $C$ to an equivariant smooth structure all of $(X',\tau')$ after sufficiently many free stabilizations with $\#^{p}S^{2}\times S^{2}$. But as $\tau'$ acts freely on $X'$, this is equivalent to assertion that there exists a smooth structure on the quotient $X'/\ZZ_{p}$ after sufficiently many stabilizations with $S^{2}\times S^{2}$. As $X'$ is assumed to have vanishing Kirby-Siebenmann invariant, this follows from Freedman--Quinn's Sum-Stable Smoothing Theorem (\cite{FQ90}, p.125). (See also \cite{FNOP19}, Theorem 8.6). 
\end{proof}

In order to prove Theorem \ref{theorem:intro_stabilizations}, we will need to assemble a few more results:

\begin{lemma}
\label{lemma:stabilization_S_invariant}
Let $(X,\frak{t})$ be a compact spin 4-manifold equipped with a $\frak{t}$-preserving locally linear pseudofree $\ZZ_{p}$-action $\tau:X\to X$. Then $\vec{\SSS}(X,\tau)\in\QQ^{p}$ is invariant under the operations of homologically trivial stabilization by $S^{2}\times S^{2}$ and free stabilization by $\#^{p}S^{2}\times S^{2}$.
\end{lemma}

\begin{proof}
For free stabilizations this follows immediately from the fact  that the signature and fixed-point data is unchanged. For homologically trivial stabilizations, the signature is unchanged and the effect on fixed-point data amounts to the addition of two additional cancelling isolated fixed points, say $(a,b)$ and $(-a,b)$. Using (\ref{eq:S_invariant}) we see that their contributions cancel with each other.
\end{proof}

\begin{proposition}
\label{prop:comparing_n_and_S}
For all primes $p\geq 5$, we have that
\begin{align*}
    &\tfrac{1}{8}\SSS_{0}(N(2pn),\tau_{p,n})=n_{p/2}(-\Sigma(2,3,12pn-1),\rho_{p},g,\nabla^{\infty}), \\
    &\tfrac{1}{8}\SSS_{0}(P(2pn-p+1),\tau'_{p,n})=n_{p/2}(-\Sigma(2,3,12pn-6p+1),\rho_{p},g,\nabla^{\infty})\\
    &\qquad\qquad\qquad\qquad\qquad\qquad+\left\{
		\begin{array}{ll}
			4 & \mbox{if } p=5, \\
            2 & \mbox{if } p\equiv 13,17\pmod{20}, \\
			0 & \mbox{if } p=7\text{ or }p\equiv 1,9,11,19\pmod{20}, \\
            -2 & \mbox{if } p\equiv 3,7\pmod{20},\;p\neq 7.
		\end{array}
	\right.
\end{align*}
\end{proposition}

The proof of Proposition \ref{prop:comparing_n_and_S} will be deferred to Section \ref{subsec:proposition}.

\begin{proof}[Proof of Theorem \ref{theorem:intro_stabilizations}]
We first consider the case of $N(2pn)$. Let $p\geq 5$ be prime, $n\geq 1$, and for $M,N\geq 0$ let 
\[X=N(2pn)\#(\#^{M+pN}S^{2}\times S^{2}),\]
along with the $\ZZ_{p}$-action $\tau:X\to X$ obtained from the nonsmoothable $\ZZ_{p}$-action $\tau_{p,n}$ from Theorem \ref{theorem:intro_locally_linear_extension} by performing $M$ homologically trivial stabilizations and $N$ free stabilizations.

If $N$ is odd we can perform an additional free stabilization to ensure that $N$ is even, so that $b_{2}^{+}(X,\tau)_{j}=N$ is even for all $j=1,\dots,p-1$. Therefore if $\tau$ is smoothable, inequality (\ref{eq:filling_1}) from Theorem \ref{prop:10_8ths_fillings} implies that
\[(p-1)N+\kappa_{1}\geq \tfrac{1}{8}\SSS(X,\tau)_{0}\]
for every $(\kappa_{0},\kappa_{1})\in\K^{\pi}(-\Sigma(2,3,12pn-1),\rho_{p})$. Note that by Proposition \ref{prop:seifert} we have that
\[\big(2pn-2B_{pn,p,0}-n_{0},-(2pn-2B_{pn,p,0}-n_{0})\big)\in \K^{\pi}(-\Sigma(2,3,12pn-1),\rho_{p}),\]
where:
\begin{align*}
    &B_{pn,p,0}=\#\{1\le k\le pn\;|\;12k\equiv 7\pmod{p}\}, \\
    &n_{0}=n_{0}(-\Sigma(2,3,12pn-1),\rho_{p},g,\nabla^{\infty}).
\end{align*}
Hence in particular we have that
\[(p-1)N-(2pn-2B_{pn,p,0}-n_{0})\geq \tfrac{1}{8}\SSS(X,\tau)_{0}.\]
By Lemma \ref{lemma:stabilization_S_invariant} and Proposition \ref{prop:comparing_n_and_S}, have that
\[\tfrac{1}{8}\SSS(X,\tau)_{0}-n_{0}=\tfrac{1}{8}\SSS(N(2n),\tau_{p,n})_{0}-n_{0}=0,\]
and therefore we obtain
\[(p-1)N\geq 2pn-2B_{pn,p,0}.\]
Now note that since $(p,12)=1$, we have that
\[\{1\le k\le pn\;|\;12k\equiv 7\pmod{p}\}=\{a,a+p,\dots,a+(n-1)p\}\]
where $a$ is the unique solution to $12a\equiv 7\pmod{p}$ satisfying $1\le a\le p$. It follows that $B_{pn,p,0}=n$, and hence by rearranging the above inequality we see that $\tau$ is smoothable only if $N\geq 2n$. Therefore (recalling that $N$ was assumed to be even) if $N\le 2n-2$, $\tau$ must be non-smoothable regardless of the value of $M\geq 0$.

Next let
\[X'=P(2pn-p+1)\#(\#^{M+pN}S^{2}\times S^{2}).\]
along with the $\ZZ_{p}$-action $\tau':X'\to X'$ obtained from the nonsmoothable $\ZZ_{p}$-action on $P(2pn-p+1)$ by performing $M$ homologically trivial stabilizations and $N$ free stabilizations. Again without loss of generality we assume that $N$ is even. A similar argument as above implies 
\[(p-1)N+\kappa_{1}\geq \tfrac{1}{8}\SSS(X',\tau')_{0}\]
Proposition \ref{prop:10_8ths_fillings} implies that
\[(p-1)N-(2pn-2A_{pn-\frac{p-1}{2},p,0}-n'_{0})\geq \tfrac{1}{8}\SSS(X',\tau')_{0},\]
where:
\begin{align*}
    &A_{pn-\frac{p-1}{2},p,0}=\#\{1\le k\le pn-\tfrac{p-1}{2}\;|\;12k\equiv 11\pmod{p}\}, \\
    &n'_{0}=n_{0}(-\Sigma(2,3,12pn-6p+1),\rho_{p},g,\nabla^{\infty}).
\end{align*}
Note that $A_{pn-\frac{p-1}{2},p,0}\le A_{pn,p,0}=n$ by the same argument as above, and that
\[\tfrac {1}{8}\SSS_{0}(X',\tau')-n'_{0}=\tfrac{1}{8}\SSS_{0}(P(2pn-p+1),\tau'_{p,n})-n'_{0}\geq -2\]
by Proposition \ref{prop:comparing_n_and_S}. Hence
\[(p-1)N\geq 2pn-2A_{pn-\frac{p-1}{2},p,0}+\tfrac{1}{8}\SSS(X',\tau')_{0}-n'_{0}\geq 2pn-2n-2,\]
And so $\tau'$ is smoothable only if $N\geq 2n-\tfrac{2}{p-1}$. Since $\tfrac{2}{p-1}\le 1$ for all $p\geq 5$, we have that $\tau'$ is nonsmoothable if $N\le 2n-2$, regardless of the value of $M\geq 0$.
\end{proof}

Next we give a proof of Corollary \ref{cor:equivariant_embeddings}:

\begin{proof}[Proof of Corollary \ref{cor:equivariant_embeddings}]
We prove the result for the family $(\Sigma(2,3,12pn-1),\rho_{p})$, as the proof for $(\Sigma(2,3,12pn-6p+1),\rho_{p})$ is similar. Note that (2) implies that suffices to prove the result for $k=2$, since one can simply perform homologically trivial stabilizations on the ambient manifold to obtain the result for arbitrary $k\geq 2$.

Fix $p\geq 5$ prime, and for each $n\geq 1$ let $Y_{p,n}=\Sigma(2,3,12pn-1)$. Note that the double $X=-N(2pn)\cup_{Y_{p,n}}N(2pn)$ is diffeomorphic to $\#^{2}S^{2}\times S^{2}$, equipped with a distinguished smooth embedding $e_{p,n}:Y_{p,n}\hookrightarrow X$. Let $\tau''_{p,n}:X\to X$ be the locally linear $\ZZ_{p}$-action which on each side restricts to the nonsmoothable action $\tau_{p,n}:N(2pn)\to N(2pn)$ from Corollary \ref{cor:intro_nonsmoothable_actions}. Note that the fixed-point data of the actions $\tau''_{p,n}$ all agree with that of the smooth $\ZZ_{p}$-action $\tau_{p}:=\tau_{(p;2,3)}\#\tau_{(p;-2,3)}$ on $X$ obtained by taking the equivariant connected sum of the actions $\tau_{(p;\pm 2,3)}$ on two copies of $S^{2}\times S^{2}$ along a fixed point from each side. By (\cite{BW96} Theorem 1.3), the locally linear actions $\tau''_{p,n}$ are all topologically conjugate via a self-homeomorphism $f_{p,n}:X\to X$ to $\tau_{p}$. The composite $i_{p,n}:=f_{p,n}\circ e_{p,n}$ is then a locally flat equivariant embedding $Y_{p,n}\hookrightarrow X$.

If $i_{p,n}$ were equivariantly isotopic to a smooth embedding, say $i'_{p,n}$, then the restriction of $\tau_{p}$ to each side of $\#^{2}S^{2}\times S^{2}\setminus i'_{p,n}(Y_{p,n})$ would be a smooth homologically trivial $\ZZ_{p}$-action on $\pm N(2pn)$, contradicting Theorem \ref{theorem:intro_no_smooth_extension}. This proves (1).

For (2): Let $\tau_{p,M,N}$ denote the smooth $\ZZ_{p}$-action obtained from $\tau_{p}$ after $M$ homologically trivial and $N$ free stabilizations, and let $i_{p,n,M,N}$ denote the induced embedding of $Y_{p,n}$. If $i_{p,n,M,N}$ were equivariantly isotopic to a smooth embedding $i'_{p,n,M,N}$, then the induced $\ZZ_{p}$ actions on each side of $\#^{2}S^{2}\times S^{2}\setminus i'_{p,n,M,N}(Y_{p,n})$ are smooth $\ZZ_{p}$-actions on $N(2pn)\#^{2+M_{1}+pN_{1}}S^{2}\times S^{2}$ and $-N(2pn)\#^{2+M_{2}+pN_{2}}S^{2}\times S^{2}$ for some $M_{1}+M_{2}=M$, $N_{1}+N_{2}=N$. Furthermore, the fixed-point data and induced action on homology of each piece coincide precisely with those of the stabilized actions considered in Theorem \ref{theorem:intro_stabilizations}. But by the proof of Theorem \ref{theorem:intro_stabilizations}, this implies that $N_{1}$ and $N_{2}\geq 2n-1$. Hence if one of $N_{1}$ or $N_{2}$ is less than or equal to $2n-2$, then such a smooth embedding $i'_{p,n,M,N}$ cannot exist.
\end{proof}

\subsection{Calculating Equivariant Correction Terms}
\label{subsec:proposition}

In this section we prove Proposition \ref{prop:comparing_n_and_S}. First, let us say more about the equivariant correction term $n(Y,\wh{\rho}_{r},g,\nabla^{\infty})$ in our context. Let $Y=\Sigma(\alpha_{1},\dots,\alpha_{n})$ be a Seifert-fibered homology sphere of negative fibration, and define
\[\rho(Y):=\twopartdef{0}{\text{the }\alpha_{i}\text{ are all odd}}{\frac{1}{2}}{\text{one of the }\alpha_{i}\text{ is even}}\]
as in Section \ref{sec:s_1_action}. Let $r\geq 2$ (not necessarily prime) be such that $(r,\alpha_{i})=1$ for all $i=1,\dots,n$, so that the $\ZZ_{r}$-action $\rho_{r}:Y\to Y$ acts \emph{freely}, and let $\dirac^{\infty}$ denote the Dirac operator defined with respect to the connection $\nabla^{\infty}$. For $q\in\ZZ$ we denote by $\eta_{\dirac^{\infty}}^{(q,r)}\in\CC$ the equivariant eta-invariant of $Y$ at $(\wh{\rho}_{r})^{q}$, where $\wh{\rho}_{r}$ denotes the distinguished spin lift of $\rho_{r}$ induced from the unique spin lift $\wh{\rho}$ of the $S^{1}$-action $\rho:S^{1}\times Y\to Y$. Using the definition of the correction term in \cite{Mon:SW} and noting that the kernel of $\dirac^{\infty}$ is zero \cite{Nic00}, we have that
\[n_{L}(Y,\wh{\rho}_{r},g,\nabla^{\infty})=\tfrac{1}{r}n(Y,g,\nabla^{\infty})+\tfrac{1}{2r}\sum_{q=1}^{r-1}e^{-2\pi iqL/r}\eta_{\dirac^{\infty}}^{(q,r)}\in\QQ\]
for all $L\in\frac{1}{2}\ZZ$ with $L\equiv\rho(Y)\pmod{\ZZ}$, $0\le L\le p-\frac{1}{2}$. For the following let $\beta_{1},\dots,\beta_{n}$ and $\gamma_{1},\dots,\gamma_{n}$ be the unique sets of integers satisfying
\begin{align}
    &\tfrac{\alpha\beta_{i}}{\alpha_{i}}\equiv-1\pmod{\alpha_{i}}, & &1\le \beta\le \alpha_{i}-1, \label{eq:beta} \\
    &\tfrac{\alpha\gamma_{i}}{\alpha_{i}}\equiv\rho(Y)+\sum_{i=1}^{n}\tfrac{\alpha(\alpha_{i}-1)}{2\alpha_{i}}\pmod{\alpha_{i}}, & &0\le \gamma_{i}\le \alpha_{i}-1, \label{eq:gamma}
\end{align}
for all $i=1,\dots, n$, where $\alpha=\alpha_{1}\cdots\alpha_{n}$.

\begin{proposition}[\cite{Mon:eta}]
\label{prop:equivariant_eta}
Let $Y=\Sigma(\alpha_{1},\dots,\alpha_{n})$ and $r\geq 2$ as above, and let $L\in\frac{1}{2}\ZZ$ be such that $L\equiv\rho(Y)\pmod{\ZZ}$, $0\le L\le p-\frac{1}{2}$. For each $i=1,\dots,n$:
\begin{enumerate}
    \item Let $p_{i}$ be any integer such that $p_{i}\beta_{i}\equiv 1\pmod{\alpha_{i}}$.
    \item Let $\alpha'_{i}$ be any integer such that $\alpha_{i}\alpha_{i}'\equiv 1\pmod{2p}$ if $\alpha_{i}$ is odd, and $\alpha_{i}\alpha_{i}'\equiv 2\pmod{2p}$ if $\alpha_{i}$ is even.
    \item Let $A_{i}:=2\alpha_{i}\{\frac{p_{i}\gamma_{i}+\rho(Y)}{\alpha_{i}}\}-\alpha_{i}\in\ZZ$.
    \item Let $A_{i}':=\frac{1}{2}\alpha_{i}'(A_{i}-2L)\in\frac{1}{2}\ZZ$.
\end{enumerate}
Furthermore, let
\[s(b,a;x,y):=\sum_{k=0}^{|a|-1}\Big(\Big(x+\frac{b}{a}(k+y)\Big)\Big)\Big(\Big(\frac{k+y}{a}\Big)\Big),\qquad b,a\in\ZZ,\;x,y\in\RR,\]
denote the Dedekind-Rademacher sum from \cite{Rademacher64}, where $x\mapsto ((x))$ denotes the generalized sawtooth function
\begin{align*}
    &((x)):=\twopartdef{\{x\} - \frac{1}{2}}{x\not\in\ZZ,}{0}{x\in\ZZ,} & &\{x\}:=x-\lfloor x\rfloor.
\end{align*}
Finally, let $s(b,a):=s(b,a;0,0)$ denote the ordinary Dedekind sum, and for $x\in\RR$, $n\in\ZZ$, let
\[\delta(x,n):=\left\{
		\begin{array}{ll}
			n & \mbox{if } x\in\ZZ\text{ and }x\equiv 0\pmod{n}, \\
            0 & \mbox{otherwise.}
		\end{array}
\right.\]
Then the following formulas hold:
\begin{enumerate}
    \item[(a)] If $\rho(Y)=0$, then:
    \begin{align*}
        &n_{L}(Y,\wh{\rho}_{r},g,\nabla^{\infty})=\sum_{i=1}^{n}\Big(s(r\beta_{i},\alpha_{i};\tfrac{\gamma_{i}}{\alpha_{i}},-\tfrac{L}{r})+\tfrac{1}{2r}s(\beta_{i},\alpha_{i})+\tfrac{1}{2r}((\tfrac{p_{i}\gamma_{i}}{\alpha_{i}}))\Big) \\
        &\qquad+\Big(\sum_{i=1}^{n}\tfrac{1}{4}(1-2\{\tfrac{A_{i}'}{r}\})\Big)-\Big(\sum_{i=1}^{n}\tfrac{1}{2\alpha_{i}}\Big)((\tfrac{L}{r}))+\tfrac{r}{12\alpha}+\tfrac{L(L-r)}{2r\alpha}+\tfrac{1}{24r\alpha}-\tfrac{1}{8r}.
    \end{align*}
    \item[(b)] If $\rho(Y)=\frac{1}{2}$, then:
    \begin{align*}
        &n_{L}(Y,\wh{\rho}_{r},g,\nabla^{\infty})=\sum_{i=1}^{n}\Big(s(r\beta_{i},\alpha_{i};\tfrac{\gamma_{i}+\frac{1}{2}\beta_{i}}{\alpha_{i}},-\tfrac{L}{r})+\tfrac{1}{2r}s(\beta_{i},\alpha_{i})+\tfrac{1}{2r}((\tfrac{p_{i}\gamma_{i}+\frac{1}{2}}{\alpha_{i}}))\Big) \\
        &\qquad+\sum_{i|\alpha_{i}\text{ even}}\tfrac{1}{2}\delta(A_{i}-2L,2)\{A_{i}'-\tfrac{1}{2}\}\Big(\{\tfrac{r}{2}\{\tfrac{A_{i}'}{r}\}\}-\{\tfrac{r}{2}\}\{\tfrac{A_{i}'}{r}\}-4\{\tfrac{r-1}{2}\}\{\tfrac{r}{2}\{\tfrac{A_{i}'}{r}\}\}\{\tfrac{A_{i}'}{r}\}\Big) \\
        &\qquad+\Big(\sum_{i|\alpha_{i}\text{ odd}}\tfrac{1}{2}\{A_{i}'\}(1-2\{\tfrac{A_{i}'}{r}\})\Big)-\Big(\sum_{i=1}^{n}\tfrac{1}{2\alpha_{i}}\Big)((\tfrac{L}{r}))+\tfrac{r}{12\alpha}+\tfrac{L(L-r)}{2r\alpha}+\tfrac{1}{24r\alpha}-\tfrac{1}{8r}.
    \end{align*}
\end{enumerate}
\end{proposition}

The following lemma gives a way to efficiently compute the Dedekind and Dedekind-Rademacher sums appearing in Proposition \ref{prop:equivariant_eta}, and essentially follows from the corresponding reciprocity theorems:

\begin{lemma}[\cite{Mon:eta}]
\label{lemma:reciprocity_dedekind_rademacher}
The following statements are true:
\begin{enumerate}
    \item Let $b,a\in\ZZ_{+}$ be such that $b>a$, $(b,a)=1$, and let $x,y\in\RR$ be such that:
    \begin{itemize}
        \item $x,y$ are not both integers, and
        \item $ax+by\in\ZZ$.
    \end{itemize}
    Let $a_{0},\dots,a_{n}$ be defined inductively by $a_{0}=b$, $a_{1}=a$, and $a_{j}=a_{j-2}-q_{j-1}a_{j-1}$ for some $q_{1},\dots,q_{n-1}\in\ZZ_{+}$ terminating at $a_{n}=1$. Furthermore, let:
    \begin{itemize}
        \item $x_{0},\dots,x_{n}\in\RR/\ZZ$ be defined by $x_{0}\equiv x$, $x_{1}\equiv y$, and $x_{j}\equiv q_{j-1}x_{j-1}+x_{j-2}$ for $2\le j\le n$.
        \item $s_{0},\dots,s_{n}$ be defined by $s_{0}=0$, $s_{1}=1$, and $s_{j}=q_{j-1}s_{j-1}+s_{j-2}$ for $2\le j\le n$.
    \end{itemize}
    Then
    \begin{align*}
        s(b,a;x,y)&=\sum_{j=1}^{n-1}(-1)^{j+1}((x_{j}))((x_{j+1}))+\frac{a_{0}}{2a_{1}}\scr{B}_{2}(x_{1})+(-1)^{n}\frac{a_{n-1}}{2}\scr{B}_{2}(x_{n})\\
        &\qquad+\frac{(-1)^{n}a_{1}s_{n}+1}{12a_{0}a_{n-1}}+\sum_{j=1}^{n-1}(-1)^{j}\frac{q_{j}}{2}\scr{B}_{2}(x_{j}),
    \end{align*}
    where
    \[\scr{B}_{2}(x):=\{x\}^{2}-\{x\}+\tfrac{1}{6}.\]
    \item Let $b,a\in\ZZ_{+}$ be such that $b<a$ and $(b,a)=1$. Let $a_{0},\dots,a_{n}$ be defined inductively by $a_{0}=a$, $a_{1}=b$, and $a_{j}=a_{j-2}-q_{j-1}a_{j-1}$ for some $q_{1},\dots,q_{n-1}\in\ZZ_{+}$ terminating at $a_{n}=1$. Furthermore, let $s_{0},\dots,s_{n-1}$ be defined by $s_{0}=0$, $s_{1}=1$, and $s_{j}=q_{j-1}s_{j-1}+s_{j-2}$ for $2\le j\le n$. Then
    \begin{align*}
        s(b,a)&=(-1)^{n-1}\frac{a_{n-1}^{2}+2}{12a_{n-1}}+\frac{a_{1}}{12a_{0}}+(-1)^{n}\frac{a_{n-2}}{12a_{n-1}} \\
        &\qquad+(-1)^{n}\frac{s_{n-1}}{12a_{0}a_{n-1}}+\sum_{j=1}^{n-2}(-1)^{j-1}\frac{q_{j}}{12}+\twopartdef{0}{n\text{ even},}{-\frac{1}{4}}{n\text{ odd}.}
    \end{align*}
\end{enumerate}
\end{lemma}

We now calculate the equivariant correction terms appearing in Proposition \ref{prop:comparing_n_and_S}:

\begin{proposition}
\label{prop:correction_terms}
The following formulas hold for all primes $p\geq 5$ and all $n\geq 1$:
\begin{align*}
    &n_{p/2}(\Sigma(2,3,12pn-1),\wh{\rho}_{p},g,\nabla^{\infty})=\twopartdef{\frac{p^{2}\mp 14p+13}{144p}}{p\equiv \pm 1\pmod{12},}{\frac{p^{2}\pm 50p+13}{144p}}{p\equiv \pm 5\pmod{12},} \\
    &n_{p/2}(\Sigma(2,3,12pn-6p+1),\wh{\rho}_{p},g,\nabla^{\infty})=\twopartdef{\frac{-p^{2}\pm 158p-13}{144p}}{p\equiv \pm 1\pmod{12},}{\frac{-p^{2}\pm 94p-13}{144p}}{p\equiv \pm 5\pmod{12}.}
\end{align*}
\end{proposition}

To prove Proposition \ref{prop:correction_terms}, we will make use of the following periodicity property of the correction terms for the family of Brieskorn spheres $\Sigma(2,3,6n\pm 1)$:

\begin{lemma}[\cite{Mon:eta}]
\label{lemma:periodicity}
Let $a\geq -7$ be such that $(a,6)=1$, let $p\geq 5$ be prime, and let $n\geq 1$ be such that $(12n+a,p)=1$. Then for any $m\geq 1$ we have that
\[n(\Sigma(2,3,12n+a),\wh{\rho}_{p},g,\nabla^{\infty})=n(\Sigma(2,3,12(n+mp)+a),\wh{\rho}_{p},g,\nabla^{\infty})\in R(\ZZ_{2p})\otimes\QQ.\]
\end{lemma}

\begin{proof}[Proof of Proposition \ref{prop:correction_terms}]
Note that Lemma \ref{lemma:periodicity} implies that
\begin{align*}
    &n_{p/2}(\Sigma(2,3,12pn-1),\wh{\rho}_{p},g,\nabla^{\infty})= n_{p/2}(\Sigma(2,3,12p-1),\wh{\rho}_{p},g,\nabla^{\infty}), \\
    &n_{p/2}(\Sigma(2,3,12pn-6p+1),\wh{\rho}_{p},g,\nabla^{\infty})=n_{p/2}(\Sigma(2,3,6p+1),\wh{\rho}_{p},g,\nabla^{\infty}),
\end{align*}
i.e., it suffices to consider the case $n=1$. 

Let $(\alpha_{1},\alpha_{2},\alpha_{3})=(2,3,12p-1)$. From (\ref{eq:beta}) and (\ref{eq:gamma}) we obtain
\begin{align*}
    &(\beta_{1},\beta_{2},\beta_{3})=(1,2,10p-1), & &(\gamma_{1},\gamma_{2},\gamma_{3})=(1,0,7p-1),
\end{align*}
from which we can deduce that
\begin{align*}
    &(p_{1},p_{2},p_{3})=(1,2,12p-7), \\
    &(A_{1},A_{2},A_{3})=(1,-2,6), \\
    &(\alpha_{1}',\alpha_{2}',\alpha_{3}')=\twopartdef{(1,\frac{2p+1}{3},-1)}{p\equiv 1\pmod{6},}{(1,\frac{-2p+1}{3},-1)}{p\equiv 5\pmod{6},} \\
    &(A_{1}',A_{2}',A_{3}')=\twopartdef{(\frac{1-p}{2},\frac{-2p^{2}-5p-2}{6},\frac{p-6}{2})}{p\equiv 1\pmod{6},}{(\frac{1-p}{2},\frac{2p^{2}+3p-2}{6},\frac{p-6}{2})}{p\equiv 5\pmod{6}.}
\end{align*}
Therefore by Proposition \ref{prop:equivariant_eta} we have that
\begin{align*}
    &n_{p/2}(\Sigma(2,3,12p-1),\wh{\rho}_{p},g,\nabla^{\infty}) \\
    &=s(p,2;\tfrac{3}{4},-\tfrac{1}{2})+s(2p,3;\tfrac{1}{3},-\tfrac{1}{2})+s(10p^{2}-p,12p-1;\tfrac{24p-3}{24p-2},-\tfrac{1}{2}) \\
    &\qquad+\tfrac{1}{2p}\Big(s(1,2)+s(2,3)+s(10p-1,12p-1)+((\tfrac{3}{4}))+((\tfrac{1}{6}))+((\tfrac{12p+5}{24p-2}))\Big) \\
    &\qquad+\tfrac{1}{2}\delta(1-p,2)\{\tfrac{1-p}{2}-\tfrac{1}{2}\}\Big(\{\tfrac{p}{2}\{\tfrac{1-p}{2p}\}\}-\{\tfrac{p}{2}\}\{\tfrac{1-p}{2p}\}-4\{\tfrac{p-1}{2}\}\{\tfrac{p}{2}\{\tfrac{1-p}{2p}\}\}\{\tfrac{1-p}{2p}\}\Big) \\
    &\qquad+\left\{
		\begin{array}{ll}
            \tfrac{1}{2}\{\frac{-2p^{2}-5p-2}{6}\}(1-2\{\frac{-2p^{2}-5p-2}{6p}\}) & \mbox{if } p\equiv 1\pmod{6},\\
			\tfrac{1}{2}\{\frac{2p^{2}+3p-2}{6}\}(1-2\{\frac{2p^{2}+3p-2}{6p}\}) & \mbox{if } p\equiv 5\pmod{6},
		\end{array}
	\right. \\
    &\qquad+\tfrac{1}{2}\{\tfrac{p-6}{2}\}(1-2\{\tfrac{p-6}{2p}\}) \\
    &\qquad-\Big(\tfrac{1}{4}+\tfrac{1}{6}+\tfrac{1}{24p-2}\Big)\Big(\Big(\tfrac{\frac{p}{2}}{p}\Big)\Big)+\tfrac{p}{12(72p-6)}+\tfrac{\frac{p}{2}(\frac{p}{2}-p)}{2p(72p-6)}+\tfrac{1}{24p(72p-6)}-\tfrac{1}{8p} \\
    &=s(p,2;\tfrac{3}{4},-\tfrac{1}{2})+s(2p,3;\tfrac{1}{3},-\tfrac{1}{2})+s(10p^{2}-p,12p-1;\tfrac{24p-3}{24p-2},-\tfrac{1}{2}) \\
    &\qquad+\tfrac{1}{2p}\Big(s(1,2)+s(2,3)+s(10p-1,12p-1)\Big) \\
    &\qquad+\tfrac{-p^{2}-216p+235}{144p(12p-1)}+\left\{\renewcommand{\arraystretch}{1.2}
		\begin{array}{ll}
            -\frac{1}{5} & \mbox{if } p=5\\
			\frac{3}{2p} & \mbox{if } p\geq 7
		\end{array}
	\right.+\left\{\renewcommand{\arraystretch}{1.2}
		\begin{array}{ll}
            \mp\frac{1}{24} & \mbox{if } p\equiv \pm 1\pmod{12},\\
			\pm\frac{7}{24} & \mbox{if } p\equiv \pm 5\pmod{12}.
		\end{array}
	\right. \\
\end{align*}

First we calculate the Dedekind sums appearing in the above expression. By direct calculation we obtain
\begin{align*}
    &s(1,2)=((0))((0))+((\tfrac{1}{2}))((\tfrac{1}{2}))=0, \\
    &s(2,3)=((0))((0))+((\tfrac{2}{3}))((\tfrac{1}{3}))+((\tfrac{4}{3}))((\tfrac{2}{3}))=-\tfrac{1}{18}.
\end{align*}
For the sum $s(10p-1,12p-1)$, the Euclidean algorithm outlined in (2) of Lemma \ref{lemma:reciprocity_dedekind_rademacher} gives us
\begin{align*}
    &(a_{0},a_{1},a_{2},a_{3},a_{4})=(12p-1,10p-1,2p,2p-1,1), & &(q_{1},q_{2},q_{3})=(1,4,1), \\
    &(s_{0},s_{1},s_{2},s_{3})=(0,1,1,5), & &
\end{align*}
from which we deduce that
\begin{align*}
    s(10p-1,12p-1)&=-\tfrac{(2p-1)^{2}+2}{12(2p-1)}+\tfrac{10p-1}{12(12p-1)}+\tfrac{2p}{12(2p-1)} \\
    &\qquad+\tfrac{5}{12(12p-1)(2p-1)}+\tfrac{1}{12}-\tfrac{4}{12}+0 \\
    &=\tfrac{-4p^{2}-1}{24p-2}.
\end{align*}
Next we calculate the Dedekind-Rademacher sums. Note that all three sums satisfy the required properties in (1) of Lemma \ref{lemma:reciprocity_dedekind_rademacher}. For $s(p,2;\frac{3}{4},\frac{1}{2})$ the corresponding sequences are given by
\begin{align*}
    &(a_{0},a_{1},a_{2})=(p,2,1), & &(q_{1},q_{2})=(\tfrac{p-1}{2},2), \\
    &(x_{0},x_{1},x _{2})\equiv(\tfrac{3}{4},-\tfrac{1}{2},\tfrac{4-p}{4}), & &(s_{0},s_{1},s_{2})=(0,1,\tfrac{p-1}{2}),
\end{align*}
and hence
\begin{align*}
    s(p,2;\tfrac{3}{4},\tfrac{1}{2})&=((-\tfrac{1}{2}))((\tfrac{4-p}{4}))+\tfrac{p}{2(2)}\scr{B}_{2}(-\tfrac{1}{2})+\tfrac{2}{2}\scr{B}_{2}(\tfrac{4-p}{4})\\
    &\qquad+\tfrac{(2)(\frac{p-1}{2})+1}{12(p)(2)}-\tfrac{\frac{p-1}{2}}{2}\scr{B}_{2}(-\tfrac{1}{2})=0.
\end{align*}
Next we compute $s(2p,3;\frac{1}{3},-\frac{1}{2})$. If $p\equiv 1\pmod{6}$, then:
\begin{align*}
    &(a_{0},a_{1},a_{2},a_{3})=(2p,3,2,1), & &(q_{1},q_{2},q_{3})=(\tfrac{2p-2}{3},1,2), \\
    &(x_{0},x_{1},x_{2},x_{3})\equiv(\tfrac{1}{3},-\tfrac{1}{2},\tfrac{2-p}{3},\tfrac{1-2p}{6}), & &(s_{0},s_{1},s_{2},s_{3})=(0,1,\tfrac{2p-2}{3},\tfrac{2p+1}{3}),
\end{align*}
\begin{align*}
    \implies s(2p,3;\tfrac{1}{3},-\tfrac{1}{2})&=((-\tfrac{1}{2}))((\tfrac{2-p}{3}))-((\tfrac{2-p}{3}))((\tfrac{1-2p}{6}))+\tfrac{2p}{2(3)}\scr{B}_{2}(-\tfrac{1}{2})-\tfrac{2}{2}\scr{B}_{2}(\tfrac{1-2p}{6})\\
    &\qquad+\tfrac{-(3)(\frac{2p+1}{3})+1}{12(2p)(2)}-\tfrac{\frac{2p-2}{3}}{2}\scr{B}_{2}(-\tfrac{1}{2})+\tfrac{1}{2}\scr{B}_{2}(\tfrac{2-p}{3})=-\tfrac{1}{18},
\end{align*}
and similarly if $p\equiv 5\pmod{6}$, then:
\begin{align*}
    &(a_{0},a_{1},a_{2})=(2p,3,1), & &(q_{1},q_{2})=(\tfrac{2p-1}{3},3), \\
    &(x_{0},x_{1},x_{2})\equiv(\tfrac{1}{3},-\tfrac{1}{2},\tfrac{3-2p}{6}), & &(s_{0},s_{1},s_{2})=(0,1,\tfrac{2p-1}{3}),
\end{align*}
\begin{align*}
    \implies s(2p,3;\tfrac{1}{3},-\tfrac{1}{2})&=((-\tfrac{1}{2}))((\tfrac{3-2p}{6}))+\tfrac{2p}{2(3)}\scr{B}_{2}(-\tfrac{1}{2})+\tfrac{3}{2}\scr{B}_{2}(\tfrac{3-2p}{6})\\
    &\qquad+\tfrac{(3)(\frac{2p-1}{3})+1}{12(2p)(3)}-\tfrac{\frac{2p-1}{3}}{2}\scr{B}_{2}(-\tfrac{1}{2})=\tfrac{1}{18}.
\end{align*}
Finally we compute the Dedekind-Rademacher sum $s(10p^{2}-p,12p-1;\frac{24p-3}{24p-2},-\frac{1}{2})$. There are four cases corresponding to the value of $p$ modulo 12 --- we first consider the case $p\equiv 5\pmod{12}$. If $p=5$, by direct calculation one can show that
\[s(245,59;\tfrac{117}{118},-\tfrac{1}{2})=\tfrac{43}{118}.\]
For $p\geq 17$, we see that
\begin{align*}
    &(a_{0},a_{1},a_{2},a_{3},a_{4},a_{5},a_{6},a_{7})=(10p^{2}-p,12p-1,\tfrac{11p-1}{6},p,\tfrac{5p-1}{6},\tfrac{p+1}{6},\tfrac{p-5}{6},1), \\
    &(q_{1},q_{2},q_{3},q_{4},q_{5},q_{6},q_{7})=(\tfrac{5p-1}{6},6,1,1,4,1,\tfrac{p-5}{6}) \\
    &(x_{0},x_{1},x_{2},x_{3},x_{4},x_{5},x_{6},x_{7})\equiv(\tfrac{24p-3}{24p-2},-\tfrac{1}{2},\tfrac{24p-3}{24p-2},\tfrac{12p-7}{24p-2},\tfrac{6p-4}{12p-1},\tfrac{24p-15}{24p-2},\tfrac{6p-30}{12p-1},\tfrac{12p-73}{24p-2}), \\
    &(s_{0},s_{1},s_{2},s_{3},s_{4},s_{5},s_{6},s_{7})=(0,1,\tfrac{5p-1}{6},5p,\tfrac{35p-1}{6},\tfrac{65p-1}{6},\tfrac{295p-5}{6},60p-1),
\end{align*}
from which we deduce that
\begin{align*}
    &s(10p^{2}-p,12p-1;\tfrac{24p-3}{24p-2},-\tfrac{1}{2})=((-\tfrac{1}{2}))((\tfrac{24p-3}{24p-2}))-((\tfrac{24p-3}{24p-2}))((\tfrac{12p-7}{24p-2}))+((\tfrac{12p-7}{24p-2}))((\tfrac{6p-4}{12p-1})) \\
    &\qquad-((\tfrac{6p-4}{12p-1}))((\tfrac{24p-15}{24p-2}))+((\tfrac{24p-15}{24p-2}))((\tfrac{6p-30}{12p-1}))-((\tfrac{6p-30}{12p-1}))((\tfrac{12p-73}{24p-2}))\\
    &\qquad+\tfrac{10p^{2}-p}{2(12p-1)}\scr{B}_{2}(-\tfrac{1}{2})-\tfrac{\frac{p-5}{6}}{2}\scr{B}_{2}(\tfrac{12p-73}{24p-2})+\tfrac{-(12p-1)(60p-1)+1}{12(10p^{2}-p)(\frac{p-5}{6})}-\tfrac{\frac{5p-1}{6}}{2}\scr{B}_{2}(-\tfrac{1}{2}) \\
    &\qquad+\tfrac{6}{2}\scr{B}_{2}(\tfrac{24p-3}{24p-2})-\tfrac{1}{2}\scr{B}_{2}(\tfrac{12p-7}{24p-2})+\tfrac{1}{2}\scr{B}_{2}(\tfrac{6p-4}{12p-1})-\tfrac{4}{2}\scr{B}_{2}(\tfrac{24p-15}{24p-2})+\tfrac{1}{2}\scr{B}_{2}(\tfrac{6p-30}{12p-1}) \\
    &=\tfrac{p^{2}+12p-181}{12(12p-1)}.
\end{align*}
The other cases $p\equiv 1,7,11\pmod{12}$ are similar, and one can show that
\[s(10p^{2}-p,12p-1;\tfrac{24p-3}{24p-2},-\tfrac{1}{2})=\left\{\renewcommand{\arraystretch}{1.2}
    \begin{array}{ll}
        \frac{43}{118} & \mbox{if } p=5,\\
	\tfrac{p^{2}+12p-181}{12(12p-1)} & \mbox{if } p\geq 7.
    \end{array}
    \right.\]
Putting this all together, we have that
\begin{align*}
    &n_{p/2}(\Sigma(2,3,12p-1),\wh{\rho}_{p},g,\nabla^{\infty}) \\
    &=s(p,2;\tfrac{3}{4},-\tfrac{1}{2})+s(2p,3;\tfrac{1}{3},-\tfrac{1}{2})+s(10p^{2}-p,12p-1;\tfrac{24p-3}{24p-2},-\tfrac{1}{2}) \\
    &\qquad+\tfrac{1}{2p}\Big(s(1,2)+s(2,3)+s(10p-1,12p-1)\Big) \\
    &\qquad+\tfrac{-p^{2}-216p+235}{144p(12p-1)}+\left\{\renewcommand{\arraystretch}{1.2}
		\begin{array}{ll}
            -\frac{1}{5} & \mbox{if } p=5\\
			\frac{3}{2p} & \mbox{if } p\geq 7
		\end{array}
	\right.+\left\{\renewcommand{\arraystretch}{1.2}
		\begin{array}{ll}
            \mp\frac{1}{24} & \mbox{if } p\equiv \pm 1\pmod{12},\\
			\pm\frac{7}{24} & \mbox{if } p\equiv \pm 5\pmod{12}.
		\end{array}
	\right. \\
    &=0+\left\{\renewcommand{\arraystretch}{1.2}
		\begin{array}{ll}
            -\frac{1}{18} & \mbox{if } p\equiv 1\pmod{6}\\
		\frac{1}{18} & \mbox{if } p\equiv 5\pmod{6}
		\end{array}
	\right.
        +\left\{\renewcommand{\arraystretch}{1.2}
    \begin{array}{ll}
        \frac{43}{118} & \mbox{if } p=5\\
	\tfrac{p^{2}+12p-181}{12(12p-1)} & \mbox{if } p\geq 7
    \end{array}
    \right. \\
    &\qquad+\tfrac{1}{2p}\Big(0-\tfrac{1}{18}+\tfrac{-4p^{2}-1}{24p-2}\Big) \\
    &\qquad+\tfrac{-p^{2}-216p+235}{144p(12p-1)}+\left\{\renewcommand{\arraystretch}{1.2}
		\begin{array}{ll}
            -\frac{1}{5} & \mbox{if } p=5\\
			\frac{3}{2p} & \mbox{if } p\geq 7
		\end{array}
	\right.+\left\{\renewcommand{\arraystretch}{1.2}
		\begin{array}{ll}
            \mp\frac{1}{24} & \mbox{if } p\equiv \pm 1\pmod{12},\\
			\pm\frac{7}{24} & \mbox{if } p\equiv \pm 5\pmod{12}.
		\end{array}
	\right. \\
    &=\twopartdef{\frac{p^{2}\mp 14p+13}{144p}}{p\equiv \pm 1\pmod{12},}{\frac{p^{2}\pm 50p+13}{144p}}{p\equiv \pm 5\pmod{12}.}
\end{align*}
Now let $(\alpha_{1},\alpha_{2},\alpha_{3})=(2,3,6p+1)$, for which we have:
\begin{align*}
    &(\beta_{1},\beta_{2},\beta_{3})=(1,1,p), \\
    &(\gamma_{1},\gamma_{2},\gamma_{3})=(1,2,\tfrac{11p+1}{2}), \\
    &(p_{1},p_{2},p_{3})=(1,1,6p-5), \\
    &(A_{1},A_{2},A_{3})=(1,2,6), \\
    &(\alpha_{1}',\alpha_{2}',\alpha_{3}')=\twopartdef{(1,\frac{2p+1}{3},1)}{p\equiv 1\pmod{6},}{(1,\frac{-2p+1}{3},1)}{p\equiv 5\pmod{6},} \\
    &(A_{1}',A_{2}',A_{3}')=\twopartdef{(\frac{1-p}{2},\frac{-2p^{2}+3p+2}{6},\frac{6-p}{2})}{p\equiv 1\pmod{6},}{(\frac{1-p}{2},\frac{2p^{2}-5p+2}{6},\frac{6-p}{2})}{p\equiv 5\pmod{6}.}
\end{align*}
Again via Proposition \ref{prop:equivariant_eta} we obtain:
\begin{align*}
    &n_{L}(\Sigma(2,3,6p+1),\wh{\rho}_{p},g,\nabla^{\infty}) \\
    &=s(p,2;\tfrac{3}{4},-\tfrac{1}{2})+s(p,3;\tfrac{5}{6},-\tfrac{1}{2})+s(p^{2},6p+1;\tfrac{12p+1}{12p+2},-\tfrac{1}{2}) \\
    &\qquad+\tfrac{1}{2p}\Big(s(1,2)+s(1,3)+s(p,6p+1)+((\tfrac{3}{4}))+((\tfrac{5}{6}))+((\tfrac{6p+7}{12p+2}))\Big) \\
    &\qquad+\tfrac{1}{2}\delta(1-2(\tfrac{p}{2}),2)\{\tfrac{1-p}{2}-\tfrac{1}{2}\}\Big(\{\tfrac{p}{2}\{\tfrac{1-p}{2p}\}\}-\{\tfrac{p}{2}\}\{\tfrac{1-p}{2p}\}-4\{\tfrac{p-1}{2}\}\{\tfrac{p}{2}\{\tfrac{1-p}{2p}\}\}\{\tfrac{1-p}{2p}\}\Big) \\
    &\qquad+\left\{
		\begin{array}{ll}
            \tfrac{1}{2}\{\frac{-2p^{2}+3p+2}{6}\}(1-2\{\frac{-2p^{2}+3p+2}{6p}\}) & \mbox{if } p\equiv 1\pmod{6},\\
			\tfrac{1}{2}\{\frac{2p^{2}-5p+2}{6}\}(1-2\{\frac{2p^{2}-5p+2}{6p}\}) & \mbox{if } p\equiv 5\pmod{6},
		\end{array}
	\right. \\
    &\qquad+\tfrac{1}{2}\{\tfrac{6-p}{2}\}(1-2\{\tfrac{6-p}{2p}\}) \\
    &\qquad-\Big(\tfrac{1}{4}+\tfrac{1}{6}+\tfrac{1}{12p+2}\Big)\Big(\Big(\tfrac{\frac{p}{2}}{p}\Big)\Big)+\tfrac{p}{12(36p+6)}+\tfrac{\frac{p}{2}(\frac{p}{2}-p)}{2p(36p+6)}+\tfrac{1}{24p(36p+6)}-\tfrac{1}{8p} \\
    &=s(p,2;\tfrac{3}{4},-\tfrac{1}{2})+s(p,3;\tfrac{5}{6},-\tfrac{1}{2})+s(p^{2},6p+1;\tfrac{12p+1}{12p+2},-\tfrac{1}{2}) \\
    &\qquad+\tfrac{1}{2p}\Big(s(1,2)+s(1,3)+s(p,6p+1)\Big) \\
    &\qquad+\tfrac{-p^{2}-108p+199}{144p(6p+1)}+\left\{\renewcommand{\arraystretch}{1.2}
		\begin{array}{ll}
            \frac{1}{5} & \mbox{if } p=5\\
			-\frac{3}{2p} & \mbox{if } p\geq 7
		\end{array}
	\right.+\left\{\renewcommand{\arraystretch}{1.2}
		\begin{array}{ll}
            \pm\frac{7}{24} & \mbox{if } p\equiv \pm 1\pmod{12},\\
			\mp\frac{1}{24} & \mbox{if } p\equiv \pm 5\pmod{12}.
		\end{array}
	\right. \\
\end{align*}
By similar calculations as above, one can show that
\begin{align*}
    &s(1,2)=0, & &s(1,3)=\tfrac{1}{18}, & &s(p,6p+1)=\tfrac{-p^{2}+6p}{12p+2},
\end{align*}
and
\begin{align*}
    &s(p,2;\tfrac{3}{4},-\tfrac{1}{2})=0, \\
    &s(p,3;\tfrac{5}{6},-\tfrac{1}{2})=\left\{\renewcommand{\arraystretch}{1.2}
		\begin{array}{ll}
            \frac{1}{18} & \mbox{if } p\equiv 1\pmod{6},\\
			-\frac{1}{18} & \mbox{if } p\equiv 5\pmod{6},
		\end{array}
	\right. \\
    &s(p^{2},6p+1;\tfrac{12p+1}{12p+2},-\tfrac{1}{2})=\left\{\renewcommand{\arraystretch}{1.2}
		\begin{array}{ll}
            \frac{1}{2} & \mbox{if } p=5,\\
            \frac{-p^{2}+114p+199}{24(6p+1)} & \mbox{if } p\equiv 1\pmod{4},\;p\neq 5,\\
			\frac{-p^{2}-102p+163}{24(6p+1)} & \mbox{if } p\equiv 3\pmod{4}.
		\end{array}
	\right.\\
\end{align*}
Hence
\begin{align*}
    &n_{L}(\Sigma(2,3,6p+1),\wh{\rho}_{p},g,\nabla^{\infty}) \\
    &=0+\left\{\renewcommand{\arraystretch}{1.2}
		\begin{array}{ll}
            \frac{1}{18} & \mbox{if } p\equiv 1\pmod{6}\\
			-\frac{1}{18} & \mbox{if } p\equiv 5\pmod{6}
		\end{array}
	\right.
    +\left\{\renewcommand{\arraystretch}{1.2}
		\begin{array}{ll}
            \frac{1}{2} & \mbox{if } p=5\\
            \frac{-p^{2}+114p+199}{24(6p+1)} & \mbox{if } p\equiv 1\pmod{4},\;p\neq 5\\
			\frac{-p^{2}-102p+163}{24(6p+1)} & \mbox{if } p\equiv 3\pmod{4}
		\end{array}
	\right.\\
    &\qquad+\tfrac{1}{2p}\Big(0+\tfrac{1}{18}+\tfrac{-p^{2}+6p}{12p+2}\Big) \\
    &\qquad+\tfrac{-p^{2}-108p+199}{144p(6p+1)}+\left\{\renewcommand{\arraystretch}{1.2}
		\begin{array}{ll}
            \frac{1}{5} & \mbox{if } p=5\\
			-\frac{3}{2p} & \mbox{if } p\geq 7
		\end{array}
	\right.+\left\{\renewcommand{\arraystretch}{1.2}
		\begin{array}{ll}
            \pm\frac{7}{24} & \mbox{if } p\equiv \pm 1\pmod{12}\\
			\mp\frac{1}{24} & \mbox{if } p\equiv \pm 5\pmod{12}
		\end{array}
	\right. \\
    &=\twopartdef{\frac{-p^{2}\pm 158p-13}{144p}}{p\equiv \pm 1\pmod{12},}{\frac{-p^{2}\pm 94p-13}{144p}}{p\equiv \pm 5\pmod{12}.}
\end{align*}
\end{proof}

In order to prove Proposition \ref{prop:comparing_n_and_S}, we will need to introduce yet another sum. For $p,q,r\in\ZZ$ such that $(q,p)=(r,p)=1$ and $\varepsilon\in\{\pm 1\}$, consider the expression
\begin{equation}
    S(q,r,p;\varepsilon):=\tfrac{1}{p}\sum_{j=1}^{|p|-1}\varepsilon^{j}\csc(\tfrac{jq\pi}{p})\csc(\tfrac{jr\pi}{p}),
\end{equation}
which we refer to as the \emph{Dedekind cosecant sum}. The specialization $S(q,1,p;\varepsilon)$ appears in \cite{Fuk03}, and constitutes ``one-half'' of Fukumoto-Furuta-Ue's sum $\sigma(q,p;\varepsilon)$ \cite{FFU01}, in the sense that
\[\sigma(q,p;\varepsilon)=4s(q,p)+2S(q,1,p;\varepsilon).\]
For our purposes, we shall be interested in the Dedekind cosecant sum as (\ref{eq:S_invariant}) implies that
\begin{equation}
\label{eq:S_in_terms_of_S}
    \SSS_{0}(X,\tau)=\tfrac{1}{p}\sigma(X)+\sum_{(a,b)\in\D(X,\tau)}2S(a,b,p;(-1)^{a+b})
\end{equation}
for any locally linear pseudofree $\ZZ_{p}$-action $\tau:X\to X$ with corresponding fixed-point data $\D(X,\tau)$. The following properties will be helpful in computing these sums:

\begin{proposition}
\label{prop:cosecant_sums}
The sum $S(q,r,p,\varepsilon)$ satisfies the following properties:
\begin{enumerate}
    \item $S(-q,r,p,\varepsilon)=S(q,-r,p,\varepsilon)=S(q,r,-p,\varepsilon)=-S(q,r,p,\varepsilon)$.
    \item $S(q+cp,r,p,\varepsilon)=S(q,r+cp,p,\varepsilon)=S(q,r,p,(-1)^{c}\varepsilon)$.
    \item $S(q,r,p;\varepsilon)=S(r'q,1,p;(-1)^{(q+r+p)(r-1)}\varepsilon^{r+r'+1})$, where $r'\in\ZZ$ is such that $r'r\equiv 1+(r-1)p\pmod{2p}$.
    \item For $p$ even we have that
    \[S(1,1,p;-1)=\frac{-p^{2}-2}{6p}.\]
    \item For all $p,q\in\ZZ$ such that $p\not\equiv q\pmod{2}$ we have that
    \[S(q,1,p;-1)+S(p,1,q;-1)=-\frac{p}{6q}-\frac{q}{6p}-\frac{1}{6pq}.\]
    \item For $p,q\in\ZZ$ such that $p\not\equiv q\pmod{2}$ and $|p|>|q|$, let $q_{0},q_{1},\dots,q_{n}$ be a sequence of integers defined by $q_{0}=p$, $q_{1}=q$, and $q_{j}=\alpha_{j-1}q_{j-1}-q_{j-2}$ for all $2\le j\le n$ for some sequence of non-zero \emph{even} integers $\alpha_{1},\dots,\alpha_{n-1}$ such that $|q_{0}|>|q_{1}|>\cdots>|q_{n-1}|>|q_{n}|=1$. Furthermore let $s_{0},\dots,s_{n-1}$ be the sequence of integers defined by $s_{0}=0$, $s_{1}=1$, and $s_{j}=\alpha_{j-1}s_{j-1}-s_{j-2}$ for all $2\le j\le n-1$. Then:
    \[S(q,1,p;-1)=-\frac{q_{n}(q_{n-1}^{2}+2)+q_{n-2}}{6q_{n-1}}-\frac{q_{1}}{6q_{0}}-\frac{s_{n-1}}{6q_{0}q_{n-1}}-\sum_{j=1}^{n-2}\frac{\alpha_{j}}{6}.\]
\end{enumerate}
\end{proposition}

\begin{proof}
Properties (1) and (2) follow from the identities $\csc(-x)=\csc(x+\pi)=-\csc(x)$ and Property (3) follows from a straight-forward calculation via re-arranging the sum. Property (5) follows from (\cite{Fuk03}, Cor. 1.3 (4)), from which we deduce Property (4) as a special case, and Property (6) follows from Properties (4) and (5) by an argument analogous to the proof of (\cite{FFU01}, Prop. 9).
\end{proof}

We are now ready to prove Proposition \ref{prop:comparing_n_and_S}:

\begin{proof}[Proposition \ref{prop:comparing_n_and_S}]
For the first assertion, note that $\DDD(N(2pn),\tau_{p,n})=\{(2,3),(2,3),(-2,3)\}$, and hence by (\ref{eq:S_in_terms_of_S}) we have that
\begin{equation}
\label{eq:S_0_nucleus}
    \SSS_{0}(N(2pn),\tau_{p,n})=\tfrac{1}{p}\sigma(N(2pn))+4S(2,3,p;-1)+2S(-2,3,p;-1)=2S(2,3,p;-1).
\end{equation}
We can calculate $S(2,3,p;-1)$ via Proposition \ref{prop:cosecant_sums}. There are four cases depending on the value of $p$ modulo $12$ --- we will go through the calculation corresponding to the case $p\equiv 1\pmod{12}$.

First note that $S(2,3,p;-1)=S(\frac{-2p+2}{3},1,p;-1)$ via Property 3 of Proposition \ref{prop:cosecant_sums}. Taking $(p,q)=(p,\frac{-2p+2}{3})$, one can check that the corresponding sequences $\{q_{j}\}$, $\{\alpha_{j}\}$ and $\{s_{j}\}$ from Property (6) of Proposition \ref{prop:cosecant_sums} are given by:
\begin{align*}
    &(q_{0},q_{1},q_{2},q_{3},q_{4})=(p,\tfrac{-2p+2}{3},\tfrac{p-4}{3},2,1), &(\alpha_{1},\alpha_{2},\alpha_{3},\alpha_{4})=(-2,-2,\tfrac{p-1}{6},2), \\
    &(s_{0},s_{1},s_{2},s_{3})=(0,1,-2,3). &
\end{align*}
Hence
\begin{align*}
    S(2,3,p;-1)&=S(\tfrac{-2p+2}{3},1,p;-1) \\
    &=-\frac{(1)(2^{2}+2)+\frac{p-4}{3}}{6(2)}-\frac{\frac{-2p+2}{3}}{6(p)}-\frac{3}{6(p)(2)}-\frac{-2}{6}-\frac{-2}{6} \\
    &=\frac{-p^{2}+14p-13}{36p}\qquad\text{ if }p\equiv 1\pmod{12}.
\end{align*}
By similar calculations for $p\equiv 5,7,11\pmod{12}$, we obtain:
\[S(2,3,p;-1)=\twopartdef{\frac{-p^{2}\pm14p-13}{36p}}{p\equiv \pm 1\pmod{12},}{\frac{-p^{2}\mp 50p-13}{36p}}{p\equiv \pm 5\pmod{12},}\]
and therefore
\begin{align*}
    \tfrac{1}{8}\SSS_{0}(N(2pn),\tau_{p,n})&=-n_{p/2}(\Sigma(2,3,12pn-1),\rho_{p},g,\nabla^{\infty}) \\
    &=n_{p/2}(-\Sigma(2,3,12pn-1),\rho_{p},g,\nabla^{\infty}).
\end{align*}
by (\ref{eq:S_0_nucleus}) and Proposition \ref{prop:correction_terms}.

Next we prove the second assertion of Proposition \ref{prop:comparing_n_and_S}. By the construction of $\tau_{p,n}':P(2pn-p+1)\to P(2pn-p+1)$ in the proof of Theorem \ref{theorem:intro_locally_linear_extension}, we have for all $n\geq 1$:
\begin{align*}
    &\D(P(10n-4),\tau_{5,n}')=\{(1,1),(1,1),(1,1),(1,1),(1,2),(-1,2), \\
        &\qquad\qquad\qquad\qquad\qquad\;(-2,3),(-2,3),(-2,3),(-2,3),(-2,3)\}, \\
    &\D(P(14n-6),\tau'_{7,n})=\{(1,1),(1,1),(1,2),(-1,2),(-2,3),(-2,3), \\
        &\qquad\qquad\qquad\qquad\qquad\;(-2,3),(-2,3),(-2,3),(3,3),(3,3)\}, \\
    &\D(P(22n-10),\tau'_{11,n})=\{(1,1),(1,2),(-1,2),(-2,3),(-2,3),(-3,4), \\
        &\qquad\qquad\qquad\qquad\qquad\;\;\;\;(-4,5),(-5,6),(-6,7),(-7,8),(-8,9)\}, \\
    &\D(P(2pn-p+1),\tau'_{p,n})=\{(1,2),(-1,10),(-2,3),(-2,3),(-3,4),(-4,5), \\
        &\qquad\qquad\qquad\qquad\qquad\qquad(-5,6),(-6,7),(-7,8),(-8,9),(-9,10)\}\qquad\text{for all }p\geq 13.
\end{align*}
From Proposition \ref{prop:correction_terms} we have that
\begin{align*}
    &n_{5/2}(-\Sigma(2,3,60n-29),\wh{\rho}_{5},g,\nabla^{\infty})=-n_{5/2}(\Sigma(2,3,60n-29),\wh{\rho}_{5},g,\nabla^{\infty})=-\tfrac{3}{5}, \\
    &n_{7/2}(-\Sigma(2,3,84n-41),\wh{\rho}_{7},g,\nabla^{\infty})=-n_{7/2}(\Sigma(2,3,84n-41),\wh{\rho}_{7},g,\nabla^{\infty})=\tfrac{5}{7}, \\
    &n_{11/2}(-\Sigma(2,3,132n-65),\wh{\rho}_{11},g,\nabla^{\infty})=-n_{11/2}(\Sigma(2,3,132n-65),\wh{\rho}_{11},g,\nabla^{\infty})=\tfrac{13}{11},
\end{align*}
and a direct calculation using (\ref{eq:S_in_terms_of_S}) shows that 
\begin{align*}
    &\tfrac{1}{8}\SSS_{0}(P(10n-4),\tau'_{5,n})=\tfrac{17}{5}=n_{5/2}(-\Sigma(2,3,60n-29),\wh{\rho}_{5},g,\nabla^{\infty})+4, \\
    &\tfrac{1}{8}\SSS_{0}(P(14n-6),\tau'_{7,n})=\tfrac{5}{7}=n_{7/2}(-\Sigma(2,3,84n-41),\wh{\rho}_{7},g,\nabla^{\infty}), \\
    &\tfrac{1}{8}\SSS_{0}(P(22n-10),\tau'_{11,n})=\tfrac{13}{11}=n_{11/2}(-\Sigma(2,3,132n-65),\wh{\rho}_{11},g,\nabla^{\infty}),
\end{align*}
as claimed. We will now focus on the case $p\geq 13$, for which we have
\begin{align*}
    \SSS_{0}(P(2pn-p+1),\tau'_{p,n})&=-\tfrac{8}{p}+2S(1,2,p;-1)+2S(-1,10,p;-1)+4S(-2,3,p;-1) \\
    &\qquad+2S(-3,4,p;-1)+2S(-4,5,p;-1)+2S(-5,6,p,-1) \\
    &\qquad+2S(-6,7,p;-1)+2S(-7,8,p;-1)+2S(-8,9,p;-1) \\
    &\qquad+2S(-9,10,p;-1).
\end{align*}
By similar calculations as in the case of $S(2,3,p;-1)$, we can conclude that
\begin{align*}
    &S(1,2,p;-1)=\tfrac{-p^{2}\pm 6p-5}{12p}\text{ if }p\equiv \pm 1\pmod{4}, \\
    &S(-1,10,p;-1)=\left\{
		\begin{array}{ll}
            \frac{p^{2}\mp 102p+101}{60p} & \mbox{if } p\equiv \pm 1\pmod{20},\\
			\frac{p^{2}\mp 90p+101}{60p} & \mbox{if } p\equiv \pm 3,\pm 7\pmod{20}, \\
            \frac{p^{2}\mp 198p+101}{60p} & \mbox{if } p\equiv \pm 9\pmod{20},
		\end{array}
	\right. \\
    &S(-2,3,p;-1)=\left\{
		\begin{array}{ll}
            \frac{p^{2}\mp 14p+13}{36p} & \mbox{if } p\equiv \pm 1\pmod{12},\\
			\frac{p^{2}\pm 50p+13}{36p} & \mbox{if } p\equiv \pm 5\pmod{12},
		\end{array}
	\right. \\
    &S(-3,4,p;-1)=\left\{
		\begin{array}{ll}
            \frac{p^{2}\mp 26p+25}{72p} & \mbox{if } p\equiv \pm 1\pmod{24},\\
			\frac{p^{2}\mp 10p+25}{72p} & \mbox{if } p\equiv \pm 5 \pmod{24}, \\
            \frac{p^{2}\pm 154p+25}{72p} & \mbox{if } p\equiv \pm 7 \pmod{24}, \\
			\frac{p^{2}\mp 118p+25}{72p} & \mbox{if } p\equiv \pm 11 \pmod{24},
		\end{array}
	\right. \\
    &S(-4,5,p;-1)=\left\{
		\begin{array}{ll}
            \frac{p^{2}\mp 42p+41}{120p} & \mbox{if } p\equiv \pm 1\pmod{40},\\
		  \frac{p^{2}\pm 90p+41}{120p} & \mbox{if } p\equiv \pm 3, \mp 13 \pmod{40},\\
            \frac{p^{2}\mp 150p+41}{120p} & \mbox{if } p\equiv \pm 7, \mp 17\pmod{40},\\
            \frac{p^{2}\pm 342p+41}{120p} & \mbox{if } p\equiv \pm 9\pmod{40},\\
            \frac{p^{2}\mp 102p+41}{120p} & \mbox{if } p\equiv \pm 11\pmod{40},\\
            \frac{p^{2}\pm 282p+41}{120p} & \mbox{if } p\equiv \pm 19\pmod{40},\\
		\end{array}
	\right. \\
    &S(-5,6,p;-1)=\left\{
		\begin{array}{ll}
            \frac{p^{2}\mp 62p+61}{180p} & \mbox{if } p\equiv \pm 1\pmod{60},\\
		  \frac{p^{2}\pm 190p+61}{180p} & \mbox{if } p\equiv \pm 7, \mp 17 \pmod{60},\\
            \frac{p^{2}\pm 638p+61}{180p} & \mbox{if } p\equiv \pm 11\pmod{60},\\
            \frac{p^{2}\mp 350p+61}{180p} & \mbox{if } p\equiv \pm 13, \mp 23 \pmod{60},\\
            \frac{p^{2}\mp 98p+61}{180p} & \mbox{if } p\equiv \pm 19\pmod{60},\\
            \frac{p^{2}\mp 478p+61}{180p} & \mbox{if } p\equiv \pm 29\pmod{60},\\
		\end{array}
	\right. \\
    &S(-6,7,p;-1)=\left\{
		\begin{array}{ll}
            \frac{p^{2}\mp 86p+85}{252p} & \mbox{if } p\equiv \pm 1\pmod{84},\\
            \frac{p^{2}\mp 22p+85}{252p} & \mbox{if } p\equiv \pm 5, \pm 17 \pmod{84},\\
            \frac{p^{2}\mp 202p+85}{252p} & \mbox{if } p\equiv \pm 11, \pm 23 \pmod{84},\\
            \frac{p^{2}\pm 1066p+85}{252p} & \mbox{if } p\equiv \pm 13\pmod{84},\\
            \frac{p^{2}\mp 554p+85}{252p} & \mbox{if } p\equiv \pm 19, \pm 31 \pmod{84},\\
            \frac{p^{2}\pm 778p+85}{252p} & \mbox{if } p\equiv \pm 25, \pm 37 \pmod{84},\\
            \frac{p^{2}\mp 310p+85}{252p} & \mbox{if } p\equiv \pm 29 \pmod{84},\\
            \frac{p^{2}\pm 842p+85}{252p} & \mbox{if } p\equiv \pm 41 \pmod{84},\\
		\end{array}
	\right. \\
\end{align*}
\begin{align*}
    &S(-7,8,p;-1)=\left\{
		\begin{array}{ll}
            \frac{p^{2}\mp 114p+113}{336p} & \mbox{if } p\equiv \pm 1\pmod{112},\\
            \frac{p^{2}\pm 258p+113}{336p} & \mbox{if } p\equiv \pm 3, \pm 19, \mp 37, \mp 53 \pmod{112},\\
            \frac{p^{2}\pm 510p+113}{336p} & \mbox{if } p\equiv \pm 5, \mp 11, \pm 45, \mp 51 \pmod{112},\\
            \frac{p^{2}\pm 78p+113}{336p} & \mbox{if } p\equiv \pm 9, \pm 25 \pmod{112},\\
            \frac{p^{2}\mp 642p+113}{336p} & \mbox{if } p\equiv \pm 13, \mp 43 \pmod{112},\\
            \frac{p^{2}\pm 1650p+113}{336p} & \mbox{if } p\equiv \pm 15 \pmod{112},\\
            \frac{p^{2}\mp 498p+113}{336p} & \mbox{if } p\equiv \pm 17, \pm 33 \pmod{112},\\
            \frac{p^{2}\mp 846p+113}{336p} & \mbox{if } p\equiv \pm 23, \pm 39 \pmod{112},\\
            \frac{p^{2}\mp 894p+113}{336p} & \mbox{if } p\equiv \pm 27, \mp 29 \pmod{112},\\
            \frac{p^{2}\pm 1266p+113}{336p} & \mbox{if } p\equiv \pm 31, \pm 47 \pmod{112},\\
            \frac{p^{2}\mp 306p+113}{336p} & \mbox{if } p\equiv \pm 41\pmod{112},\\
            \frac{p^{2}\mp 1230p+113}{336p} & \mbox{if } p\equiv \pm 55\pmod{112},\\
		\end{array}
	\right. \\
    &S(-8,9,p;-1)=\left\{
		\begin{array}{ll}
            \frac{p^{2}\mp 146p+145}{432p} & \mbox{if } p\equiv \pm 1\pmod{144},\\
            \frac{p^{2}\mp 34p+145}{432p} & \mbox{if } p\equiv \pm 5, \pm 29, \mp 43, \mp 67 \pmod{144},\\
            \frac{p^{2}\pm 466p+145}{432p} & \mbox{if } p\equiv \pm 7, \mp 41 \pmod{144},\\
            \frac{p^{2}\pm 290p+145}{432p} & \mbox{if } p\equiv \pm 11, \mp 13, \pm 59, \mp 61 \pmod{144},\\
            \frac{p^{2}\pm 2414p+145}{432p} & \mbox{if } p\equiv \pm 17\pmod{144},\\
            \frac{p^{2}\mp 1118p+145}{432p} & \mbox{if } p\equiv \pm 19, \mp 53 \pmod{144},\\
            \frac{p^{2}\pm 722p+145}{432p} & \mbox{if } p\equiv \pm 23, \mp 25 \pmod{144},\\
            \frac{p^{2}\mp 1262p+145}{432p} & \mbox{if } p\equiv \pm 31, \mp 65 \pmod{144},\\
            \frac{p^{2}\pm 1442p+145}{432p} & \mbox{if } p\equiv \pm 35, \mp 37 \pmod{144},\\
            \frac{p^{2}\mp 1006p+145}{432p} & \mbox{if } p\equiv \pm 47, \mp 49 \pmod{144},\\
            \frac{p^{2}\mp 686p+145}{432p} & \mbox{if } p\equiv \pm 55 \pmod{144},\\
            \frac{p^{2}\pm 1874p+145}{432p} & \mbox{if } p\equiv \pm 71 \pmod{144},\\
		\end{array}
	\right. \\
\end{align*}
\begin{align*}
    &S(-9,10,p;-1)=\left\{
		\begin{array}{ll}
            \frac{p^{2}\mp 182p+181}{540p} & \mbox{if } p\equiv \pm 1 \pmod{180},\\
            \frac{p^{2}\mp 650p+181}{540p} & \mbox{if } p\equiv \pm 7, \pm 43, \pm 67, \mp 77 \pmod{180},\\
            \frac{p^{2}\pm 758p+181}{540p} & \mbox{if } p\equiv \pm 11, \mp 49 \pmod{180},\\
            \frac{p^{2}\pm 970p+181}{540p} & \mbox{if } p\equiv \pm 13, \mp 23, \mp 47, \mp 83 \pmod{180},\\
            \frac{p^{2}\mp 790p+181}{540p} & \mbox{if } p\equiv \pm 17, \pm 53 \pmod{180},\\
            \frac{p^{2}\pm 3382p+181}{540p} & \mbox{if } p\equiv \pm 19 \pmod{180},\\
            \frac{p^{2}\mp 1078p+181}{540p} & \mbox{if } p\equiv \pm 29, \mp 31 \pmod{180},\\
            \frac{p^{2}\pm 2410p+181}{540p} & \mbox{if } p\equiv \pm 37, \pm 73 \pmod{180},\\
            \frac{p^{2}\mp 1942p+181}{540p} & \mbox{if } p\equiv \pm 41, \mp 79 \pmod{180},\\
            \frac{p^{2}\pm 1622p+181}{540p} & \mbox{if } p\equiv \pm 59, \mp 61 \pmod{180},\\
            \frac{p^{2}\mp 682p+181}{540p} & \mbox{if } p\equiv \pm 71 \pmod{180},\\
            \frac{p^{2}\mp 2518p+181}{540p} & \mbox{if } p\equiv \pm 89 \pmod{180}.\\
		\end{array}
	\right. \\
\end{align*}
After a (long) calculation we obtain the following expression for $\SSS_{0}(P(2pn-p+1),\tau'_{p,n})$ (assuming $p\geq 13)$:
\[\SSS_{0}(P(2pn-p+1),\tau'_{p,n})=\left\{
	\begin{array}{ll}
        \frac{p^{2}\mp 158p+13}{18p} & \mbox{if } p\equiv \pm 1,\mp 11\pmod{60},\\
        \frac{p^{2}\mp 194p+13}{18p} & \mbox{if } p\equiv \pm 7,\mp 17\pmod{60},\\
        \frac{p^{2}\pm 130p+13}{18p} & \mbox{if } p\equiv \pm 13,\mp 23\pmod{60},\\
        \frac{p^{2}\pm 94p+13}{18p} & \mbox{if } p\equiv \pm 19,\mp 29\pmod{60}.
	\end{array}
    \right.\]
It then follows that
\begin{align*}
    &\tfrac{1}{8}\SSS_{0}(P(2pn-p+1),\tau'_{p,n}) \\
    &\qquad=\twopartdef{\frac{p^{2}\mp 158p+13}{144p}}{p\equiv \pm 1\pmod{12}}{\frac{p^{2}\mp 94p+13}{144p}}{p\equiv \pm 5\pmod{12}}
    +\left\{
		\begin{array}{ll}
		0 & \mbox{if } p\equiv \pm 1,\pm 9\pmod{20} \\
        \mp 2 & \mbox{if } p\equiv \pm 3,\pm 7\pmod{20}
		\end{array}
	\right. \\
    &\qquad=-n_{p/2}(\Sigma(2,3,12pn-6p+1),\wh{\rho}_{p},g,\nabla^{\infty})+\left\{
		\begin{array}{ll}
		0 & \mbox{if } p\equiv \pm 1,\pm 9\pmod{20} \\
        \mp 2 & \mbox{if } p\equiv \pm 3, \pm 7\pmod{20}
		\end{array}
	\right.
\end{align*}
for $p\geq 13$. The result then follows from combining this with our previous calculations for $p=5,7,11$.
\end{proof}

\bibliographystyle{alpha}
\bibliography{refs}

\end{document}